%% file: main.tex
\documentclass[12pt]{article}
\input{packagesandmore}
\title{Efficient counting of permutation patterns via double posets}

\author{Joscha Diehl, Emanuele Verri}
\begin{document}

\maketitle

\begin{abstract}
Corner trees, introduced in ``Even-Zohar and Leng, 2021, Proceedings of the 2021 ACM-SIAM Symposium on Discrete Algorithms'', allow for the efficient counting
of certain permutation patterns.

Here we identify corner trees
as a subset of 
finite (strict) double posets,
which we term \emph{twin-tree double posets}.
They are contained in both \emph{twin double posets}
and \emph{tree double posets}, giving candidate sets for generalizations of corner tree countings.
We provide the generalization of an algorithm proposed by Even-Zohar/Leng to a class of tree double posets,
thereby enlarging the space of permutations
that can be counted in $\tilde{\mathcal O}(n^{5/3})$.
\end{abstract}

\tableofcontents

\section*{Introduction}
\addcontentsline{toc}{section}{Introduction}

\emph{Corner trees}, introduced by \cite{even2021counting}, provide a powerful framework for counting occurrences of permutation patterns in near-linear time $\tilde{\mathcal O}(n)$\footnote{It is linear up to a polylogarithmic factor, i.e. $\tilde{\mathcal O}(f(n)) := \mathcal{O}(f(n)\,\log^{c}(n))$ where $c$ is some constant.}. A corner tree always counts a fixed linear combination of pattern counts. Notably, using corner trees with at most three vertices, one can compute the full $3$-profile of a permutation $\Pi \in \permutations$—that is, the tuple
\begin{align*}
    (\#\perm{123}(\Pi),\#\perm{132}(\Pi),\#\perm{213}(\Pi),\#\perm{231}(\Pi),\#\perm{312}(\Pi),\#\perm{321}(\Pi))
\end{align*}
in $\tilde{\mathcal O}(n)$ time, offering a significant improvement over the naive $\mathcal{O}(n^{3})$ algorithm.

At level four, the subspace of 4-patterns counted by corner trees (with at most four vertices) is only $23$-dimensional, missing one direction. To tackle this problem, \cite{even2021counting} proposes a specific algorithm that counts the pattern $\perm{3214}$ in $\tilde{\mathcal{O}}(n^{5/3})$ time. Since the pattern $\perm{3214}$ does not lie in the span of corner trees, then the full $4$-profile can also be computed in $\tilde{\mathcal{O}}(n^{5/3})$ time. 

At higher levels, corner trees fail to span even more directions. Furthermore, even if corner trees always yielded linearly independent vectors, the number of isomorphism classes with $n$ vertices grows asymptotically slower than $n!$. See the entry \cite{OEIS_A052763} of the on-line encyclopedia of the integer sequences.

In this work, we show that counting occurrences of permutation patterns can be naturally framed as counting occurrences of strict double posets. Namely, we will show that double poset occurrences on permutations can always be translated to a linear combination of pattern counts. In particular, we show that both corner trees and permutations can be encoded as certain families of double posets. As a concrete application of our method, we generalize the algorithm from \cite{even2021counting} that counts $\perm{3214}$ to a certain family of double posets. This family allows us to add twelve new directions at level $5$ computable in $\tilde{\mathcal{O}}(n^{5/3})$ time.

\section*{Overview}

This work is structured as follows.
\begin{enumerate}
\item In \Cref{section:permutation_corners} we introduce corner trees and their occurrences on permutations.
    \item In \Cref{section:double_posets}, we show how derooting corner trees yields a certain class of (strict) double poset, that we name \textit{twin tree double poset}. We show how permutations can be naturally encoded as pairs of strict linear orders.  This leads to corner tree occurrences being seen as strict double order preserving maps.
    \item In \Cref{section:count_perm_usingDP}, we show how all double posets count permutation patterns. Counting occurrences of corner trees is, therefore, just an instance of a more general phenomenon.  
      \item In \Cref{section:gen_algo}, we introduce a family of tree double posets, whose occurrences on permutations, can be counted in $\tilde{\mathcal{O}}(n^{5/3})$ time. This generalizes an algorithm that counts the occurrences of $\perm{3214}$ introduced in \cite{even2021counting}. We show how the elements of this family at level five allow us to count twelve more directions to the space of permutations counted by corner trees with at most five nodes.
      \item In \Cref{section:suppl}, we characterize regular monomorphisms and epimorphisms in the category of strict double posets \DPoset. We also show that the category \DPoset\, admits a (\Epi,\RegMono)-factorization. The latter factorization allows us to define the linear endomorphism that translates occurrences of double posets on permutations into linear pattern counts.
\end{enumerate}

\subsection*{Contributions}

The contributions of this work can be summarized as follows.

\begin{enumerate}
    \item We show that corner trees and permutations can be represented as specific kinds of strict double posets. In particular, unrooted corner trees are in one-to-one correspondence with \textit{twin tree double posets}. These are double posets where the two underlying Hasse diagrams, when viewed as labeled undirected graphs, are identical (hence the term \quot{twin}) and also connected and acyclic (hence the term \quot{tree}). On the other hand, permutations are encoded as pairs of strict linear orders. This allows us to interpret corner tree occurrences as double-order preserving maps.
    \item We show that the category of strict double posets \DPoset\, admits an (\Epi,\RegMono)-factorization. This allows us to define a linear endomorphism that translates occurrences of double posets on permutations to linear combinations of patterns. This construction resembles certain maps in graph theory, originally introduced by Lovàsz (\cite{lovasz2012large}), which were later interpreted as linear endomorphisms in \cite{Caudillo2025}. 
    
    \item Within this framework, we introduce a new family of tree double posets, denoted $\ArboNE$, whose occurrences in permutations can be counted in $\tilde{\mathcal{O}}(n^{5/3})$ time. To achieve this, we generalize an algorithm of \cite{even2021counting} that counts occurrences of the pattern $\perm{3214}$. In addition, we show that the $\ArboNE$ family allows us to count twelve new directions at level five that are not spanned by corner trees.
\end{enumerate}

\section{Corner trees: counting linear combinations of permutation patterns}
\label{section:permutation_corners}
In the original work by \cite{even2021counting}, corner trees are finite rooted trees whose vertices, except for the root, are labeled with the four directions $\NE,\NW,\SE$, and $\SW$.
Here we use an equivalent formulation and
instead label the edges with the four directions.%
\footnote{This, for example, makes the formulation of the algorithms
more transparent to us (see \Cref{algorithm_occurrences}).}

\begin{definition}
 \label{definition:ct}   
 A \DEF{corner tree}, for us, is a rooted tree $\CT:=(\vertexset(\CT),\edgeset(\CT))$ equipped with a labeling of the edges by four cardinal directions, i.e. a map
%\begin{align*}
$\edgelabel: \edgeset(\CT) \to \{\NE,\NW,\SE,\SW\}.$
% \end{align*}
\end{definition}

Here we recall the definition presented in \cite{even2021counting} which motivates the labels assigned to the edges.

\begin{definition}
\label{definition:ct_occ}
  An \DEF{occurrence} of a corner tree $\CT$ in a permutation of size $n$, $\Pi \in \permutations(n)$, is  a mapping $f:\vertexset(\CT) \to [n]$ such that for all edges $e \in \edgeset(\CT)$,
  given that $e=(v,v')$ (where $v'$ is the child of $v$),
  the label of $e$ determines the allowed order in the image of $f$ as follows
\begin{center}
  \begin{tabular}{c|| c c c c} 
    \diagbox[innerwidth=3cm]{\textnormal{condition}}{$\mathfrak e(e)$} & $\NE$ & $\NW$ & $\SE$ & $\SW$\\
    \hline
    $f(v') < f(v)$ & $\times$ & $\checkmark$ & $\times$ & $\checkmark$\\
    $f(v') > f(v)$ & $\checkmark$ & $\times$ & $\checkmark$ & $\times$\\
    $\Pi(f(v')) < \Pi(f(v))$ & $\times$ & $\times$ & $\checkmark$ & $\checkmark$\\
    $\Pi(f(v')) > \Pi(f(v))$ & $\checkmark$ & $\checkmark$ & $\times$ & $\times$\\
  \end{tabular}
\end{center}
See \Cref{example:permutation}.
\end{definition}
 Corner trees count permutation patterns.
\begin{definition}
\label{definition:perm_linear_func}
Consider the free $\Q$-vector space on permutations, $\Q[\permutations]:=\bigoplus_{n}\Q[\permutations(n)]$, and fix a ``large'' permutation $\Pi \in \permutations(n)$. Define on basis elements $\sigma \in \permutations$ the linear functional
\begin{align*}
    \langle \PC(\Pi),\sigma \rangle := \#\sigma(\Pi):= |\{A \subseteq [n]\mid\textup{std}(\Pi|_{A})=\sigma\}|
\end{align*}
where $[n]\supseteq A=\{a_{1},\dots ,a_{\ell}\}$ with $a_1<\cdots <a_{\ell}$, $\Pi|_{A}:=[\Pi(a_{1}),...,\Pi(a_{\ell})]$, $\textup{std}(\Pi|_{A}):=[f(a_{1}),\dots,f(a_{m})]$ where $f: A \to |A|$ is the unique strict order preserving bijection with the usual orders. As an example, we have $\textup{std}([\mathtt{1\,5\,3\,2\,4}]|_{\{2,3,5\}}) = \textup{std}([\mathtt{5\,3\,4}])=[\mathtt{3\,1\,2}]$. Then $\langle \PC(\Pi),\sigma \rangle$ corresponds to the number of times that $\sigma$ arises as a \DEF{permutation pattern} in $\Pi$. \end{definition}

Occurrences of corner trees count linear combinations of permutation patterns, as the following example hints at, and as will be shown in \Cref{proposition:countpermasperm}. See also \Cref{remark:ct_lc}.
\begin{example}
If we consider occurrences of the corner tree
\scalebox{0.4}{
    \begin{forest}
      for tree={
        circle,
        draw,
        minimum size=0.5cm,
        edge={-},
        s sep=20mm,
        l sep=15mm,
      },
      [
      [, edge label={node[midway,above,sloped]{\NE}}
          ]
      ]
    \end{forest}
  } in a permutation $\Lambda$, these are exactly the
occurrences of the permutation $\perm{12}$ in $\Lambda$,
i.e. their number is equal to $\langle \PC(\Lambda),[\mathtt{1\,2}] \rangle$. While for the corner tree
\scalebox{0.4}{
    \begin{forest}
      for tree={
        circle,
        draw,
        minimum size=0.5cm,
        edge={-},
        s sep=20mm,
        l sep=15mm,
      },
      [[, edge label={node[midway,above,sloped]{\NE}}]
      [, edge label={node[midway,above,sloped]{\NE}}]
      ]
    \end{forest}
  }
the number of occurrences in $\Lambda$ is equal to 
\begin{align*}
    \langle \PC(\Lambda),[\mathtt{1\,2}] + 2\cdot [\mathtt{1\,2\,3}] + 2\cdot[\mathtt{1\,3\,2}] \rangle.
\end{align*}
We will see that for a fixed corner tree, counting its occurrences in a permutation always corresponds to a certain linear combination of permutation patterns. We will also see that these maps are order-preserving when we frame corner trees and permutations in the context of double posets.
\end{example}

\section{Corner trees and permutations as double posets}
\label{section:double_posets}
In this section, we first show that corner trees can be unrooted without loss of information. We name these graphs SN polytrees, \Cref{subsection:unrooting_corner_trees}. We then show that SN polytrees are pairs of Hasse diagrams of certain kinds of double posets. Since permutations can also be encoded as pairs of linear orders, we can encode occurrences of corner trees as double order-preserving maps.

\subsection{Unrooting corner trees}
\label{subsection:unrooting_corner_trees}
Let $\CornerTrees$ be the set of isomorphism classes of corner trees\footnote{The isomorphism between corner trees is the usual notion of isomorphism between rooted trees which respect the edge labels.}. Unrooted corner trees can be seen as polytrees\footnote{The terminology \textit{polytree} was first introduced in \cite{rebane1987recovery} to denote directed graphs whose underlying undirected graphs are trees.} endowed with a binary labeling on the edges. 

\begin{definition}
\label{defintion:SN_poly}
An \DEF{SN polytree} is a polytree whose edges are labeled either with \S 
  or \N. We denote with $\SNPolyTrees$, the set of isomorphism classes of SN polytrees\footnote{This is the isomorphism between directed graphs which respect SN labels.}.
\end{definition}

\begin{proposition}
\label{proposition:CT_SN_Poly}
Define the map $\CTtoSNpoly: \CornerTrees \to \SNPolyTrees$ as follows. Given a corner tree $\CT=(\vertexset(\CT),\edgeset(\CT))$, construct $\CTtoSNpoly(\CT)$ by 
\begin{itemize}
    \item preserving the vertex set $\vertexset(\CTtoSNpoly(\CT)):= \vertexset(\CT)$
    \item transforming the edges: each edge in $\edgeset(\CT)$ corresponds to a unique directed edge in $\CTtoSNpoly(\CT)$,
    \begin{itemize}
        \item directed westward
        \item labeled by the SN position of the target relative to the source
    \end{itemize}
\end{itemize}
Then, the map $\CTtoSNpoly$ is surjective, meaning that every SN polytree arises from some corner tree under this transformation.
\end{proposition}
\begin{remark}
\label{remark:rooting_SN_poly}
If $\Tsn \in \SNPolyTrees$ and $v \in \vertexset(\SNPolyTrees)$, then $\SNpolytoCT(\Tsn,v)$ denotes the corner tree obtained by picking $v$ as the root and relabeling the edges accordingly.
\end{remark}

\begin{figure}[H]
\centering
\begin{subfigure}[t]{0.4\textwidth}
\centering
\begin{align*}
\CTtoSNpoly \scalebox{2.5}{\text{$($}}\raisebox{-0.2cm}{\scalebox{0.26}{\begin{tikzpicture}[every node/.style={circle, draw, inner sep=4pt, minimum size=6mm}]
\node (X) at (0,0) {\phantom{\scalebox{2}{X}}};
\node (A) at (-2,-3) {\phantom{\scalebox{2}{X}}};
\node (B) at (2,-3) {\phantom{\scalebox{2}{X}}};
\tikzset{mid arrow/.style={
        postaction={decorate,decoration={
            markings,
            mark=at position .5 with {\arrow[scale=3]{stealth}}
        }}
    }}
        \draw[-] (X) -- (A) node[draw = none, pos=0.5, left, yshift=10pt] {\scalebox{2}{\NW}};
        \draw[-] (X) -- (B) node[draw = none, pos=0.5, right, yshift=10pt] {\scalebox{2}{\NW}};
\end{tikzpicture}}}\scalebox{2.5}{\text{$)$}} &=\phantom{\scalebox{2.5}{\text{$\{$}}}\raisebox{-0.2cm}{\scalebox{0.26}{\begin{tikzpicture}[every node/.style={circle, draw, inner sep=4pt, minimum size=6mm}]
\node (X) at (0,0) {\phantom{\scalebox{2}{X}}};
\node (A) at (-2,-3) {\phantom{\scalebox{2}{X}}};
\node (B) at (2,-3) {\phantom{\scalebox{2}{X}}};
\tikzset{mid arrow/.style={
        postaction={decorate,decoration={
            markings,
            mark=at position .5 with {\arrow[scale=3]{stealth}}
        }}
    }}
        \draw[-{Stealth[scale=3]}] (X) -- (A) node[draw = none, pos=0.5, left, yshift=10pt] {\scalebox{2}{\N}};
        \draw[-{Stealth[scale=3]}] (X) -- (B) node[draw = none, pos=0.5, right, yshift=10pt] {\scalebox{2}{\N}};
\end{tikzpicture}}}\phantom{\scalebox{2}{,}\quad\raisebox{-0.2cm}{\scalebox{0.26}{\begin{tikzpicture}[every node/.style={circle, draw, inner sep=4pt, minimum size=6mm}]
\node (X) at (0,0) {\phantom{\scalebox{2}{X}}};
\node (A) at (0,-2) {\phantom{\scalebox{2}{X}}};
\node (B) at (0,2) {\phantom{\scalebox{2}{X}}};
\tikzset{mid arrow/.style={
        postaction={decorate,decoration={
            markings,
            mark=at position .5 with {\arrow[scale=3]{stealth}}
        }}
    }}
        \draw[-] (X) -- (A) node[draw = none, pos=0.8, right, yshift=10pt] {\scalebox{2}{ \NW}};
        \draw[-] (X) -- (B) node[draw = none, pos=0.1, right, yshift=12pt] {\scalebox{2}{ \SE}};
\end{tikzpicture}}}
\scalebox{2.5}{\text{$\}$}}}\end{align*}
\caption*{Unrooted corner tree}
\end{subfigure}
\hfill
\begin{subfigure}[t]{0.4\textwidth}
\centering
\begin{align*}
\SNpolytoCT \scalebox{2.5}{\text{$($}}\raisebox{-0.2cm}{\scalebox{0.26}{\begin{tikzpicture}[every node/.style={circle, draw, inner sep=4pt, minimum size=6mm}]
\node (X) at (0,0) {\phantom{\scalebox{2}{X}}};
\node (A) at (-2,-3) {\phantom{\scalebox{2}{X}}};
\node (B) at (2,-3) {\phantom{\scalebox{2}{X}}};
\tikzset{mid arrow/.style={
        postaction={decorate,decoration={
            markings,
            mark=at position .5 with {\arrow[scale=3]{stealth}}
        }}
    }}
        \draw[-{Stealth[scale=3]}] (X) -- (A) node[draw = none, pos=0.5, left, yshift=10pt] {\scalebox{2}{\N}};
        \draw[-{Stealth[scale=3]}] (X) -- (B) node[draw = none, pos=0.5, right, yshift=10pt] {\scalebox{2}{\N}};
\end{tikzpicture}}}\scalebox{2.5}{\text{$)$}} &=\scalebox{2.5}{\text{$\{$}}\raisebox{-0.2cm}{\scalebox{0.26}{\begin{tikzpicture}[every node/.style={circle, draw, inner sep=4pt, minimum size=6mm}]
\node (X) at (0,0) {\phantom{\scalebox{2}{X}}};
\node (A) at (-2,-3) {\phantom{\scalebox{2}{X}}};
\node (B) at (2,-3) {\phantom{\scalebox{2}{X}}};
\tikzset{mid arrow/.style={
        postaction={decorate,decoration={
            markings,
            mark=at position .5 with {\arrow[scale=3]{stealth}}
        }}
    }}
        \draw[-] (X) -- (A) node[draw = none, pos=0.5, left, yshift=10pt] {\scalebox{2}{\NW}};
        \draw[-] (X) -- (B) node[draw = none, pos=0.5, right, yshift=10pt] {\scalebox{2}{\NW}};
\end{tikzpicture}}}\scalebox{2}{,}\quad\raisebox{-0.2cm}{\scalebox{0.26}{\begin{tikzpicture}[every node/.style={circle, draw, inner sep=4pt, minimum size=6mm}]
\node (X) at (0,0) {\phantom{\scalebox{2}{X}}};
\node (A) at (0,-2) {\phantom{\scalebox{2}{X}}};
\node (B) at (0,2) {\phantom{\scalebox{2}{X}}};
\tikzset{mid arrow/.style={
        postaction={decorate,decoration={
            markings,
            mark=at position .5 with {\arrow[scale=3]{stealth}}
        }}
    }}
        \draw[-] (X) -- (A) node[draw = none, pos=0.8, right, yshift=10pt] {\scalebox{2}{ \NW}};
        \draw[-] (X) -- (B) node[draw = none, pos=0.1, right, yshift=12pt] {\scalebox{2}{ \SE}};
\end{tikzpicture}}}
\scalebox{2.5}{\text{$\}$}}\end{align*}
\caption*{Corner trees from all rootings of an SN polytree}
\end{subfigure}

\vspace{1em}
\begin{subfigure}[t]{0.4\textwidth}
\centering
\begin{align*}
\SNpolytoCT \scalebox{2.5}{\text{$($}}\raisebox{-0.2cm}{\scalebox{0.26}{\begin{tikzpicture}[every node/.style={circle, draw, inner sep=4pt, minimum size=6mm}]
\node (X) at (0,0) {\scalebox{2}{X}};
\node (A) at (-2,-3) {\scalebox{2}{Y}};
\node (B) at (2,-3) {\scalebox{2}{Z}};
\tikzset{mid arrow/.style={
        postaction={decorate,decoration={
            markings,
            mark=at position .5 with {\arrow[scale=3]{stealth}}
        }}
    }}
        \draw[-{Stealth[scale=3]}] (X) -- (A) node[draw = none, pos=0.5, left, yshift=10pt] {\scalebox{2}{\N}};
        \draw[-{Stealth[scale=3]}] (X) -- (B) node[draw = none, pos=0.5, right, yshift=10pt] {\scalebox{2}{\N}};
\end{tikzpicture}}}\scalebox{2}{,} \text{X}\scalebox{2.5}{\text{$)$}} &=\phantom{\scalebox{2.5}{\text{$\{$}}}\raisebox{-0.2cm}{\scalebox{0.26}{\begin{tikzpicture}[every node/.style={circle, draw, inner sep=4pt, minimum size=6mm}]
\node (X) at (0,0) {\phantom{\scalebox{2}{X}}};
\node (A) at (-2,-3) {\phantom{\scalebox{2}{X}}};
\node (B) at (2,-3) {\phantom{\scalebox{2}{X}}};
\tikzset{mid arrow/.style={
        postaction={decorate,decoration={
            markings,
            mark=at position .5 with {\arrow[scale=3]{stealth}}
        }}
    }}
        \draw[-] (X) -- (A) node[draw = none, pos=0.5, left, yshift=10pt] {\scalebox{2}{\NW}};
        \draw[-] (X) -- (B) node[draw = none, pos=0.5, right, yshift=10pt] {\scalebox{2}{\NW}};
\end{tikzpicture}}}\phantom{\scalebox{2}{,}\quad\raisebox{-0.2cm}{\scalebox{0.26}{\begin{tikzpicture}[every node/.style={circle, draw, inner sep=4pt, minimum size=6mm}]
\node (X) at (0,0) {\phantom{\scalebox{2}{X}}};
\node (A) at (0,-2) {\phantom{\scalebox{2}{X}}};
\node (B) at (0,2) {\phantom{\scalebox{2}{X}}};
\tikzset{mid arrow/.style={
        postaction={decorate,decoration={
            markings,
            mark=at position .5 with {\arrow[scale=3]{stealth}}
        }}
    }}
        \draw[-] (X) -- (A) node[draw = none, pos=0.8, right, yshift=10pt] {\scalebox{2}{ \NW}};
        \draw[-] (X) -- (B) node[draw = none, pos=0.1, right, yshift=12pt] {\scalebox{2}{ \SE}};
\end{tikzpicture}}}
\scalebox{2.5}{\text{$\}$}}}\end{align*}
\caption*{Rooting in X}
\end{subfigure}
\hfill
\begin{subfigure}[t]{0.4\textwidth}
\centering
\begin{align*}
\SNpolytoCT \scalebox{2.5}{\text{$($}}\raisebox{-0.2cm}{\scalebox{0.26}{\begin{tikzpicture}[every node/.style={circle, draw, inner sep=4pt, minimum size=6mm}]
\node (X) at (0,0) {\scalebox{2}{X}};
\node (A) at (-2,-3) {\scalebox{2}{Y}};
\node (B) at (2,-3) {\scalebox{2}{Z}};
\tikzset{mid arrow/.style={
        postaction={decorate,decoration={
            markings,
            mark=at position .5 with {\arrow[scale=3]{stealth}}
        }}
    }}
        \draw[-{Stealth[scale=3]}] (X) -- (A) node[draw = none, pos=0.5, left, yshift=10pt] {\scalebox{2}{\N}};
        \draw[-{Stealth[scale=3]}] (X) -- (B) node[draw = none, pos=0.5, right, yshift=10pt] {\scalebox{2}{\N}};
\end{tikzpicture}}}\scalebox{2}{,} \text{Y}\scalebox{2.5}{\text{$)$}} &=\quad\raisebox{-0.2cm}{\scalebox{0.26}{\begin{tikzpicture}[every node/.style={circle, draw, inner sep=4pt, minimum size=6mm}]
\node (X) at (0,0) {\phantom{\scalebox{2}{X}}};
\node (A) at (0,-2) {\phantom{\scalebox{2}{X}}};
\node (B) at (0,2) {\phantom{\scalebox{2}{X}}};
\tikzset{mid arrow/.style={
        postaction={decorate,decoration={
            markings,
            mark=at position .5 with {\arrow[scale=3]{stealth}}
        }}
    }}
        \draw[-] (X) -- (A) node[draw = none, pos=0.8, right, yshift=10pt] {\scalebox{2}{ \NW}};
        \draw[-] (X) -- (B) node[draw = none, pos=0.1, right, yshift=12pt] {\scalebox{2}{ \SE}};
\end{tikzpicture}}}
\phantom{\scalebox{2.5}{\text{$\}$}}}\end{align*}
\caption*{Rooting in Y}
\end{subfigure}

\caption{Illustration of the conversions between SN polytrees and corner trees}
\end{figure}
\FloatBarrier

\subsection{Double posets}
\label{subsection:DP}
A (strict) \DEF{double poset} is a triple $\smalldp = \triplesmalldp$ where $A$ is a (finite) set, and $P_{A}$ and $ Q_{A}$ are strict posets on $A$. We refer to Section\;\ref{subsection:Poset} for a detailed description of the category of double posets.

Since we work with \textit{finite} double posets, we can always draw their respective \textbf{Hasse diagrams}, i.e., their respective underlying cover relation (see \Cref{definition:Hasse_diagram}),
which are in \emph{one-to-one correspondence} with the double poset. We denote the Hasse diagram of a poset $P$ with $H_{P}$.

\begin{remark}
 We work with \emph{strict} posets, where the morphisms correspond to \emph{strict} order-preserving maps, since this will lead to the connection with corner tree occurrences, as shown in \Cref{proposition:occ_as_DP}.
\end{remark}

We work with maps that are simultaneously respecting both the first and the second relation.
\begin{definition}
\label{definition:morphism_DP}
A \DEF{morphism of double posets} $\smalldp =\triplesmalldp$ and $\smalldpprime =\triplesmalldpprime$ is a map $f:A \to B$ such that 
\begin{align*}
    &\forall a_{1},a_{2} \in A: a_{1} <_{P_{A}} a_{2} \implies  f(a_{1}) <_{P_{B}} f(a_{2})\\
      &\forall a_{1},a_{2} \in A: a_{1} <_{Q_{A}} a_{2} \implies  f(a_{1}) <_{Q_{B}} f(a_{2})\\
\end{align*}
\end{definition}

We now define certain classes of double posets. Corner trees and permutations will arise as elements of these classes.

\begin{definition}[Twin double poset]
\label{definition:TwinDP}
  A double poset $\smalldp$ is a \textbf{twin double poset}
  if the underlying two Hasse diagrams are equal as (vertex-)labeled, undirected graphs.
  We denote a general twin double poset with $\smalltwin$.
\end{definition}

\begin{definition}[Tree double poset]
\label{definition:TreeDP}
  A double poset $\smalldp$ is a \textbf{tree double poset}
  if the underlying Hasse diagrams are both trees
  (in the graph-theoretic, undirected sense). See the terminology used in \cite{trotter1977dimension} at the beginning of Section 7.
  We denote a general tree double poset with $\smalltree$.
\end{definition}

\begin{definition}[Twin tree double posets]
\label{definition:TwinTreeDP}
  A double poset $\smalldp=(A,P_{A},Q_{A})$ is a \textbf{twin tree double poset}
  if it is both a twin double poset and a tree double poset.
  We denote a general twin tree double poset with $\smallct$. Spelled out, if $H_{P}$ and $H_{Q}$ are the Hasse diagrams of $P_{A}$ and $Q_{A}$, respectively,
  then as \emph{undirected graphs}
  \begin{itemize}
      \item they are (labeled) trees;
      \item they are equal as unrooted, labeled trees.
  \end{itemize}
\end{definition}

\begin{figure}[H]
    \centering
    \includegraphics[scale=0.4]{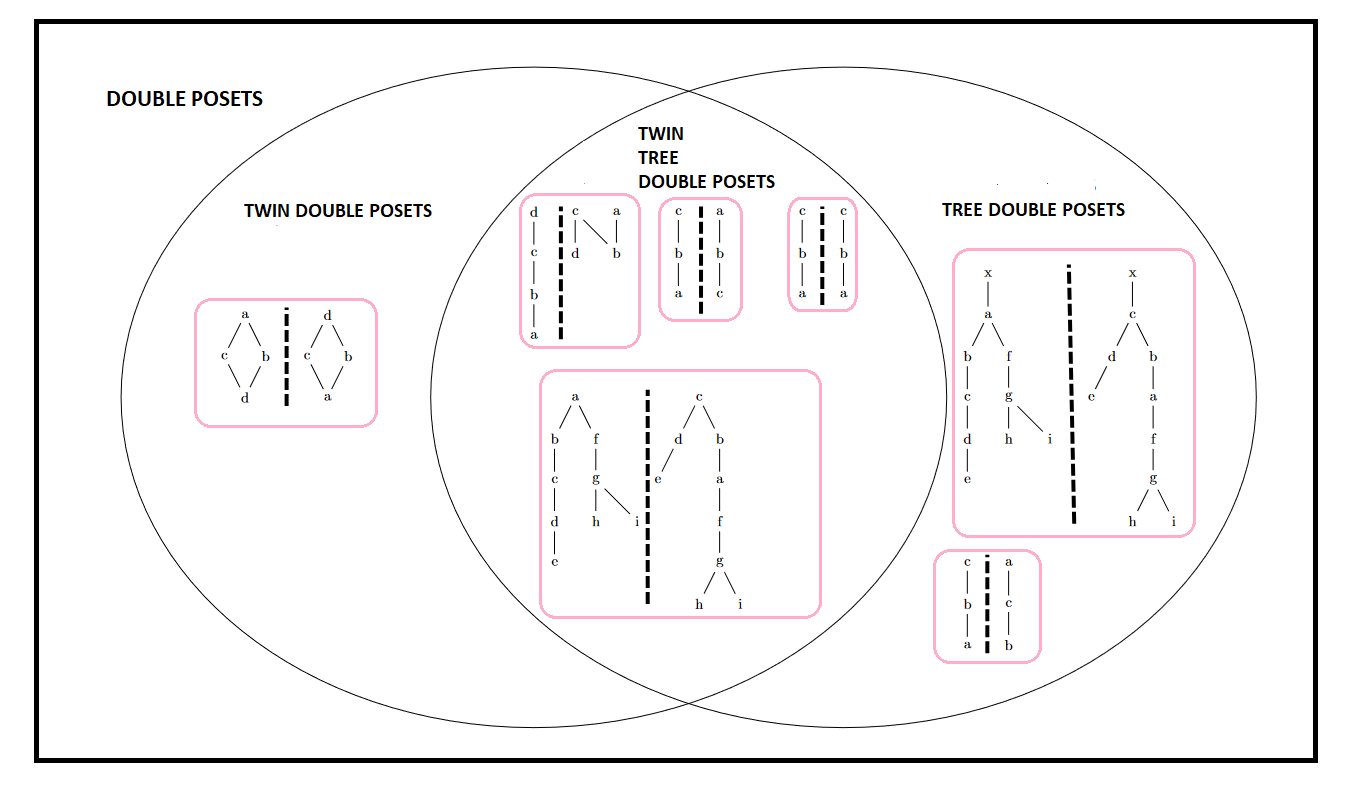}
    \caption{Double posets represented as pairs of Hasse diagrams}
\end{figure}
Corner trees are twin tree double posets in disguise, as the following lemma shows. Indeed, SN polytrees are
equivalent to twin tree double posets.

\begin{proposition}
\label{lemma:corresp_SN_poly_Twin_tree_DP}
 Let  $\Tsn$ be an SN polytree.
 Define the relations $<^{\prime}_{\West}$ and $<^{\prime}_{\South}$ on $\vertexset(\Tsn)$ by
 \begin{align*}
       u <^{\prime}_{\West} v &\iff (u,v) \in \edgeset(\Tsn) \\
       u <^{\prime}_{\South} v &\iff (u,v) \in \edgeset(\Tsn) \text{ and its label is \S}\\
       v <^{\prime}_{\South} u &\iff (u,v) \in \edgeset(\Tsn) \text{ and its label is \N}.
       \end{align*}
(Here, $(u,v) \in \edgeset(\Tsn)$ means that the arrow is pointing towards $u$.)
Let $<_{\West}$ and $<_{\South}$ be the respective
 transitive closures of $ <^{\prime}_{\West}$ and $<^{\prime}_{\South}$.
 Then ${(\vertexset(\Tsn),<_{\West}, <_{\South})}$ is a twin tree double poset. Furthermore, the map 
 \begin{align*}
 \SNpolytoSDP: \SNPolyTrees&\to \twintreedoubleposets\\
    \Tsn &\mapsto {(\vertexset(\Tsn),<_{\West}, <_{\South})}
 \end{align*}
is a bijection whose inverse we denote with $\SDPtoSNpoly$.
 \end{proposition}
 
\begin{figure}[H]
\centering
  \begin{align*}
\SNpolytoSDP \scalebox{3}{\text{$($}}\raisebox{-0.3cm}{\scalebox{0.3}{\begin{tikzpicture}[every node/.style={circle, draw, inner sep=4pt, minimum size=6mm}]
\node (X) at (0,0) {\scalebox{2}{X}};
\node (A) at (-2,-3) {\scalebox{2}{Y}};
\node (B) at (2,-3) {\scalebox{2}{Z}};
\tikzset{mid arrow/.style={
        postaction={decorate,decoration={
            markings,
            mark=at position .5 with {\arrow[scale=3]{stealth}}
        }}
    }}
        \draw[-{Stealth[scale=3]}] (X) -- (A) node[draw = none, pos=0.5, left, yshift=10pt] {\scalebox{2}{\N}};
        \draw[-{Stealth[scale=3]}] (X) -- (B) node[draw = none, pos=0.5, right, yshift=10pt] {\scalebox{2}{\N}};
\end{tikzpicture}}}\scalebox{3}{\text{$)$}} &=\raisebox{-0.5cm}{\scalebox{0.3}{\begin{tikzpicture}[every node/.style={circle, draw, inner sep=4pt, minimum size=6mm}]
\node (X) at (0,0) {\scalebox{2}{X}};
\node (A) at (-2,-3) {\scalebox{2}{Y}};
\node (B) at (2,-3) {\scalebox{2}{Z}};
\node[draw=none] (L1) at (4,1) {};
\node[draw=none] (M1) at (4,-4) {};
\node (X1) at (8,-3) {\scalebox{2}{X}};
\node (A1) at (10,0) {\scalebox{2}{Y}};
\node (B1) at (6,0) {\scalebox{2}{Z}};
\tikzset{mid arrow/.style={
        postaction={decorate,decoration={
            markings,
            mark=at position .5 with {\arrow[scale=3]{stealth}}
        }}
    }}
        \draw[-] (X) -- (A);
        \draw[-] (X) -- (B);
         \draw[-] (A1) -- (X1);
        \draw[-] (X1) -- (B1);
         \draw[-,dashed] (M1) -- (L1); 
\end{tikzpicture}}}
\end{align*}
\caption{Example of a twin tree double poset arising from an SN polytree}
\label{figure:ex_twin_tree_dp}
\end{figure}

The following map allows us to encode permutations as pairs of strict linear orders. We refer the reader to the related notion of special poset from \cite{malvenuto2011self}.
\begin{definition}
\label{definition:perm_as_double}
    % Define
    For $\sigma \in \permutations$
    define
    \begin{align*}
         \SntoDP( \sigma ) := (\{1,...,n\},1 < \cdots < n, \sigma^{-1}(1) < \cdots < \sigma^{-1}(n)) \in \DP,
    \end{align*}
    and set
    \begin{align*}
        \mathcal{S}
        :=
        \operatorname{im} \SntoDP.
        %\{ (\{1,...,n\},1 < \cdots < n, \sigma^{-1}(1) < \cdots < \sigma^{-1}(n)) : \sigma \in \permutations \}
        %\subset \mathsf{DP}.
    \end{align*}
    Then $\SntoDP: \permutations \to    \mathcal{S}$
    is obviously bijective and
    we denote its inverse by $\DPtoSn$.
\end{definition}

\begin{example}
\label{example:permutation}
Consider the permutation $\sigma = \perm{1423}$. Then 
\begin{align*}
  \SntoDP(\sigma) = (\{1,2,3,4\},1<2<3<4,1<3<4<2)
\end{align*}
and
$\DPtoSn((\{1,2,3,4\},1<2<3<4,1<3<4<2)) = \perm{1423}$. We also represent double poset that underlie permutations as in \Cref{figure:pattern1} or as in  \Cref{figure:pattern}.

\begin{figure}[H]
\centering
\begin{minipage}{0.49\textwidth}
\centering
  \begin{tikzpicture}[scale=0.4]
      \draw[->,opacity=0] (0,0) -- (5,0) {};
            \draw[->,opacity=0] (0,0) -- (0,5)  {};
      \node[draw=none] at (1, -0.5) {};
         \node[draw=none]  at (2, -0.5) {};
         \node[draw=none]  at (3, -0.5) {};
 \node[draw=none]  at (4, -0.5) {};
                    % Draw permutation points
            \foreach \x/\p in {1/1, 1/2, 1/3,1/4} {
                \fill[black] (\x,\p) circle (4pt);
            }
             \foreach \x/\p in {3/1, 3/2, 3/3,3/4} {
                \fill[black] (\x,\p) circle (4pt);
            }
    \draw[-,thick] (1,1) -- (1,2);
    \draw[-,thick] (1,3) -- (1,4);
\draw[-, thick] (1,1)  (1,4);
     \draw[-,thick] (1,2) -- (1,3);
       \draw[-,thick] (3,1) -- (3,2);
    \draw[-,thick] (3,3) -- (3,4);
     \draw[-,thick] (3,2) -- (3,3);
     \draw[-,thick,dashed] (2,0) -- (2,5);
        \node[draw=none,left] at (1,1) {\scalebox{0.75}{1}};
   \node[draw=none,left] at (1,2) {\scalebox{0.75}{2}};
   \node[draw=none,left] at (1,3) {\scalebox{0.75}{3}};
    \node[draw=none,left] at (1,4) {\scalebox{0.75}{4}};
          \node[draw=none,right] at (3,1) {\scalebox{0.75}{1}};
   \node[draw=none,right] at (3,2) {\scalebox{0.75}{3}};
   \node[draw=none,right] at (3,3) {\scalebox{0.75}{4}};
    \node[draw=none,right] at (3,4) {\scalebox{0.75}{2}};
\end{tikzpicture}
\caption{Pairs of Hasse diagrams}
    \label{figure:pattern1}
\end{minipage}
\begin{minipage}{0.49\textwidth}
\centering
    \begin{tikzpicture}[scale=0.4,every node/.style={circle, draw, inner sep=2pt, minimum size=6mm}]
            % Add ticks to x axis with custom labels
               \draw[->] (0,0) -- (5,0) {};
            \draw[->] (0,0) -- (0,5)  {};
         \node[draw=none] at (1, -0.5) {\textup{1}};
         \node[draw=none]  at (2, -0.5) {\textup{2}};
         \node[draw=none]  at (3, -0.5) {\textup{3}};
 \node[draw=none]  at (4, -0.5) {\textup{4}};
                    % Draw permutation points
            \foreach \x/\p in {1/1, 2/4, 3/2,4/3} {
                \fill[black] (\x,\p) circle (4pt);
            }

         %   \draw (1,1) -- (2,4);  % Connect points 3 to 4
          %  \draw (2,4) -- (3,2);  % Connect points 4 to 2
           % \draw (3,2) -- (4,3);  % Connect points 4 to 2
\end{tikzpicture}
\caption{\West and \South posets}
    \label{figure:pattern}
\end{minipage}
\end{figure}

\end{example}

\begin{remark}
   For any permutation $\sigma$, $\SntoDP(\sigma) \in \treedoubleposets$. We have $\SntoDP(\sigma) \in \twintreedoubleposets$ only for the cases $\sigma = \id$
   and $\sigma = \mathtt{[n\,\cdots\,3\,2\,1]}$ (the full reversal).
\end{remark}

With these maps in hand,
we can now characterize corner tree occurrences
as morphisms in the category of double posets.
\begin{proposition}
\label{proposition:occ_as_DP}
Let $\Pi\in \permutations(n)$ and $\CT \in \CornerTrees$. Then 
    \begin{align*}
\{f:\vertexset(\CT)\to [n]\,:\,\,f\,\,\text{is an occurrence of }\,\CT\,\,\text{in}\,\,\Pi\} 
&= \Hom(  \CTtoDP(\CT), \SntoDP(\Pi) ),
\end{align*}
where $\Hom(\smalldp,\smalldpprime)$ denotes the set of morphisms between the double posets $\smalldp,\smalldp^{\prime}$.
\end{proposition}
The proof is simple and is left as an exercise to the reader.

\begin{example}
\label{example:occurrence}
Consider the permutation from \Cref{example:permutation} and the corner tree from \Cref{figure:ex_twin_tree_dp_from_corner}. The map $f:\{\textup{X},\textup{Y},\textup{Z}\}\to [4]$, $f(\textup{X}) = 3, f(\textup{Y}) = 2, f(\textup{Z})=2$ is an occurrence of the corner tree on the permutation. It can also be seen as a morphism between the two underlying double posets.
\begin{figure}[H]
\centering
  \begin{align*}
\CTtoDP \scalebox{3}{\text{$($}}\raisebox{-0.3cm}{\scalebox{0.3}{\begin{tikzpicture}[every node/.style={circle, draw, inner sep=4pt, minimum size=6mm}]
\node (X) at (0,0) {\scalebox{2}{X}};
\node (A) at (-2,-3) {\scalebox{2}{Y}};
\node (B) at (2,-3) {\scalebox{2}{Z}};
\tikzset{mid arrow/.style={
        postaction={decorate,decoration={
            markings,
            mark=at position .5 with {\arrow[scale=3]{stealth}}
        }}
    }}
        \draw[-] (X) -- (A) node[draw = none, pos=0.5, left, yshift=10pt] {\scalebox{2}{\NW}};
        \draw[-] (X) -- (B) node[draw = none, pos=0.5, right, yshift=10pt] {\scalebox{2}{\NW}};
\end{tikzpicture}}}\scalebox{3}{\text{$)$}} &=\raisebox{-0.5cm}{\scalebox{0.3}{\begin{tikzpicture}[every node/.style={circle, draw, inner sep=4pt, minimum size=6mm}]
\node (X) at (0,0) {\scalebox{2}{X}};
\node (A) at (-2,-3) {\scalebox{2}{Y}};
\node (B) at (2,-3) {\scalebox{2}{Z}};
\node[draw=none] (L1) at (4,1) {};
\node[draw=none] (M1) at (4,-4) {};
\node (X1) at (8,-3) {\scalebox{2}{X}};
\node (A1) at (10,0) {\scalebox{2}{Y}};
\node (B1) at (6,0) {\scalebox{2}{Z}};
\tikzset{mid arrow/.style={
        postaction={decorate,decoration={
            markings,
            mark=at position .5 with {\arrow[scale=3]{stealth}}
        }}
    }}
        \draw[-] (X) -- (A);
        \draw[-] (X) -- (B);
         \draw[-] (A1) -- (X1);
        \draw[-] (X1) -- (B1);
         \draw[-,dashed] (M1) -- (L1); 
\end{tikzpicture}}}
\end{align*}
\caption{Example of a twin tree double poset arising from a corner tree}
\label{figure:ex_twin_tree_dp_from_corner}
\end{figure}
Observe that $\CTtoDP:= \SNpolytoSDP \circ \CTtoSNpoly$.
\end{example}

\begin{remark}
\label{remark:rooting_tt}
If $\smallct \in \twintreedoubleposets$ and $v \in \vertexset(\smallct)$, then $\TTDPtoCT(\smallct,v)$ denotes the corner tree obtained by picking $v$ as the root and labeling the edges accordingly.
\end{remark}

\section{Counting permutations using double posets}
\label{section:count_perm_usingDP}

\subsection{Switching between homomorphisms and embeddings}

In this section, we work with the free $\Q$-vector space on $\DP$, the set of isomorphism classes of finite double posets
\begin{align*}
   \Q[\DP]:=\bigoplus_{n} \Q[\doubleposets(n)],
\end{align*}
where $\doubleposets(n)$ denotes the set of equivalence classes of double posets with $n$ elements. We introduce two linear functionals on $\Q[\DP]$ that encode the number of homomorphisms and double order embeddings respectively. The regular monomorphisms in the category of strict double posets are precisely the double order embeddings, as shown in \Cref{section:suppl}.
\begin{definition}
\label{definition:DP_linear_functionals}
Fix a (large) double poset $\largedp\in \bigcup_n \DP(n)$. On basis elements $\smalldp \in \bigcup_n \doubleposets(n)$ define linear functionals $\bigoplus_{n} \Q[\doubleposets(n)] \to \Q$.
   \begin{align*}
    \langle \Profile_{\hom}(\largedp), \smalldp \rangle&:=|\Hom(\smalldp,\largedp)|\\
    \langle \Profile_{\regmono}(\largedp), \smalldp \rangle&:=|\RegMono(\smalldp,\largedp)|.
\end{align*} 
\end{definition}
We denote them with $\Profile$, since they encode two different  graph profiles of a double poset, here regarded as a directed graph.

\begin{example}
Consider the following two examples where we write the Hasse diagrams of the underlying isomorphism classes of the double posets.
\begin{align*}
\Profile_{\hom}\;\raisebox{-0.2cm}{\scalebox{0.2}{\begin{tikzpicture}[every node/.style={circle, draw, inner sep=2pt, minimum size=6mm}]
  % Define the vertices
   \node[fill=black][fill=black,label=right:{\scalebox{3}{1}}] 
   (A) at (0,0) {};
  \node[fill=black][fill=black,label=right:{\scalebox{3}{2}}]  
   (C) at (0,2) {};
      \node[fill=black][fill=black,label=right:{\scalebox{3}{1}}]  (A1) at (4,0) {};
  \node[fill=black][fill=black,label=right:{\scalebox{3}{2}}]  (C1) at (4,2) {};
   \node[draw=none]  (A2) at (2,3) {};
    \node[draw=none] (C2) at (2,-1) {};
  \tikzset{mid arrow/.style={
    postaction={decorate,decoration={
      markings,
      mark=at position .5 with {\arrow[scale=3]{stealth}}
    }}
  }}
% Draw the edges with arrows in the middle and labels
  \draw[line width=3pt,-](C) -- (A) node[draw = none, pos=0.5, above, yshift=10pt] {};
   \draw[line width=3pt,-](C1) -- (A1) node[draw = none, pos=0.5, above, yshift=10pt] {};
   \draw[line width=3pt, dashed,-](C2) -- (A2) node[draw = none, pos=0.5, above, yshift=10pt] {};
\end{tikzpicture}}} &= \e +2\,\, \raisebox{-0.1cm}{\scalebox{0.2}{\begin{tikzpicture}[every node/.style={circle, draw, inner sep=2pt, minimum size=6mm}]
  % Define the vertices
   \node[fill=black][fill=black,label=right:{\scalebox{3}{1}}]  
   (A) at (0,0) {};
\node[draw=none]  
   (C) at (2,2) {};
   \node[draw=none]  
   (D) at (2,-1) {};
   \node[fill=black][fill=black,label=right:{\scalebox{3}{1}}]  
   (B) at (4,0) {};
  \tikzset{mid arrow/.style={
    postaction={decorate,decoration={
      markings,
      mark=at position .5 with {\arrow[scale=3]{stealth}}
    }}
  }}
  \draw[line width=3pt,dashed](C) -- (D) node[draw = none, pos=0.5, above, yshift=10pt] {};
\end{tikzpicture}}}\,+4\,\,\raisebox{-0.1cm}{\scalebox{0.2}{\begin{tikzpicture}[every node/.style={circle, draw, inner sep=2pt, minimum size=6mm}]
  % Define the vertices
   \node[fill=black][fill=black,label=right:{\scalebox{3}{2}}]  
   (A) at (0,0) {};
  \node[fill=black][fill=black,label=right:{\scalebox{3}{1}}]  (B) at (0,-2) {};
  \node[fill=black][fill=black,label=right:{\scalebox{3}{2}}]  (C) at (3,0) {};
    \node[fill=black][fill=black,label=right:{\scalebox{3}{2}}]  (C) at (3,0) {};
  \node[fill=black][fill=black,label=right:{\scalebox{3}{1}}]  (D) at (3,-2) {};
    \node[draw=none]  (E) at (2,1) {};
      \node[draw=none]  (F) at (2,-3) {};
  \tikzset{mid arrow/.style={
    postaction={decorate,decoration={
      markings,
      mark=at position .5 with {\arrow[scale=3]{stealth}}
    }}
  }}
   \draw[line width=3pt,dashed](E) -- (F) node[draw = none, pos=0.5, above, yshift=10pt] {};
% Draw the edges with arrows in the middle and labels
 \end{tikzpicture}}}\,+\,\,\raisebox{-0.1cm}{\scalebox{0.2}{\begin{tikzpicture}[every node/.style={circle, draw, inner sep=2pt, minimum size=6mm}]
  % Define the vertices
   \node[fill=black][fill=black,label=right:{\scalebox{3}{2}}]  
   (A) at (0,0) {};
  \node[fill=black][fill=black,label=right:{\scalebox{3}{1}}]  (B) at (0,-2) {};
  \node[fill=black][fill=black,label=right:{\scalebox{3}{2}}]  (C) at (3,0) {};
    \node[fill=black][fill=black,label=right:{\scalebox{3}{2}}]  (C) at (3,0) {};
  \node[fill=black][fill=black,label=right:{\scalebox{3}{1}}]  (D) at (3,-2) {};
    \node[draw=none]  (E) at (2,1) {};
      \node[draw=none]  (F) at (2,-3) {};
  \tikzset{mid arrow/.style={
    postaction={decorate,decoration={
      markings,
      mark=at position .5 with {\arrow[scale=3]{stealth}}
    }}
  }}
   \draw[line width=3pt,dashed](E) -- (F) node[draw = none, pos=0.5, above, yshift=10pt] {};
   \draw[line width=3pt,-](A) -- (B) node[draw = none, pos=0.5, above, yshift=10pt] {};
% Draw the edges with arrows in the middle and labels
 \end{tikzpicture}}}\;+\,\,\raisebox{-0.1cm}{\scalebox{0.2}{\begin{tikzpicture}[every node/.style={circle, draw, inner sep=2pt, minimum size=6mm}]
  % Define the vertices
   \node[fill=black][fill=black,label=right:{\scalebox{3}{2}}]  
   (A) at (0,0) {};
  \node[fill=black][fill=black,label=right:{\scalebox{3}{1}}]  (B) at (0,-2) {};
  \node[fill=black][fill=black,label=right:{\scalebox{3}{2}}]  (C) at (3,0) {};
    \node[fill=black][fill=black,label=right:{\scalebox{3}{2}}]  (C) at (3,0) {};
  \node[fill=black][fill=black,label=right:{\scalebox{3}{1}}]  (D) at (3,-2) {};
    \node[draw=none]  (E) at (2,1) {};
      \node[draw=none]  (F) at (2,-3) {};
  \tikzset{mid arrow/.style={
    postaction={decorate,decoration={
      markings,
      mark=at position .5 with {\arrow[scale=3]{stealth}}
    }}
  }}
   \draw[line width=3pt,dashed](E) -- (F) node[draw = none, pos=0.5, above, yshift=10pt] {};
    \draw[line width=3pt,-](C) -- (D) node[draw = none, pos=0.5, above, yshift=10pt] {};
% Draw the edges with arrows in the middle and labels
 \end{tikzpicture}}}\;+\,\,\raisebox{-0.1cm}{\scalebox{0.2}{\begin{tikzpicture}[every node/.style={circle, draw, inner sep=2pt, minimum size=6mm}]
  % Define the vertices
   \node[fill=black][fill=black,label=right:{\scalebox{3}{2}}]  
   (A) at (0,0) {};
  \node[fill=black][fill=black,label=right:{\scalebox{3}{1}}]  (B) at (0,-2) {};
  \node[fill=black][fill=black,label=right:{\scalebox{3}{2}}]  (C) at (3,0) {};
    \node[fill=black][fill=black,label=right:{\scalebox{3}{2}}]  (C) at (3,0) {};
  \node[fill=black][fill=black,label=right:{\scalebox{3}{1}}]  (D) at (3,-2) {};
    \node[draw=none]  (E) at (2,1) {};
      \node[draw=none]  (F) at (2,-3) {};
  \tikzset{mid arrow/.style={
    postaction={decorate,decoration={
      markings,
      mark=at position .5 with {\arrow[scale=3]{stealth}}
    }}
  }}
   \draw[line width=3pt,dashed](E) -- (F) node[draw = none, pos=0.5, above, yshift=10pt] {};
    \draw[line width=3pt,-](A) -- (B) node[draw = none, pos=0.5, above, yshift=10pt] {};
     \draw[line width=3pt,-](C) -- (D) node[draw = none, pos=0.5, above, yshift=10pt] {};
% Draw the edges with arrows in the middle and labels
 \end{tikzpicture}}}+\cdots\\
\Profile_{\regmono}\raisebox{-0.4cm}{\scalebox{0.2}{\begin{tikzpicture}[every node/.style={circle, draw, inner sep=2pt, minimum size=6mm}]
  % Define the vertices
   \node[fill=black][fill=black,label=right:{\scalebox{3}{2}}] 
   (A) at (0,0) {};
  \node[fill=black][fill=black,label=right:{\scalebox{3}{1}}] (B) at (0,-2) {};
  \node[fill=black][fill=black,label=right:{\scalebox{3}{3}}] 
   (C) at (0,2) {};
     \node[fill=black][fill=black,label=right:{\scalebox{3}{2}}] (A1) at (4,0) {};
  \node[fill=black][fill=black,label=right:{\scalebox{3}{1}}] (B1) at (4,-2) {};
  \node[fill=black][fill=black,label=right:{\scalebox{3}{3}}] (C1) at (4,2) {};
  \node[draw=none] 
   (X) at (2,3) {};
   \node[draw=none] 
   (Y) at (2,-3) {};
 Define styles for the arrowheads
  \tikzset{mid arrow/.style={
    postaction={decorate,decoration={
      markings,
      mark=at position .5 with {\arrow[scale=3]{stealth}}
    }}
  }}
% Draw the edges with arrows in the middle and labels
  \draw[line width=3pt, -](B) -- (A) node[draw = none, pos=0.5, above, yshift=10pt] {};
   \draw[line width=3pt,-](A) -- (C) node[draw = none, pos=0.5, above, yshift=10pt] {};
      \draw[line width=3pt,-] (B1) to (C1);
      \draw[line width=3pt,-] (A1) to (C1);
      \draw[line width=3pt,dashed] (X) to (Y);
\end{tikzpicture}}}&= \e +3\,\, \raisebox{-0.1cm}{\scalebox{0.2}{\begin{tikzpicture}[every node/.style={circle, draw, inner sep=2pt, minimum size=6mm}]
  % Define the vertices
   \node[fill=black][fill=black,label=right:{\scalebox{3}{1}}]  
   (A) at (0,0) {};
\node[draw=none]  
   (C) at (2,2) {};
   \node[draw=none]  
   (D) at (2,-1) {};
   \node[fill=black][fill=black,label=right:{\scalebox{3}{1}}]  
   (B) at (4,0) {};
  \tikzset{mid arrow/.style={
    postaction={decorate,decoration={
      markings,
      mark=at position .5 with {\arrow[scale=3]{stealth}}
    }}
  }}
  \draw[line width=3pt,dashed](C) -- (D) node[draw = none, pos=0.5, above, yshift=10pt] {};
\end{tikzpicture}}}\,+3\,\,\raisebox{-0.1cm}{\scalebox{0.2}{\begin{tikzpicture}[every node/.style={circle, draw, inner sep=2pt, minimum size=6mm}]
  % Define the vertices
   \node[fill=black][fill=black,label=right:{\scalebox{3}{2}}]  
   (A) at (0,0) {};
  \node[fill=black][fill=black,label=right:{\scalebox{3}{1}}]  (B) at (0,-2) {};
  \node[fill=black][fill=black,label=right:{\scalebox{3}{2}}]  (C) at (3,0) {};
    \node[fill=black][fill=black,label=right:{\scalebox{3}{2}}]  (C) at (3,0) {};
  \node[fill=black][fill=black,label=right:{\scalebox{3}{1}}]  (D) at (3,-2) {};
    \node[draw=none]  (E) at (2,1) {};
      \node[draw=none]  (F) at (2,-3) {};
  \tikzset{mid arrow/.style={
    postaction={decorate,decoration={
      markings,
      mark=at position .5 with {\arrow[scale=3]{stealth}}
    }}
  }}
   \draw[line width=3pt,dashed](E) -- (F) node[draw = none, pos=0.5, above, yshift=10pt] {};
    \draw[line width=3pt,-](A) -- (B) node[draw = none, pos=0.5, above, yshift=10pt] {};
     \draw[line width=3pt,-](C) -- (D) node[draw = none, pos=0.5, above, yshift=10pt] {};
% Draw the edges with arrows in the middle and labels
 \end{tikzpicture}}}\,+\raisebox{-0.4cm}{\scalebox{0.2}{\begin{tikzpicture}[every node/.style={circle, draw, inner sep=2pt, minimum size=6mm}]
  % Define the vertices
   \node[fill=black][fill=black,label=right:{\scalebox{3}{2}}] 
   (A) at (0,0) {};
  \node[fill=black][fill=black,label=right:{\scalebox{3}{1}}] (B) at (0,-2) {};
  \node[fill=black][fill=black,label=right:{\scalebox{3}{3}}] 
   (C) at (0,2) {};
     \node[fill=black][fill=black,label=right:{\scalebox{3}{2}}] (A1) at (4,0) {};
  \node[fill=black][fill=black,label=right:{\scalebox{3}{1}}] (B1) at (4,-2) {};
  \node[fill=black][fill=black,label=right:{\scalebox{3}{3}}] (C1) at (4,2) {};
  \node[draw=none] 
   (X) at (2,3) {};
   \node[draw=none] 
   (Y) at (2,-3) {};
 Define styles for the arrowheads
  \tikzset{mid arrow/.style={
    postaction={decorate,decoration={
      markings,
      mark=at position .5 with {\arrow[scale=3]{stealth}}
    }}
  }}
% Draw the edges with arrows in the middle and labels
  \draw[line width=3pt, -](B) -- (A) node[draw = none, pos=0.5, above, yshift=10pt] {};
   \draw[line width=3pt,-](A) -- (C) node[draw = none, pos=0.5, above, yshift=10pt] {};
      \draw[line width=3pt,-] (B1) to (C1);
      \draw[line width=3pt,-] (A1) to (C1);
      \draw[line width=3pt,dashed] (X) to (Y);
\end{tikzpicture}}}
\end{align*}
\end{example}

As seen in 
\Cref{proposition:occ_as_DP},
$\Profile_{\hom}$ is compatible
with corner tree counting. On the other hand, $\Profile_{\regmono}$ is compatible
with permutation counting, as shown in the following lemma
\begin{lemma}
  \label{lem:regmono_perm_patt}
  \begin{align*}
    \langle \Profile_{\regmono}(\SntoDP(\Pi)), \SntoDP(\sigma) \rangle
    &=
    \langle \PC(\Pi), \sigma \rangle.
  \end{align*}
\end{lemma}

Analogous linear maps on simple graphs are defined in \cite{Caudillo2025}. They were originally introduced by Lovasz, see \cite{lovasz2012large}, to count subgraphs and induced subgraphs in terms of graph homomorphisms. In our context, these maps can be used to obtain the permutation patterns counted by double posets.

\begin{theorem}
\label{theorem:lovasz_change_basis}
Consider the linear map
 \begin{align*}
 \Phi_{\regmono \leftarrow \hom }:  \bigoplus_{n} \Q[\doubleposets(n)] &\to  \bigoplus_{n} \Q[\doubleposets(n)]\\\smalldp &\mapsto \sum_{\smalldp^{\prime}}\frac{|\Epi(\smalldp,\smalldp^{\prime})|}{|\Aut(\smalldp^{\prime})|} \smalldp^{\prime}.
 \end{align*}
Then for $\smalldp,\largedp \in \doubleposets$, we have
  \begin{align*}
      \langle \Profile_{\hom}(\largedp), \smalldp \rangle &= \langle \Profile_{\regmono}(\largedp), \Phi_{\regmono\leftarrow \hom}( \smalldp ) \rangle
  \end{align*}
\end{theorem}
\begin{proof}
The result follows immediately by applying \Cref{theorem:hom_reg_mono}.
\end{proof}
\subsection{Counting double poset morphisms when $\largedp = \SntoDP(\Pi)$}
\label{ss:countbigguyperm}

As shown in \Cref{proposition:countpermasperm}, when the ``large'' double poset corresponds to a permutation, counting double posets morphisms can always be rewritten as counting permutation patterns.

We first define the projection on the space of permutations embedded as double posets.

\begin{definition}
\label{definition:projection}
Define
  \begin{align*}
    \proj_{\mathcal{S}}: \bigoplus_{n} \Q[\doubleposets(n)] &\to  \bigoplus_{n} \Q[\doubleposets(n)],
  \end{align*}
to be the linear map,
which on basis elements is given as
  \begin{align*}
    \proj_{\mathcal{S}}(\smalldp) =
    \begin{cases}
      \smalldp, & \text{ if $\smalldp = \SntoDP(\sigma)$ for some $\sigma \in \permutations$},\\
      0, & \text{ otherwise}.
    \end{cases}
  \end{align*}
\end{definition}
The following two lemmas will be used in the proof of \Cref{proposition:countpermasperm}.
Notice that $\proj_{\mathcal{S}} \circ\;\Phi_{\regmono \leftarrow \hom }$ is the map already hinted at in \cite[Lemma 2.1]{even2021counting}. 
\begin{lemma}
\label{lemma:map_DP_lc_perm_patt}
  \begin{align*}
    \proj_{\mathcal{S}}(\Phi_{\regmono \leftarrow \hom }(\smalldp)) =  \sum_{\sigma}|\Epi(\smalldp,\SntoDP(\sigma))|\SntoDP(\sigma).
\end{align*}
\end{lemma}
\begin{proof}
Clearly $\forall \sigma \in \permutations:\;\Aut(\SntoDP(\sigma))=\{\id\}$ and therefore $|\Aut(\SntoDP(\sigma))|=1$. 
\end{proof}
\begin{lemma}
  \label{lem:projection-perm}
  \begin{align*}
    \langle \Profile_{\regmono}(\SntoDP(\Pi)), \smalldp \rangle
    =
    \langle \Profile_{\regmono}(\SntoDP(\Pi)), \proj_{\mathcal{S}}(\smalldp) \rangle
  \end{align*}
\end{lemma}
\begin{proof}
Let $\smalldp =\triplesmalldp$ and let $f:\triplesmalldp \to \SntoDP(\Pi)$ be a regular monomorphism. Then $P_{A}$ and $Q_{A}$ are total orders. Therefore, we have $\smalldp \cong \SntoDP(\sigma)$
 for some permutation $\sigma$.
\end{proof}

The following proposition shows that counting double post morphisms into a permutation can always be rewritten as a linear combination of permutation pattern counts.

\begin{proposition}
\label{proposition:countpermasperm}
For any $\smalldp \in \doubleposets, \Pi \in \permutations$,
\begin{align*}
    \langle \Profile_{\hom}(\SntoDP(\Pi)), \smalldp \rangle &=   \langle\PC(\Pi),\sum_{\sigma \in \permutations}|\Epi(\smalldp,\SntoDP(\sigma))|\ \sigma \rangle.
\end{align*}
\end{proposition}
\begin{proof}
\begin{align*}
    &\langle \Profile_{\hom}(\SntoDP(\Pi)), \smalldp\rangle \\&\stackrel{\text{\Cref{theorem:lovasz_change_basis}}}{=}   \langle \Profile_{\regmono}(\SntoDP(\Pi)), \Phi_{\regmono \leftarrow \hom }(\smalldp)\rangle\\
    &\stackrel{\text{\Cref{lem:projection-perm}}}{=} \langle \Profile_{\regmono}(\SntoDP(\Pi)), \proj_{\mathcal{S}}  \circ\; \Phi_{\regmono \leftarrow \hom }(\smalldp)\rangle\\&\stackrel{\text{\Cref{lemma:map_DP_lc_perm_patt}}}{=}   \sum_{\sigma \in \permutations} |\Epi(\smalldp,\SntoDP(\sigma))| \langle \Profile_{\regmono}(\SntoDP(\Pi)), \SntoDP(\sigma) \rangle\\
    &\stackrel{\text{\Cref{lem:regmono_perm_patt}}}{=}   \sum_{\sigma \in \permutations}|\Epi(\smalldp,\SntoDP(\sigma))|\langle\PC(\Pi),\sigma \rangle
\end{align*}
\end{proof}

\begin{remark}
\label{remark:ct_lc}
To compare with \cite{even2021counting}, denote a corner tree with $T$. There is a one-to-one correspondence between the occurrences of a corner tree $T$ on a permutation $\sigma$ and the morphisms between the corresponding twin tree double poset $\smallct:=\SNpolytoSDP(\CTtoSNpoly(T))$  and $\SntoDP(\sigma)$. If we consider the \underline{underlying set maps}, see \Cref{proposition:occ_as_DP}, they coincide and therefore $\#T(\sigma)$, as written in \cite{even2021counting}, is equal to 
\begin{align*}
  \langle \Profile_{\hom}(\SntoDP(\sigma)), \smallct\rangle
\end{align*}
in our setting.
Using \Cref{proposition:countpermasperm}, we have
\begin{align*}
   \langle \Profile_{\hom}(\SntoDP(\sigma)), \smallct \rangle &=  \sum_{\tau\in \permutations}|\Epi(\smallct,\SntoDP(\tau))|\langle\PC(\sigma),\tau \rangle
\end{align*}
which is equivalent to the equation of \cite[Lemma 2.1]{even2021counting}. In \cite{diehl_verri_generalization_2025}, the function {\algfont expand\_patterns} applied to a corner tree $\CT$ yields
\begin{align*}
     \sum_{\sigma \in \permutations}|\Epi(\CTtoDP(\CT),\SntoDP(\sigma))|\,\sigma.
\end{align*}
\end{remark}

\section{Generalization of the $[\mathtt{3\,2\,1\,4}]$ algorithm}
\label{section:gen_algo}
This section is organized as follows.
In \Cref{subsection:family_dp_algo}, we define the family of tree double posets on which our $\tilde{\mathcal{O}}(n^{5/3})$ algorithm works. We then illustrate properties specific to this family that allow for the algorithm implementation. \Cref{subsection:algo_occurrences} revisits \Cref{algorithm_occurrences} from \cite{even2021counting}, which counts corner tree occurrences in $\tilde{\mathcal{O}}(n)$ time. We then show that this algorithm can be reformulated as \Cref{algorithm_pure_west} for the case of a pure west corner tree. Our algorithm is presented in \Cref{subsection:algo_arbone}. Finally, in \Cref{subsection:new_dir_level5}, we illustrate how twelve elements of this family of double posets span new directions at level five.

\subsection{Family of double posets for which the algorithm works: $\ArboNE$}
\label{subsection:family_dp_algo}

We refer to \cite[Section 4]{even2021counting} where an algorithm designed to count occurrences of the pattern $\perm{3214}$ in $\tilde{\mathcal{O}}(n^{5/3})$ is introduced. Here we extend it to a certain family of double posets. The main building block is the family of corner trees, $\CornerTreesfivethree$, from which the family of double posets to which the algorithm applies is constructed, $\ArboNE$.

\begin{definition}[$\CornerTreesfivethree$]
\label{definition:corner_algo}
Consider the family of corner trees constructed as follows. Denote with $\ME$ (``most east'') the root of the corner tree, and then consider the case where $\ME$ has one child, denoted with 
 $\MN$ (``most north''), with an edge labeled $\NW$, or the case where $\MN$ is the child of the child of $\ME$ and both edges are labeled $\NW$
\begin{figure}[H]
\begin{minipage}{0.49\textwidth}
   \centering
        \scalebox{0.25}{
\begin{tikzpicture}[every node/.style={circle, draw, inner sep=4pt, minimum size=6mm}]
\node (A) at (-2,0) {\scalebox{3}{$\ME$}};
\node (C) at (6,-8) {\scalebox{3}{$\MN$}};
    \tikzset{mid arrow/.style={
        postaction={decorate,decoration={
            markings,
            mark=at position .5 with {\arrow[scale=3]{stealth}}
        }}
    }}
    \draw (C) -- (A) node[draw = none, pos=0.5, above, yshift=10pt] {\scalebox{2}{\NW}};
\end{tikzpicture}}
\end{minipage}
\begin{minipage}{0.49\textwidth}
   \centering
        \scalebox{0.25}{
\begin{tikzpicture}[every node/.style={circle, draw, inner sep=4pt, minimum size=6mm}]
\node (A) at (-2,0) {\scalebox{3}{$\ME$}};
\node (B) at (2,-4) {\phantom{\scalebox{2}{$\MN$}}};
\node (C) at (6,-8) {\scalebox{3}{$\MN$}};
    \tikzset{mid arrow/.style={
        postaction={decorate,decoration={
            markings,
            mark=at position .5 with {\arrow[scale=3]{stealth}}
        }}
    }}
    \draw (B) -- (A) node[draw = none, pos=0.2, above, yshift=10pt] {\scalebox{2}{\NW}};
    \draw (C) -- (B) node[draw = none, pos=0.2, above, yshift=10pt] {\scalebox{2}{\NW}};
\end{tikzpicture}}
\end{minipage}
\end{figure}
Then we can attach an arbitrary number of children as long as the label of the edges is $\SW$
\begin{figure}[H]
\begin{minipage}{0.49\textwidth}
   \centering
        \scalebox{0.25}{
\begin{tikzpicture}[every node/.style={circle, draw, inner sep=4pt, minimum size=6mm}]
\node (A) at (-2,0) {\scalebox{3}{$\ME$}};
\node (C) at (6,-8) {\scalebox{3}{$\MN$}};
\node (D) at (-2,-4) {\phantom{\scalebox{2}{$\MN$}}};
\node (E) at (-6,-4) {\phantom{\scalebox{2}{$\MN$}}};
\node (F) at (-6,-8) {\phantom{\scalebox{2}{$\MN$}}};
\node (G) at (-10,-8) {\phantom{\scalebox{2}{$\MN$}}};
\node (K) at (6,-12) {\phantom{\scalebox{2}{$\MN$}}};
\node (L) at (10,-12) {\phantom{\scalebox{2}{$\MN$}}};
    \tikzset{mid arrow/.style={
        postaction={decorate,decoration={
            markings,
            mark=at position .5 with {\arrow[scale=3]{stealth}}
        }}
    }}
      \draw (C) -- (A) node[draw = none, pos=0.5, above, yshift=10pt] {\scalebox{2}{\NW}};
       \draw (E) -- (A) node[draw = none, pos=0.2, above, yshift=10pt] {\scalebox{2}{\SW}};
             \draw (F) -- (E) node[draw = none, pos=0.3, right, yshift=10pt] {\scalebox{2}{\SW}};
              \draw (G) -- (E) node[draw = none, pos=0.2, above, yshift=10pt] {\scalebox{2}{\SW}};
               \draw (D) -- (A) node[draw = none, pos=0.2, left, yshift=10pt] {\scalebox{2}{\SW}};
                    \draw (C) -- (K) node[draw = none, pos=0.7, left, yshift=10pt] {\scalebox{2}{\SW}};
                     \draw (C) -- (L) node[draw = none, pos=0.8, above, yshift=10pt] {\scalebox{2}{\SW}};
\end{tikzpicture}}
\end{minipage}
\begin{minipage}{0.49\textwidth}
   \centering
        \scalebox{0.25}{
\begin{tikzpicture}[every node/.style={circle, draw, inner sep=4pt, minimum size=6mm}]
\node (A) at (-2,0) {\scalebox{3}{$\ME$}};
\node (B) at (2,-4) {\phantom{\scalebox{2}{$\MN$}}};
\node (C) at (6,-8) {\scalebox{3}{$\MN$}};
\node (D) at (-2,-4) {\phantom{\scalebox{2}{$\MN$}}};
\node (E) at (-6,-4) {\phantom{\scalebox{2}{$\MN$}}};
\node (F) at (-6,-8) {\phantom{\scalebox{2}{$\MN$}}};
\node (G) at (-10,-8) {\phantom{\scalebox{2}{$\MN$}}};
\node (H) at (2,-8) {\phantom{\scalebox{2}{$\MN$}}};
\node (I) at (-2,-8) {\phantom{\scalebox{2}{$\MN$}}};
\node (J) at (-2,-12) {\phantom{\scalebox{2}{$\MN$}}};
\node (K) at (6,-12) {\phantom{\scalebox{2}{$\MN$}}};
\node (L) at (10,-12) {\phantom{\scalebox{2}{$\MN$}}};
    \tikzset{mid arrow/.style={
        postaction={decorate,decoration={
            markings,
            mark=at position .5 with {\arrow[scale=3]{stealth}}
        }}
    }}
    \draw (B) -- (A) node[draw = none, pos=0.2, above, yshift=10pt] {\scalebox{2}{\NW}};
    \draw (C) -- (B) node[draw = none, pos=0.2, above, yshift=10pt] {\scalebox{2}{\NW}};
       \draw (E) -- (A) node[draw = none, pos=0.2, above, yshift=10pt] {\scalebox{2}{\SW}};
             \draw (F) -- (E) node[draw = none, pos=0.3, right, yshift=10pt] {\scalebox{2}{\SW}};
              \draw (G) -- (E) node[draw = none, pos=0.2, above, yshift=10pt] {\scalebox{2}{\SW}};
               \draw (D) -- (A) node[draw = none, pos=0.2, left, yshift=10pt] {\scalebox{2}{\SW}};
               \draw (B) -- (H) node[draw = none, pos=0.7, left, yshift=10pt] {\scalebox{2}{\SW}};
                \draw (B) -- (I) node[draw = none, pos=0.5, left, yshift=10pt] {\scalebox{2}{\SW}};
                   \draw (H) -- (J) node[draw = none, pos=0.6, above, yshift=10pt] {\scalebox{2}{\SW}};
                    \draw (C) -- (K) node[draw = none, pos=0.7, left, yshift=10pt] {\scalebox{2}{\SW}};
                     \draw (C) -- (L) node[draw = none, pos=0.8, above, yshift=10pt] {\scalebox{2}{\SW}};
\end{tikzpicture}}
\end{minipage}
\end{figure}
The nodes $\ME$ and $\MN$ are to the east and to the north, respectively, of any other node.
\end{definition}
We now introduce the family of double posets to which the algorithm applies.
\begin{definition}[$\ArboNE$]
\label{definition:algo_guys}
Consider the family of double posets constructed as follows. First, consider the twin tree double posets that underlie the family of corner trees, $\CornerTreesfivethree$, from \Cref{definition:corner_algo}
\begin{figure}[H]
\begin{minipage}{0.49\textwidth}
   \centering
        \scalebox{0.25}{
\begin{tikzpicture}[every node/.style={circle, draw, inner sep=4pt, minimum size=6mm}]
\node (A) at (-2,0) {\scalebox{3}{$\ME$}};
\node (C) at (6,-8) {\scalebox{3}{$\MN$}};
\node (D) at (-2,-4) {\phantom{\scalebox{2}{$\MN$}}};
\node (E) at (-6,-4) {\phantom{\scalebox{2}{$\MN$}}};
\node (F) at (-6,-8) {\phantom{\scalebox{2}{$\MN$}}};
\node (G) at (-10,-8) {\phantom{\scalebox{2}{$\MN$}}};
\node (K) at (6,-12) {\phantom{\scalebox{2}{$\MN$}}};
\node (L) at (10,-12) {\phantom{\scalebox{2}{$\MN$}}};
    \tikzset{mid arrow/.style={
        postaction={decorate,decoration={
            markings,
            mark=at position .5 with {\arrow[scale=3]{stealth}}
        }}
    }}
    \draw[-{Stealth[scale=3]}] (A) -- (C) node[draw = none, pos=0.4, above, yshift=10pt] {\scalebox{2}{\N}};
       \draw[-{Stealth[scale=3]}] (A) -- (E) node[draw = none, pos=0.4, above, yshift=10pt] {\scalebox{2}{\S}};
             \draw[-{Stealth[scale=3]}] (E) -- (F) node[draw = none, pos=0.3, right, yshift=10pt] {\scalebox{2}{\S}};
              \draw[-{Stealth[scale=3]}] (E) -- (G) node[draw = none, pos=0.2, above, yshift=10pt] {\scalebox{2}{\S}};
               \draw[-{Stealth[scale=3]}] (A) -- (D) node[draw = none, pos=0.5, left, yshift=10pt] {\scalebox{2}{\S}};
                  \draw[-{Stealth[scale=3]}] (C) -- (K) node[draw = none, pos=0.7, left, yshift=10pt] {\scalebox{2}{\S}};
                     \draw[-{Stealth[scale=3]}] (C) -- (L) node[draw = none, pos=0.8, above, yshift=10pt] {\scalebox{2}{\S}};
\end{tikzpicture}}
\end{minipage}
\begin{minipage}{0.49\textwidth}
   \centering
        \scalebox{0.25}{
\begin{tikzpicture}[every node/.style={circle, draw, inner sep=4pt, minimum size=6mm}]
\node (A) at (-2,0) {\scalebox{3}{$\ME$}};
\node (B) at (2,-4) {\phantom{\scalebox{2}{$\MN$}}};
\node (C) at (6,-8) {\scalebox{3}{$\MN$}};
\node (D) at (-2,-4) {\phantom{\scalebox{2}{$\MN$}}};
\node (E) at (-6,-4) {\phantom{\scalebox{2}{$\MN$}}};
\node (F) at (-6,-8) {\phantom{\scalebox{2}{$\MN$}}};
\node (G) at (-10,-8) {\phantom{\scalebox{2}{$\MN$}}};
\node (H) at (2,-8) {\phantom{\scalebox{2}{$\MN$}}};
\node (I) at (-2,-8) {\phantom{\scalebox{2}{$\MN$}}};
\node (J) at (-2,-12) {\phantom{\scalebox{2}{$\MN$}}};
\node (K) at (6,-12) {\phantom{\scalebox{2}{$\MN$}}};
\node (L) at (10,-12) {\phantom{\scalebox{2}{$\MN$}}};
    \tikzset{mid arrow/.style={
        postaction={decorate,decoration={
            markings,
            mark=at position .5 with {\arrow[scale=3]{stealth}}
        }}
    }}
    \draw[-{Stealth[scale=3]}] (A) -- (B) node[draw = none, pos=0.4, above, yshift=10pt] {\scalebox{2}{\N}};
    \draw[-{Stealth[scale=3]}] (B) -- (C) node[draw = none, pos=0.4, above, yshift=10pt] {\scalebox{2}{\N}};
       \draw[-{Stealth[scale=3]}] (A) -- (E) node[draw = none, pos=0.4, above, yshift=10pt] {\scalebox{2}{\S}};
             \draw[-{Stealth[scale=3]}] (E) -- (F) node[draw = none, pos=0.3, right, yshift=10pt] {\scalebox{2}{\S}};
              \draw[-{Stealth[scale=3]}] (E) -- (G) node[draw = none, pos=0.2, above, yshift=10pt] {\scalebox{2}{\S}};
               \draw[-{Stealth[scale=3]}] (A) -- (D) node[draw = none, pos=0.5, left, yshift=10pt] {\scalebox{2}{\S}};
               \draw[-{Stealth[scale=3]}] (B) -- (H) node[draw = none, pos=0.7, left, yshift=10pt] {\scalebox{2}{\S}};
                \draw[-{Stealth[scale=3]}] (B) -- (I) node[draw = none, pos=0.5, left, yshift=10pt] {\scalebox{2}{\S}};
                   \draw[-{Stealth[scale=3]}] (H) -- (J) node[draw = none, pos=0.6, above, yshift=10pt] {\scalebox{2}{\S}};
                    \draw[-{Stealth[scale=3]}] (C) -- (K) node[draw = none, pos=0.7, left, yshift=10pt] {\scalebox{2}{\S}};
                     \draw[-{Stealth[scale=3]}] (C) -- (L) node[draw = none, pos=0.8, above, yshift=10pt] {\scalebox{2}{\S}};
\end{tikzpicture}}
\end{minipage}
\end{figure}
Then attach a vertex, $\MAX$, which is maximal for both the west and the south poset
\begin{figure}[H]
\begin{minipage}{0.49\textwidth}
   \centering
        \scalebox{0.25}{
\begin{tikzpicture}[every node/.style={circle, draw, inner sep=4pt, minimum size=6mm}]
\node (X) at (8,4) {\scalebox{3}{$\MAX$}};
\node (A) at (-2,0) {\scalebox{3}{$\ME$}};
\node (C) at (6,-8) {\scalebox{3}{$\MN$}};
\node (D) at (-2,-4) {\phantom{\scalebox{2}{$\MN$}}};
\node (E) at (-6,-4) {\phantom{\scalebox{2}{$\MN$}}};
\node (F) at (-6,-8) {\phantom{\scalebox{2}{$\MN$}}};
\node (G) at (-10,-8) {\phantom{\scalebox{2}{$\MN$}}};
\node (K) at (6,-12) {\phantom{\scalebox{2}{$\MN$}}};
\node (L) at (10,-12) {\phantom{\scalebox{2}{$\MN$}}};
    \tikzset{mid arrow/.style={
        postaction={decorate,decoration={
            markings,
            mark=at position .5 with {\arrow[scale=3]{stealth}}
        }}
    }}
       \draw[-{Stealth[scale=3]}] (X) -- (A) node[draw = none, pos=0.4, above, yshift=10pt] {\scalebox{2}{\S}};
         \draw[-{Stealth[scale=3]}] (X) -- (C) node[draw = none, pos=0.4, right,xshift=3pt, yshift=10pt] {\scalebox{2}{\S}};
    \draw[-{Stealth[scale=3]}] (A) -- (C) node[draw = none, pos=0.4, right, yshift=10pt] {\scalebox{2}{\N}};
       \draw[-{Stealth[scale=3]}] (A) -- (E) node[draw = none, pos=0.4, above, yshift=10pt] {\scalebox{2}{\S}};
             \draw[-{Stealth[scale=3]}] (E) -- (F) node[draw = none, pos=0.3, right, yshift=10pt] {\scalebox{2}{\S}};
              \draw[-{Stealth[scale=3]}] (E) -- (G) node[draw = none, pos=0.2, above, yshift=10pt] {\scalebox{2}{\S}};
               \draw[-{Stealth[scale=3]}] (A) -- (D) node[draw = none, pos=0.5, left, yshift=10pt] {\scalebox{2}{\S}};
                  \draw[-{Stealth[scale=3]}] (C) -- (K) node[draw = none, pos=0.7, left, yshift=10pt] {\scalebox{2}{\S}};
                     \draw[-{Stealth[scale=3]}] (C) -- (L) node[draw = none, pos=0.8, above, yshift=10pt] {\scalebox{2}{\S}};
\end{tikzpicture}}
\end{minipage}
\begin{minipage}{0.49\textwidth}
   \centering
        \scalebox{0.25}{
\begin{tikzpicture}[every node/.style={circle, draw, inner sep=4pt, minimum size=6mm}]
\node (X) at (8,4) {\scalebox{3}{$\MAX$}};
\node (A) at (-2,0) {\scalebox{3}{$\ME$}};
\node (B) at (2,-4) {\phantom{\scalebox{2}{$\MN$}}};
\node (C) at (6,-8) {\scalebox{3}{$\MN$}};
\node (D) at (-2,-4) {\phantom{\scalebox{2}{$\MN$}}};
\node (E) at (-6,-4) {\phantom{\scalebox{2}{$\MN$}}};
\node (F) at (-6,-8) {\phantom{\scalebox{2}{$\MN$}}};
\node (G) at (-10,-8) {\phantom{\scalebox{2}{$\MN$}}};
\node (H) at (2,-8) {\phantom{\scalebox{2}{$\MN$}}};
\node (I) at (-2,-8) {\phantom{\scalebox{2}{$\MN$}}};
\node (J) at (-2,-12) {\phantom{\scalebox{2}{$\MN$}}};
\node (K) at (6,-12) {\phantom{\scalebox{2}{$\MN$}}};
\node (L) at (10,-12) {\phantom{\scalebox{2}{$\MN$}}};
    \tikzset{mid arrow/.style={
        postaction={decorate,decoration={
            markings,
            mark=at position .5 with {\arrow[scale=3]{stealth}}
        }}
    }}
        \draw[-{Stealth[scale=3]}] (X) -- (A) node[draw = none, pos=0.4, above, yshift=10pt] {\scalebox{2}{\S}};
        \draw[-{Stealth[scale=3]}] (X) -- (B) node[draw = none, pos=0.4, above, yshift=10pt] {\scalebox{2}{\S}};
         \draw[-{Stealth[scale=3]}] (X) -- (C) node[draw = none, pos=0.4, right,xshift=3pt, yshift=10pt] {\scalebox{2}{\S}};
    \draw[-{Stealth[scale=3]}] (A) -- (B) node[draw = none, pos=0.4, above, yshift=10pt] {\scalebox{2}{\N}};
    \draw[-{Stealth[scale=3]}] (B) -- (C) node[draw = none, pos=0.4, above, yshift=10pt] {\scalebox{2}{\N}};
       \draw[-{Stealth[scale=3]}] (A) -- (E) node[draw = none, pos=0.4, above, yshift=10pt] {\scalebox{2}{\S}};
             \draw[-{Stealth[scale=3]}] (E) -- (F) node[draw = none, pos=0.3, right, yshift=10pt] {\scalebox{2}{\S}};
              \draw[-{Stealth[scale=3]}] (E) -- (G) node[draw = none, pos=0.2, above, yshift=10pt] {\scalebox{2}{\S}};
               \draw[-{Stealth[scale=3]}] (A) -- (D) node[draw = none, pos=0.5, left, yshift=10pt] {\scalebox{2}{\S}};
               \draw[-{Stealth[scale=3]}] (B) -- (H) node[draw = none, pos=0.7, left, yshift=10pt] {\scalebox{2}{\S}};
                \draw[-{Stealth[scale=3]}] (B) -- (I) node[draw = none, pos=0.5, left, yshift=10pt] {\scalebox{2}{\S}};
                   \draw[-{Stealth[scale=3]}] (H) -- (J) node[draw = none, pos=0.6, above, yshift=10pt] {\scalebox{2}{\S}};
                    \draw[-{Stealth[scale=3]}] (C) -- (K) node[draw = none, pos=0.7, left, yshift=10pt] {\scalebox{2}{\S}};
                     \draw[-{Stealth[scale=3]}] (C) -- (L) node[draw = none, pos=0.8, above, yshift=10pt] {\scalebox{2}{\S}};
\end{tikzpicture}}
\end{minipage}
\end{figure}
We denote the set of equivalence classes of these double posets with $\ArboNE$. Observe that given $\arboNE \in \ArboNE$, we have $\arboNE \in \treedoubleposets$ but $\arboNE \not \in \twintreedoubleposets$.
\end{definition}

Notice that the examples above do not represent pairs of Hasse diagrams, but the respective transitive closures do represent double posets. The main result of this section is \Cref{theorem:gen3214}, i.e. counting $|\Hom(\arboNE,\SntoDP(\Pi))|$ is feasible in time $\tilde{\mathcal O}(n^{5/3})$ and space $\tilde\O(n)$.
We now describe some properties of the elements of $\ArboNE$. 
\begin{remark}
Let $\arboNE \in \ArboNE$. Consider $\arboNE|_{V(\arboNE)\setminus \{\MAX\}}\in \twintreedoubleposets$, i.e., the double poset obtained by removing $\MAX$ and its incident edges. Then, by construction, rooting it in $\ME$ yields
\begin{align*}
\TTDPtoCT(\arboNE|_{V(\arboNE)\setminus \{\MAX\}},\ME)\in \CornerTreesfivethree
\end{align*}
where the only edge labels are either $\SW$ or $\NW$. This property is essential for the algorithm that efficiently counts $|\Hom(\arboNE,\SntoDP(\Pi))|$.
\end{remark}

\begin{remark}
\label{remark:what_counts}
Let $\arboNE \in \ArboNE$. Now consider $\arboNE|_{V(\arboNE)\setminus \{\MAX\}}$, i.e. the twin tree double poset obtained by removing $\MAX$ and its incident edges. If $\arboNE|_{V(\arboNE)\setminus \{\MAX\}}$
counts the linear combination of patterns
\begin{align*}
\DPtoSn\left(\proj_{\mathcal{S}}\left(\Phi_{\regmono \leftarrow \hom}\left(\arboNE|_{V(\arboNE)\setminus \{\MAX\}}\right)\right)\right) &= \sum_{n \ge 0}\sum_{\sigma \in \permutations(n)}c_{\sigma}[\sigma(1)\cdots\sigma(n)],
\end{align*}
where $c_{\sigma} \in \NN_{\ge 1}$, then $\arboNE$ counts
\begin{align*}
\DPtoSn\left(\proj_{\mathcal{S}}\left(\Phi_{\regmono \leftarrow \hom}(\arboNE)\right)\right) &= \sum_{n \ge 0}\sum_{\sigma \in \permutations(n)}c_{\sigma}[\sigma(1)\cdots\sigma(n) \, n+1].
\end{align*}

As an example, given $\arboNE = \raisebox{-1.2cm}{\scalebox{0.2}{
\begin{tikzpicture}[every node/.style={circle, draw, inner sep=4pt, minimum size=6mm}]
\node (X) at (5,3) {\scalebox{3}{$\MAX$}};
\node (A) at (-2,0) {\scalebox{3}{$\ME$}};
\node (C) at (3,-4) {\scalebox{3}{$\MN$}};
\node (D) at (-2,-4) {\phantom{\scalebox{2}{$\MN$}}};
\node (K) at (3,-8) {\phantom{\scalebox{2}{$\MN$}}};
    \tikzset{mid arrow/.style={
        postaction={decorate,decoration={
            markings,
            mark=at position .5 with {\arrow[scale=3]{stealth}}
        }}
    }}
       \draw[-{Stealth[scale=3]}] (X) -- (A) node[draw = none, pos=0.4, above, yshift=10pt] {\scalebox{3}{\S}};
         \draw[-{Stealth[scale=3]}] (X) -- (C) node[draw = none, pos=0.4, right,xshift=3pt, yshift=10pt] {\scalebox{3}{\S}};
    \draw[-{Stealth[scale=3]}] (A) -- (C) node[draw = none, pos=0.4, right, yshift=10pt] {\scalebox{3}{\N}};
      \draw[-{Stealth[scale=3]}] (A) -- (D) node[draw = none, pos=0.5, left, yshift=10pt] {\scalebox{3}{\S}};
                  \draw[-{Stealth[scale=3]}] (C) -- (K) node[draw = none, pos=0.7, left, yshift=10pt] {\scalebox{3}{\S}};
                     \end{tikzpicture}}}$, we have
\begin{align*}
\DPtoSn\left(\proj_{\mathcal{S}}\left(\Phi_{\regmono \leftarrow \hom}(\arboNE)\right)\right) & = \perm{1324} + 2\perm{12435} + \perm{13425} + \perm{14235} \\&+ 2\perm{21435} + \perm{24135} +\perm{31425} + \perm{34125}.
\end{align*}
since         
\begin{align*}
\DPtoSn\left(\proj_{\mathcal{S}}\left( \Phi_{\regmono \leftarrow \hom}\left(\arboNE|_{V(\arboNE)\setminus \{\MAX\}}\right)\right)\right) & = \perm{132} + 2\perm{1243} + \perm{1342} + \perm{1423} \\&+ 2\perm{2143} + \perm{2413} +\perm{3142} + \perm{3412}.
\end{align*}
\end{remark}

\begin{remark}
\label{remark:swaptalgo}
Let $\smalldp := \triplesmalldp$ be a double poset. Define
\begin{align*}
\Swap(\smalldp) &:= (A,Q_{A},P_{A}).
\end{align*}
   If $\arboNE \in \ArboNE$, then $\Swap(\arboNE) \in \ArboNE$.

\begin{figure}[H]
\centering
\begin{minipage}{0.3\textwidth}
        \centering
      \scalebox{0.23}{\begin{tikzpicture}[every node/.style={circle, draw, inner sep=4pt, minimum size=6mm}]
\node (X) at (8,4) {\scalebox{3}{$\MAX$}};
\node (A) at (-2,0) {\scalebox{3}{$\ME$}};
\node (B) at (2,-4) {\phantom{\scalebox{2}{$\MN$}}};
\node (C) at (6,-8) {\scalebox{3}{$\MN$}};
\node (D) at (-2,-4) {\phantom{\scalebox{2}{$\MN$}}};
\node (E) at (-6,-4) {\phantom{\scalebox{2}{$\MN$}}};
\node (F) at (-6,-8) {\phantom{\scalebox{2}{$\MN$}}};
\node (G) at (-10,-8) {\phantom{\scalebox{2}{$\MN$}}};
\node (H) at (2,-8) {\phantom{\scalebox{2}{$\MN$}}};
\node (I) at (-2,-8) {\phantom{\scalebox{2}{$\MN$}}};
\node (J) at (-2,-12) {\phantom{\scalebox{2}{$\MN$}}};
\node (K) at (6,-12) {\phantom{\scalebox{2}{$\MN$}}};
\node (L) at (10,-12) {\phantom{\scalebox{2}{$\MN$}}};
\node[draw=none] (K1) at (6,-16) {\phantom{\scalebox{2}{$\MN$}}};
\node[draw=none] (L1) at (10,-16) {\phantom{\scalebox{2}{$\MN$}}};
    \tikzset{mid arrow/.style={
        postaction={decorate,decoration={
            markings,
            mark=at position .5 with {\arrow[scale=3]{stealth}}
        }}
    }}
        \draw[-{Stealth[scale=3]}] (X) -- (A) node[draw = none, pos=0.4, above, yshift=10pt] {\scalebox{2}{\S}};
        \draw[-{Stealth[scale=3]}] (X) -- (B) node[draw = none, pos=0.4, above, yshift=10pt] {\scalebox{2}{\S}};
         \draw[-{Stealth[scale=3]}] (X) -- (C) node[draw = none, pos=0.4, right,xshift=3pt, yshift=10pt] {\scalebox{2}{\S}};
    \draw[-{Stealth[scale=3]}] (A) -- (B) node[draw = none, pos=0.4, above, yshift=10pt] {\scalebox{2}{\N}};
    \draw[-{Stealth[scale=3]}] (B) -- (C) node[draw = none, pos=0.4, above, yshift=10pt] {\scalebox{2}{\N}};
       \draw[-{Stealth[scale=3]}] (A) -- (E) node[draw = none, pos=0.4, above, yshift=10pt] {\scalebox{2}{\S}};
             \draw[-{Stealth[scale=3]}] (E) -- (F) node[draw = none, pos=0.3, right, yshift=10pt] {\scalebox{2}{\S}};
              \draw[-{Stealth[scale=3]}] (E) -- (G) node[draw = none, pos=0.2, above, yshift=10pt] {\scalebox{2}{\S}};
               \draw[-{Stealth[scale=3]}] (A) -- (D) node[draw = none, pos=0.5, left, yshift=10pt] {\scalebox{2}{\S}};
               \draw[-{Stealth[scale=3]}] (B) -- (H) node[draw = none, pos=0.7, left, yshift=10pt] {\scalebox{2}{\S}};
                \draw[-{Stealth[scale=3]}] (B) -- (I) node[draw = none, pos=0.5, left, yshift=10pt] {\scalebox{2}{\S}};
                   \draw[-{Stealth[scale=3]}] (H) -- (J) node[draw = none, pos=0.6, above, yshift=10pt] {\scalebox{2}{\S}};
                    \draw[-{Stealth[scale=3]}] (C) -- (K) node[draw = none, pos=0.7, left, yshift=10pt] {\scalebox{2}{\S}};
                     \draw[-{Stealth[scale=3]}] (C) -- (L) node[draw = none, pos=0.8, above, yshift=10pt] {\scalebox{2}{\S}};
\end{tikzpicture}}
         \caption*{$\arboNE$}
    \end{minipage}\hfill
   \begin{minipage}{0.3\textwidth}
        \centering
      \scalebox{0.23}{\begin{tikzpicture}[every node/.style={circle, draw, inner sep=4pt, minimum size=6mm}]
\node (X) at (8,4) {\scalebox{3}{$\MAX$}};
\node (A) at (-2,0) {\scalebox{3}{$\ME$}};
\node (B) at (2,-4) {\phantom{\scalebox{2}{$\MN$}}};
\node (C) at (6,-8) {\scalebox{3}{$\MN$}};
\node (D) at (-2,-4) {\phantom{\scalebox{2}{$\MN$}}};
\node (E) at (-6,-4) {\phantom{\scalebox{2}{$\MN$}}};
\node (F) at (-6,-8) {\phantom{\scalebox{2}{$\MN$}}};
\node (G) at (-10,-8) {\phantom{\scalebox{2}{$\MN$}}};
\node (H) at (2,-8) {\phantom{\scalebox{2}{$\MN$}}};
\node (I) at (-2,-8) {\phantom{\scalebox{2}{$\MN$}}};
\node (J) at (-2,-12) {\phantom{\scalebox{2}{$\MN$}}};
\node (K) at (6,-12) {\phantom{\scalebox{2}{$\MN$}}};
\node (L) at (10,-12) {\phantom{\scalebox{2}{$\MN$}}};
\node[draw=none] (K1) at (6,-16) {\phantom{\scalebox{2}{$\MN$}}};
\node[draw=none] (L1) at (10,-16) {\phantom{\scalebox{2}{$\MN$}}};
    \tikzset{mid arrow/.style={
        postaction={decorate,decoration={
            markings,
            mark=at position .5 with {\arrow[scale=3]{stealth}}
        }}
    }}
        \draw[-{Stealth[scale=3]}] (X) -- (A) node[draw = none, pos=0.4, above, yshift=10pt] {\scalebox{2}{\S}};
        \draw[-{Stealth[scale=3]}] (X) -- (B) node[draw = none, pos=0.4, above, yshift=10pt] {\scalebox{2}{\S}};
         \draw[-{Stealth[scale=3]}] (X) -- (C) node[draw = none, pos=0.4, right,xshift=3pt, yshift=10pt] {\scalebox{2}{\S}};
    \draw[-{Stealth[scale=3]}] (B) -- (A) node[draw = none, pos=0.4, above, yshift=10pt] {\scalebox{2}{\N}};
    \draw[-{Stealth[scale=3]}] (C) -- (B) node[draw = none, pos=0.4, above, yshift=10pt] {\scalebox{2}{\N}};
       \draw[-{Stealth[scale=3]}] (A) -- (E) node[draw = none, pos=0.4, above, yshift=10pt] {\scalebox{2}{\S}};
             \draw[-{Stealth[scale=3]}] (E) -- (F) node[draw = none, pos=0.3, right, yshift=10pt] {\scalebox{2}{\S}};
              \draw[-{Stealth[scale=3]}] (E) -- (G) node[draw = none, pos=0.2, above, yshift=10pt] {\scalebox{2}{\S}};
               \draw[-{Stealth[scale=3]}] (A) -- (D) node[draw = none, pos=0.5, left, yshift=10pt] {\scalebox{2}{\S}};
               \draw[-{Stealth[scale=3]}] (B) -- (H) node[draw = none, pos=0.7, left, yshift=10pt] {\scalebox{2}{\S}};
                \draw[-{Stealth[scale=3]}] (B) -- (I) node[draw = none, pos=0.5, left, yshift=10pt] {\scalebox{2}{\S}};
                   \draw[-{Stealth[scale=3]}] (H) -- (J) node[draw = none, pos=0.6, above, yshift=10pt] {\scalebox{2}{\S}};
                    \draw[-{Stealth[scale=3]}] (C) -- (K) node[draw = none, pos=0.7, left, yshift=10pt] {\scalebox{2}{\S}};
                     \draw[-{Stealth[scale=3]}] (C) -- (L) node[draw = none, pos=0.8, above, yshift=10pt] {\scalebox{2}{\S}};
\end{tikzpicture}}
         \caption*{$\Swap(\arboNE)$}
    \end{minipage}\hfill
\begin{minipage}{0.3\textwidth}
        \centering
      \scalebox{0.23}{\begin{tikzpicture}[every node/.style={circle, draw, inner sep=4pt, minimum size=6mm}]
\node (X) at (8,4) {\scalebox{3}{$\MAX$}};
\node (A) at (-2,0) {\scalebox{3}{$\ME$}};
\node (B) at (2,-4) {\phantom{\scalebox{2}{$\MN$}}};
\node (C) at (6,-8) {\scalebox{3}{$\MN$}};
\node (D) at (-2,-4) {\phantom{\scalebox{2}{$\MN$}}};
\node (E) at (-6,-4) {\phantom{\scalebox{2}{$\MN$}}};
\node (H) at (2,-8) {\phantom{\scalebox{2}{$\MN$}}};
\node (I) at (-2,-8) {\phantom{\scalebox{2}{$\MN$}}};
\node (J) at (-2,-12) {\phantom{\scalebox{2}{$\MN$}}};
\node (K) at (6,-12) {\phantom{\scalebox{2}{$\MN$}}};
\node (L) at (10,-12) {\phantom{\scalebox{2}{$\MN$}}};
\node (K1) at (6,-16) {\phantom{\scalebox{2}{$\MN$}}};
\node (L1) at (10,-16) {\phantom{\scalebox{2}{$\MN$}}};
    \tikzset{mid arrow/.style={
        postaction={decorate,decoration={
            markings,
            mark=at position .5 with {\arrow[scale=3]{stealth}}
        }}
    }}
        \draw[-{Stealth[scale=3]}] (X) -- (A) node[draw = none, pos=0.4, above, yshift=10pt] {\scalebox{2}{\S}};
        \draw[-{Stealth[scale=3]}] (X) -- (B) node[draw = none, pos=0.4, above, yshift=10pt] {\scalebox{2}{\S}};
         \draw[-{Stealth[scale=3]}] (X) -- (C) node[draw = none, pos=0.4, right,xshift=3pt, yshift=10pt] {\scalebox{2}{\S}};
    \draw[-{Stealth[scale=3]}] (A) -- (B) node[draw = none, pos=0.4, above, yshift=10pt] {\scalebox{2}{\N}};
    \draw[-{Stealth[scale=3]}] (B) -- (C) node[draw = none, pos=0.4, above, yshift=10pt] {\scalebox{2}{\N}};
       \draw[-{Stealth[scale=3]}] (A) -- (E) node[draw = none, pos=0.4, above, yshift=10pt] {\scalebox{2}{\S}};
               \draw[-{Stealth[scale=3]}] (A) -- (D) node[draw = none, pos=0.5, left, yshift=10pt] {\scalebox{2}{\S}};
               \draw[-{Stealth[scale=3]}] (B) -- (H) node[draw = none, pos=0.7, left, yshift=10pt] {\scalebox{2}{\S}};
                \draw[-{Stealth[scale=3]}] (B) -- (I) node[draw = none, pos=0.5, left, yshift=10pt] {\scalebox{2}{\S}};
                   \draw[-{Stealth[scale=3]}] (H) -- (J) node[draw = none, pos=0.6, above, yshift=10pt] {\scalebox{2}{\S}};
                    \draw[-{Stealth[scale=3]}] (C) -- (K) node[draw = none, pos=0.7, left, yshift=10pt] {\scalebox{2}{\S}};
                     \draw[-{Stealth[scale=3]}] (C) -- (L) node[draw = none, pos=0.8, above, yshift=10pt] {\scalebox{2}{\S}};
                     \draw[-{Stealth[scale=3]}] (K) -- (K1) node[draw = none, pos=0.8, left, yshift=10pt] {\scalebox{2}{\S}};
                     \draw[-{Stealth[scale=3]}] (K) -- (L1) node[draw = none, pos=0.8, above, yshift=10pt] {\scalebox{2}{\S}};
\end{tikzpicture}}
         \caption*{$\arboNEprime$}
    \end{minipage}
\end{figure}

We denote with $\arboNEprime$ the relabeled version of $\Swap(\arboNE)$ where we exchange the labels $\ME$ and $\MN$ so that the nodes $\ME$ and $\MN$ are again the most east and most north nodes respectively. This property is also essential for the algorithm that efficiently counts $|\Hom(\arboNE,\SntoDP(\Pi))|$.\end{remark}

To count occurrences of $[\mathtt{3\,2\,1\,4}]$ in the algorithm presented in \cite[Section 4]{even2021counting}, the authors distinguish between three types of occurrences. Here we rephrase them in our setting.

\begin{remark}[Morphism types]
\label{remark:algo_cases}
Let $\arboNE \in \ArboNE$ and $\Pi \in \permutations(n)$. Fix an integer $m < n$ and consider the interval partition of $[n]$, $\{\intervalpartI_{i}|i \in I \}$, according to the total order $1 <\cdots <n$ where all blocks $I_{1}<\dots<I_{\lceil n/m \rceil}$ have size $m$ (expect the very last one if $m$ does note exactly divides $n$).
Consider also the interval partition of $[n]$ according to the order $\Pi^{-1}(1) <\cdots <\Pi^{-1}(n)$ induced by $\intervalpartI$ and $\Pi$, given by
\begin{align*}
    \{\intervalpartJ_{j}|j\in J \}:= \{\Pi(\intervalpartI_{i})|i\in I \}
\end{align*}
Consider
the following subsets of $\Hom(\arboNE, \SntoDP(\Pi))$,
\begin{align*}
    &\TypeA(\arboNE, \SntoDP(\Pi)) :=\{f|\;f(\MN) \in \intervalpartJ_{j} \land f(\MAX) \in \intervalpartJ_{j^{\prime}}\;\textup{and}\;j \neq j^{\prime}\}\\
     &\TypeB(\arboNE, \SntoDP(\Pi)) :=\{f|\;f(\ME) \in \intervalpartI_{i} \land f(\MN) \in \intervalpartI_{i^{\prime}}\;\textup{and}\;i \neq i^{\prime}\}\\
    &\TypeBnotA(\arboNE, \SntoDP(\Pi)) :=\{f \in \TypeB|\;f(\MN) \in \intervalpartJ_{j} \land f(\MAX) \in \intervalpartJ_{j^{\prime}}\;\textup{and}\;j = j^{\prime}\}\\
      &\TypeAnotB(\arboNE, \SntoDP(\Pi)) :=\{f \in \TypeA|\;f(\ME) \in \intervalpartI_{i} \land f(\MAX) \in \intervalpartI_{i^{\prime}}\;\textup{and}\;i = i^{\prime}\}\\
    &\TypenotAnotB(\arboNE, \SntoDP(\Pi)) :=\{f|\;f(\ME) \in \intervalpartI_{i} \land f(\MAX) \in \intervalpartI_{i^{\prime}}\;\textup{and}\;i = i^{\prime},\\&\quad \quad \quad \quad \quad \quad \quad \quad \quad \quad \quad \quad \quad \quad \quad \quad \quad \;\, f(\MN) \in \intervalpartJ_{j} \land f(\MAX) \in \intervalpartJ_{j^{\prime}}\;\textup{and}\;j = j^{\prime}\}.
\end{align*}
Then
\begin{align*}
    &\Hom(\arboNE, \SntoDP(\Pi))\\
&=\TypeA(\arboNE, \SntoDP(\Pi)) \sqcup \TypeBnotA(\arboNE, \SntoDP(\Pi)) \sqcup \TypenotAnotB(\arboNE, \SntoDP(\Pi))\\
&=\TypeB(\arboNE, \SntoDP(\Pi)) \sqcup \TypeAnotB(\arboNE, \SntoDP(\Pi)) \sqcup \TypenotAnotB(\arboNE, \SntoDP(\Pi))
\end{align*}
Furthermore we have
\begin{align*}
&|\TypeB(\arboNE, \SntoDP(\Pi))|\\
&|\TypeB(\Swap(\arboNE), \Swap(\SntoDP(\Pi)))|\\
&=|\TypeA(\arboNEprime, \SntoDP(\Pi^{-1}))|,\\
\end{align*}
and similarly
\begin{align*}
&|\TypeBnotA(\arboNE, \SntoDP(\Pi))|\\
&=|\TypeBnotA(\Swap(\arboNE), \Swap(\SntoDP(\Pi)))|\\
&=|\TypeAnotB(\arboNEprime, \SntoDP(\Pi^{-1}))|.
\end{align*}
The previous considerations allow us to write an algorithm that counts
\begin{align*}
     |\Hom(\arboNE, \SntoDP(\Pi))| &= |\TypeA(\arboNE, \SntoDP(\Pi)) |+|\TypeAnotB(\arboNEprime, \SntoDP(\Pi^{-1}))|\\&+|\TypenotAnotB(\arboNE, \SntoDP(\Pi))|.
\end{align*}
\end{remark}

\begin{figure}[H]
    \centering
    \begin{minipage}{0.33\textwidth}
        \centering
        \begin{tikzpicture}[scale=0.4]
            % Draw x and y axes with labels
            \draw[->] (0,0) -- (11,0) node[right] {};
            \draw[->] (0,0) -- (0,11) node[above] {};
            % Draw grid lines
            \foreach \x in {2.5,4.5,...,10.5} {
                \draw[dotted] (\x,0) -- (\x,10.5);
                \draw[dotted] (0,\x) -- (10.5,\x);
            }
            % Draw axis labels without ticks
            \foreach \x in {1,2,...,10} {
                \node at (\x,-0.5) {\x};
          %      \node at (-0.5,\x) {\x};
            }

            % Draw permutation points
            \foreach \x/\p in {1/7, 2/10, 3/3, 4/1, 5/9, 6/8, 7/5, 8/4, 9/2, 10/6} {
                \node[circle, draw, fill, inner sep=1pt] (\p) at (\x,\p) {};
            }
           
            \node (17) at (1,7) {\;\;\;\;\;\scalebox{0.7}{$\MN$}};;
            \node (17) at (4,1) {\;\;\;\;\;\scalebox{0.7}{$\ME$}};
            \node (17) at (5.1,9.2) {\;\;\;\;\;\scalebox{0.7}{ $\MAX$}};
        \end{tikzpicture}
        \caption*{\TypeA}
    \end{minipage}\hfill
    \begin{minipage}{0.33\textwidth}
        \centering
        \begin{tikzpicture}[scale=0.4]
            % Draw x and y axes with labels
            \draw[->] (0,0) -- (11,0) node[right] {};
            \draw[->] (0,0) -- (0,11) node[above] {};
            % Draw grid lines
            \foreach \x in {2.5,4.5,...,10.5} {
                \draw[dotted] (\x,0) -- (\x,10.5);
                \draw[dotted] (0,\x) -- (10.5,\x);
            }
            % Draw axis labels without ticks
            \foreach \x in {1,2,...,10} {
                \node at (\x,-0.5) {\x};
          %      \node at (-0.5,\x) {\x};
            }

            % Draw permutation points
            \foreach \x/\p in {1/7, 2/10, 3/3, 4/1, 5/9, 6/8, 7/5, 8/4, 9/2, 10/6} {
                \node[circle, draw, fill, inner sep=1pt] (\p) at (\x,\p) {};
            }
           
            \node (17) at (1,7) {\;\;\;\;\;\scalebox{0.7}{$\MN$}};
            \node (17) at (3,3) {};
            \node (17) at (4,1) {\;\;\;\;\;\scalebox{0.7}{$\ME$}};
            \node (17) at (6,8) {\;\;\;\;\;\scalebox{0.7}{ $\MAX$}};
        \end{tikzpicture}
        \caption*{\TypeBnotA}\end{minipage}\hfill
        \begin{minipage}{0.33\textwidth}
        \centering
        \begin{tikzpicture}[scale=0.4]
            % Draw x and y axes with labels
            \draw[->] (0,0) -- (11,0) node[right] {};
            \draw[->] (0,0) -- (0,11) node[above] {};
            % Draw grid lines
            \foreach \x in {2.5,4.5,...,10.5} {
                \draw[dotted] (\x,0) -- (\x,10.5);
                \draw[dotted] (0,\x) -- (10.5,\x);
            }
            % Draw axis labels without ticks
            \foreach \x in {1,2,...,10} {
                \node at (\x,-0.5) {\x};
              %  \node at (-0.5,\x) {\x};
            }

            % Draw permutation points
            \foreach \x/\p in {1/7, 2/10, 3/3, 4/1, 5/9, 6/8, 7/5, 8/4, 9/2, 10/6} {
                \node[circle, draw, fill, inner sep=1pt] (\p) at (\x,\p) {};
            }
            % Draw lines between points
           
            \node (17) at (7,5) {\;\;\;\;\;\scalebox{0.7}{$\MN$}};
            \node (17) at (8,4){};
            \node (17) at (9,2) {\;\;\;\;\;\scalebox{0.7}{$\ME$}};
            \node (17) at (8.7,6.6) {\;\;\;\;\;\scalebox{0.7}{ $\MAX$}};
        \end{tikzpicture}
        \caption*{\TypenotAnotB}
    \end{minipage}
    \caption*{Example of the three cases}
\end{figure}

\begin{figure}[H]
    \centering
    \makebox[\textwidth]{
        \begin{minipage}{0.5\textwidth}
        \centering
            \begin{tikzpicture}[scale=0.4]
                % Draw x and y axes with labels
                \draw[->] (0,0) -- (11,0) node[right] {};
                \draw[->] (0,0) -- (0,11) node[above] {};
                % Draw grid lines
                \foreach \x in {2.5,4.5,...,10.5} {
                    \draw[dotted] (\x,0) -- (\x,10.5);
                    \draw[dotted] (0,\x) -- (10.5,\x);
                }
                % Draw axis labels without ticks
                \foreach \x in {1,2,...,10} {
                    \node at (\x,-0.5) {\x};
                }

                % Draw permutation points
                \foreach \x/\p in {1/7, 2/10, 3/3, 4/1, 5/9, 6/8, 7/5, 8/4, 9/2, 10/6} {
                    \node[circle, draw, fill, inner sep=1pt] (\p) at (\x,\p) {};
                }
                
                \node (17) at (1,7) {\;\;\;\;\;\scalebox{0.7}{$\textcolor{red}{\MN}$}};
                \node (17) at (3,3) {};
                \node (17) at (4,1) {\;\;\;\;\;\scalebox{0.7}{$\textcolor{green}{\ME}$}};
                \node (17) at (6,8) {\;\;\;\;\;\scalebox{0.7}{ $\textcolor{blue}{\MAX}$}};
            \end{tikzpicture}
    \caption*{$f \in \TypeBnotA(\SntoDP(\arboNE),\SntoDP(\Pi))$}
          \end{minipage}%
        \begin{minipage}{0.5\textwidth}
        \centering
            \begin{tikzpicture}[scale=0.4]
                % Draw x and y axes with labels
                \draw[->] (0,0) -- (11,0) node[right] {};
                \draw[->] (0,0) -- (0,11) node[above] {};
                % Draw grid lines
                \foreach \x in {2.5,4.5,...,10.5} {
                    \draw[dotted] (\x,0) -- (\x,10.5);
                    \draw[dotted] (0,\x) -- (10.5,\x);
                }
                % Draw axis labels without ticks
                \foreach \x in {1,2,...,10} {
                    \node at (\x,-0.5) {\x};
            %        \node at (-0.5,\x) {\x};
                }

                % Draw permutation points
                \foreach \x/\p in {1/4, 2/9, 3/3, 4/8, 5/7, 6/10, 7/1, 8/6, 9/5, 10/2} {
                    \node[circle, draw, fill, inner sep=1pt] (\p) at (\x,\p) {};
                }
                % Draw lines between points
               
                \node (17) at (7,1) {\;\;\;\;\;\scalebox{0.7}{$\textcolor{green}{\ME}$}};
                \node (17) at (3,3) {};
                \node (17) at (1,4) {\;\;\;\;\;\scalebox{0.7}{$\textcolor{red}{\MN}$}};
                \node (17) at (8,6) {\;\;\;\;\;\;\scalebox{0.7}{$\textcolor{blue}{\MAX}$}};
            \end{tikzpicture}\caption*{$\xi(f)\in \TypeAnotB(\SntoDP(\arboNEprime),\SntoDP(\Pi^{-1}))$}
          \end{minipage}
    }
    \caption{One to one correspondence}
    \label{figure:one_to_one}
\end{figure}
 In \Cref{figure:one_to_one}, given $\arboNE=\SntoDP([\mathtt{2}\,\mathtt{1}\,\mathtt{3}])$ and $\Pi = [\mathtt{7}\text{-}\mathtt{10}\text{-}\mathtt{3}\text{-}\mathtt{1}\text{-}\mathtt{9}\text{-}\mathtt{8}\text{-}\mathtt{5}\text{-}\mathtt{4}\text{-}\mathtt{2}\text{-}\mathtt{6}]$ an example of the bijection \begin{align*}
     \xi: \TypeBnotA(\SntoDP(\arboNE),\SntoDP(\Pi))&\to \TypeAnotB(\SntoDP(\arboNEprime),\SntoDP(\Pi^{-1}))\end{align*} is shown where the points are labeled with the respective preimages.
 
\subsection{The algorithm for counting corner tree occurrences}
\label{subsection:algo_occurrences}

We can compute the occurrences of a corner tree on a permutation recursively (from the leaves to the root) relying on a \quot{vertex} and an \quot{edge} function.
\begin{definition}[Vertex and Edge functions]
Let $\Pi \in \permutations(n), \CT \in \CornerTrees, a \in \vertexset(\CT)$. Then, define
\begin{align*}
    \NN^{n} \ni \textup{vertex}(\Pi, \CT, a) := \bigodot_{\substack{e \in E(\CT):\\\\e \,=\,(a,b)}} \textup{edge}(\Pi, \CT, (a,b))
\end{align*}
where $\bigodot$ is the entry-wise product for tuples, and $\textup{edge}(\Pi, \CT, (a,b)) \in  \NN^{n}$ for $ i \in \{1,...,n\}$ is defined as
\begin{align*}
    \textup{edge}(\Pi, \CT, (a,b))[i] &:= \sum_{g(i,j)} \textup{vertex}(\Pi, \CT, b)[j]
\end{align*}
where $g(i, j)$ is defined as 
\begin{itemize}
    \item $g(i, j) = j < i \land \pi(j) < \pi(i)$, if $\edgelabel((a, b)) = \SW$,
    \item $g(i, j) = j < i \land \pi(j) > \pi(i)$, if $\edgelabel((a, b)) = \NW$,
    \item $g(i, j) = j > i \land \pi(j) < \pi(i)$, if $\edgelabel((a, b)) = \SE$,
    \item $g(i, j) = j > i \land \pi(j) > \pi(i)$, if $\edgelabel((a, b)) = \NE$.
\end{itemize}
\end{definition}
Clearly, for a vertex $v \in \vertexset(\CT)$ we have
\begin{align*}
        &\textup{vertex}(\Pi, \CT, v)[i] \\&=     |\{f|f:\vertexset(\CT)\to [n]:\,\,f\,\,\text{is an occurrence of the subtree of $\CT$ rooted at $v$ on}\;\Pi,\;f(v)=i\}|.
\end{align*}
In particular, we have
\begin{align*}
        &\textup{vertex}(\Pi, \CT, \text{root})[i] \\&=     |\{f|f:\vertexset(\CT)\to [n]:\,\,f\,\,\text{is an occurrence  of $\CT$ on}\;\Pi,\;f(\text{root})=i\}|.
\end{align*}
We now recall Theorem 1.1 from \cite{even2021counting}.
\begin{theorem}[\cite{even2021counting}]
\label{theorem:zohar_compl}
Let $\Pi \in \permutations(n)$ and let $\CT\in \CornerTrees$. Then, using \Cref{algorithm_occurrences}, we can count
\begin{align*}
    |\{f|f:\vertexset(\CT)\to [n]:\,\,f\,\,\text{is an occurrence of }\,\CT\,\,\text{on}\,\,\Pi\}|
\end{align*}
in $\tilde{\mathcal O}(n)$ time and $\tilde{\mathcal O}(n)$ space.
\end{theorem}

\begin{algorithm}[H]
\caption{Counting corner tree occurrences \parencite{even2021counting}}
\label{algorithm_occurrences}
\begin{spacing}{1.2}
\begin{algorithmic}[1]
\algfont
\Function{VERTEX}{permutation $\Pi$, vertex $v$} 
    \Comment{{\footnotesize $v$ is the subtree of $\CT$ rooted at $v$}}
      \State $\text{X} \gets \algfont{\textup{ARRAY}}(1, \ldots, 1)$ \Comment{{\footnotesize of size $n = |\,\Pi\,|$}}
        \For{$e \in \text{child-edges}(v)$}
            \State $\text{X}\gets \text{X} \odot \Call{EDGE}{\Pi,e}$ \Comment{{\footnotesize term-wise product of arrays}}
        \EndFor
        \State \Return \algfont{\textup{X}}
    \EndFunction
\Function{EDGE}{permutation $\Pi$, edge $e$}
    
    \State $\text{X}\gets  \Call{VERTEX}{\Pi, \text{child}(e)}$
    \State $\text{Z} \gets \text{ARRAY}(0, \ldots, 0)$ \Comment{{\footnotesize of size $n$}}
    \State $\text{Y} \gets \text{SUM-TREE}(0, \ldots, 0)$ \Comment{{\footnotesize of size $n$}}

    \If{$\edgelabel(e) \in \{\NW, \SW\}$}
        \State $\text{order} \gets [1, \ldots, n]$
    \Else
        \State $\text{order} \gets [n, \ldots, 1]$
    \EndIf

    \For{$i \in \text{order}$}
        \State $\text{Y}[\Pi[i]] \gets \text{X}[i]$
        \State $\text{Z}[i] \gets \begin{cases}
            \text{Y.prefix\_sum}(\Pi[i]), & \text{if } \text{label}(e) \in \{\SE, \SW\} \\
            \text{Y.suffix\_sum}(\Pi[i]), & \text{otherwise}
        \end{cases}$
    \EndFor
    \State \Return Z
  \EndFunction
\end{algorithmic}
\end{spacing}
\end{algorithm}
\begin{remark}
Observe that summing the entries of the tuple {\algfont VERTEX(permutation $\Pi$, root)} will yield the number of occurences. 
In \Cref{algorithm_occurrences}, tuples of lenght $n$, for which we compute strict prefix and suffix sums, are stored as leaves of perfect binary trees\footnote{A perfect binary tree is a binary tree where every internal node has exactly two children, and all leaf nodes are at the same depth or level.} where each parent node stores the sum of its two children. The tree has depth $\lceil\log_{2}(n)\rceil+1$. Updating an entry of a tuple stored this way costs $\lceil\log_{2}(n)\rceil+1$ since it affects the corresponding leaf and its ancestors. To compute a prefix sum at the $i$-th leaf, we look at its left sibling and its ancestors' left siblings (except for the root). This involves $\lceil\log_{2}(n)\rceil$ operations. Similarly, $\lceil\log_{2}(n)\rceil$ operations yield an efficient suffix sum. We refer to the class {\algfont{SumTree}} in \cite{diehl_verri_generalization_2025}. The use of this data structure guarantees a time complexity of $\mathcal{O}(n\log(n))$.
\begin{figure}[H]
\begin{minipage}{0.45\textwidth}
\centering
\scalebox{0.7}{\begin{forest}
  [30
    [14
      [9  [4]
      [5]]
      [5  [2]
      [3]]
    ]
    [16
      [8 [1]
      [7]]
      [8  [2]
      [6]]
    ]
  ]
\end{forest}}
    \end{minipage}
     $\longrightarrow$
 \hspace*{1cm}   \begin{minipage}{0.45\textwidth}
\scalebox{0.7}{\begin{forest}
  [38
    [14
      [9  [4]
      [5]]
      [5  [2]
      [3]]
    ]
    [24
      [8 [1]
      [7]]
      [16  [10]
      [6]]
    ]
  ]
\end{forest}}
\end{minipage}
\caption{Update of an array stored as a sum-tree}
\label{fig:sum_tree}
\end{figure}
\begin{figure}[H]
\centering
\scalebox{0.7}{\begin{forest}
   [30
    [\textcolor{red}{14}
      [9  [4]
      [5]]
      [5  [2]
      [3]]
    ]
    [16
      [\textcolor{red}{8} [1]
      [7]]
      [8  [\textcolor{red}{2}]
      [6]]
    ]
  ]
\end{forest}}
\caption{Strict prefix sum at $i = 8$ equals $24$}
\label{fig:sum_tree}
\end{figure}
\end{remark}

\begin{example}
Let
\begin{align*}
    \CT&=\raisebox{-1.5em}{\scalebox{0.4}{
\begin{forest}
  for tree={
    circle,
    draw,
    minimum size=1.5cm,
    edge={-},
    s sep=20mm,
    l sep=15mm,
  },
  [ \scalebox{2}{a}
    [\scalebox{2}{b}, edge label={node[midway,above,sloped]{{\Large \NW}}}  ]
   ]
\end{forest}}},\quad \Pi = \perm{34251}.
\end{align*}

Below, we illustrate the loop that is run when evaluating $\textup{edge}(\Pi, \CT, (a,b))$. Since the label is $\NW$, we only perform suffix sums. The columns of the tables on the right are the arrays $Y$ and $Z$ at each iteration, from left to right.
The respective column header indicates which entry of the array is updated.
Zeros are omitted for readability.
\begin{figure}[H]
\hspace{-2cm}
    \begin{minipage}{0.5\textwidth}
    \centering
        \begin{tikzpicture}[scale=0.5]
            % Add ticks to x axis with custom labels
         \node at (1, -0.5) {1};
         \node at (2, -0.5) {2};
        \node at (3, -0.5) {3};
      \node at (4, -0.5) {4};
       \node at (5, -0.5) {5};
        
        % Add ticks to y axis with custom labels
      \node at (-2.5, 1) {$\Pi(5)=1$};
     \node at (-2.5, 2) {$\Pi(3)=2$};
     \node at (-2.5, 3) {$\Pi(1)=3$};
      \node at (-2.5, 4) {$\Pi(2)=4$};
     \node at (-2.5, 5) {$\Pi(4)=5$};

            % Draw permutation points
            \foreach \x/\p in {1/3, 2/4, 3/2, 4/5, 5/1} {
                \fill[black] (\x,\p) circle (4pt);
            }
            
            % Draw lines between points
           
        \end{tikzpicture}
    \end{minipage}
    \begin{minipage}{0.5\textwidth}
    \centering
        \begin{tabular}{llllll}
            & $Y[\Pi(1)]$ & $Y[\Pi(2)]$ & $Y[\Pi(3)]$ & $Y[\Pi(4)]$ & $Y[\Pi(5)]$ \\
           \gray{5}    &                     &  &  &  1 & 1  \\
            \gray4    &                     & 1 & 1 & 1 & 1  \\
            \gray3   &           1          & 1 & 1 & 1 & 1  \\
            \gray2    &                     &  & 1 & 1 & 1  \\
           \gray1    &                     &  &  &  &  1 \\
            & $Z[1]$ & $Z[2]$ & $Z[3]$  & $Z[4]$  &  $Z[5]$  \\
            \gray5    &                     &  &  &  &   4\\
            \gray4    &                     &  &  &  &   \\
            \gray3    &                     &  & 2 & 2 &  2  \\
            \gray2    &                     &  &  &  &   \\
            \gray1    &                     &  &  &  &   
        \end{tabular}
    \end{minipage}
\end{figure}

\end{example}
\begin{example}
Let
 \begin{align*}
\CT =\scalebox{0.4}{\raisebox{-1cm}{\begin{forest}
  for tree={
    circle,
    draw,
    minimum size=1.5cm,
    edge={-},
    s sep=20mm,
    l sep=15mm,
  },
  [\scalebox{2}{r} [\scalebox{2}{a}, edge label={node[midway,above,sloped]{\Large\SE}}[\scalebox{2}{b}, edge label={node[midway,above,sloped]{\Large\NE}}]
      [\scalebox{2}{c}, edge label={node[midway,above,sloped]{\Large\NW}}]]]
\end{forest}}},\quad \Pi = \perm{231546}   
\end{align*}
We start counting at the leaves 
\begin{align*}
\textup{vertex}(\Pi,\CT,b) &= \scalebox{0.8}{\begin{tabular}{|>{\columncolor{white}[0pt][\tabcolsep]}l|l|l|l|l|l|}
\arrayrulecolor{gray} % Set the line color to gray
\hline
  &  &  &  &  & 1 \\ \hline
  &  &  & 1 &  &  \\ \hline
  &  &  &  & 1 &  \\ \hline
  & 1 &  &  &  &  \\ \hline
 1 &  &  &  &  &  \\ \hline
  &  & 1 &  &  &  \\ \hline
\end{tabular}}
\,, 
\textup{vertex}(\Pi,\CT,c) = \scalebox{0.8}{\begin{tabular}{|>{\columncolor{white}[0pt][\tabcolsep]}l|l|l|l|l|l|}
\arrayrulecolor{gray} % Set the line color to gray
\hline
  &  &  &  &  & 1 \\ \hline
  &  &  & 1 &  &  \\ \hline
  &  &  &  & 1 &  \\ \hline
  & 1 &  &  &  &  \\ \hline
 1 &  &  &  &  &  \\ \hline
  &  & 1 &  &  &  \\ \hline
\end{tabular}}\end{align*}
and then going up the edges
\begin{align*}
\textup{edge}(\Pi,\CT,(a,b)) &= \scalebox{0.8}{\begin{tabular}{|>{\columncolor{white}[0pt][\tabcolsep]}l|l|l|l|l|l|}
\arrayrulecolor{gray} % Set the line color to gray
\hline
  &  &  &  &  & 0 \\ \hline
  &  &  & 1 &  &  \\ \hline
  &  &  &  & 1 &  \\ \hline
  & 3 &  &  &  &  \\ \hline
 4 &  &  &  &  &  \\ \hline
  &  & 3 &  &  &  \\ \hline
\end{tabular}},\
 \textup{edge}(\Pi,\CT,(a,c)) =\scalebox{0.8}{\begin{tabular}{|>{\columncolor{white}[0pt][\tabcolsep]}l|l|l|l|l|l|}
\arrayrulecolor{gray} % Set the line color to gray
\hline
  &  &  &  &  & 0 \\ \hline
  &  &  & 0 &  &  \\ \hline
  &  &  &  & 1 &  \\ \hline
  & 0 &  &  &  &  \\ \hline
 0 &  &  &  &  &  \\ \hline
  &  & 2 &  &  &  \\ \hline
\end{tabular}}
\end{align*}
We then take the entry-wise product
\begin{align*}
\textup{edge}(\Pi,\CT,(a,b)) \cdot \textup{edge}(\Pi,\CT,(a,c)) = \textup{vertex}(\Pi,\CT,a) &= \scalebox{0.8}{\begin{tabular}{|>{\columncolor{white}[0pt][\tabcolsep]}l|l|l|l|l|l|}
\arrayrulecolor{gray} % Set the line color to gray
\hline
  &  &  &  &  & 0 \\ \hline
  &  &  & 0 &  &  \\ \hline
  &  &  &  & 1 &  \\ \hline
  & 0 &  &  &  &  \\ \hline
 0 &  &  &  &  &  \\ \hline
  &  & 6 &  &  &  \\ \hline
\end{tabular}}
\end{align*}
and finally
\begin{align*}
\textup{edge}(\Pi,\CT,(r,a)) =  \textup{vertex}(\Pi,\CT, r) = \scalebox{0.8}{\begin{tabular}{|>{\columncolor{white}[0pt][\tabcolsep]}l|l|l|l|l|l|}
\arrayrulecolor{gray} % Set the line color to gray
\hline
  &  &  &  &  & 0 \\ \hline
  &  &  & 1 &  &  \\ \hline
  &  &  &  & 0 &  \\ \hline
  & 6 &  &  &  &  \\ \hline
 6 &  &  &  &  &  \\ \hline
  &  & 0 &  &  &  \\ \hline
\end{tabular}}
\end{align*}
Summing $\textup{vertex}(\Pi,\CT, r)$ yields the number of occurrences. We represented the tuples as grids to help visualize the directions.
\end{example}

\subsubsection*{Algorithm to count occurrences of a pure-west corner tree}
\label{subsection:ct_purewest}
We now present \Cref{algorithm_pure_west}, to count occurrences of a pure west corner tree. Slight variations of this algorithm are used in \Cref{subsection:algo_arbone} to count $\TypeA$ and $\TypeAnotB$, respectively. 

Let $\CT \in \CornerTrees$ such that all its edge labels are either $\SW$ or $\NW$. Let $a \in V(\CT)$ be a vertex of $\CT$ with children $c_{1}, \dots ,c_{k}$. Let $\Pi \in \permutations(n)$ be a permutation. For each $i \in \{1,...,n\}$, we can write 
\begin{align*}
   \textup{vertex}(\Pi, \CT, a)[i] &=\left(\sum_{\substack{j<i\\\cdots}} \textup{vertex}(\Pi, \CT, c_{1})[j]\right)\cdots \left(\sum_{\substack{j<i\\\cdots}} \textup{vertex}(\Pi, \CT, c_{k})[j]\right).
\end{align*}
This means that the number of occurrences where the root is mapped to $i$ only depends on the occurrences of the child vertices for which $j < i$. This allows us to count the occurrences of such corner trees by scanning the permutation only once from left to right. Indeed, to count the occurrences of a pure west corner tree, \Cref{algorithm_occurrences} can be rewritten as \Cref{algorithm_pure_west}. 

\begin{algorithm}
\caption{Counting pure-west corner trees occurrences}
\label{algorithm_pure_west}
\begin{spacing}{1.2}

\begin{algorithmic}[1]
\algfont
\Function{COUNT-W}{permutation $\Pi$, root $r$} \Comment{{\scriptsize $r$ is the root of the pure-west corner tree $\CT$}} 
        \State $\text{vertex\_dict} \gets \text{ DICTIONARY}()$\Comment{{\scriptsize indexed by subtrees}} %isomorphism classes of subtrees}}
        % \State $\text{touched} \gets \text{ DICTIONARY}()$\Comment{{\scriptsize indexed by isomorphism classes of subtrees}}
        \State $\text{edge\_dict} \gets \text{ DICTIONARY}()$ \Comment{{\scriptsize indexed by subtrees and edge labels}}
        \State $\text{result} \gets \text{ARRAY}(0, \ldots, 0)$ \Comment{{\scriptsize of size $n = |\Pi|$}}

        \Function{VERTEX-W}{permutation $\Pi$, vertex $v$, index $i$} \Comment{{\scriptsize $v$ is the subtree rooted at $v$}}
            % \State $h \gets \text{hash}(v)$ \Comment{{\scriptsize isomorphism class of subtree rooted at $v$}}

            \If{$v \notin \text{vertex\_dict}$}
                \State $\text{vertex\_dict}[v] \gets  \text{ARRAY}(1, \ldots, 1)$ \Comment{{\scriptsize of size $n$}}
                % \State $\text{touched}[h] \gets  \text{ARRAY}(\textbf{False}, \ldots, \textbf{False})$ \Comment{{\scriptsize of size $n$}}
            \EndIf
            % \If{\textup{not} $\text{touched}[h][i]$}
            % \State $\text{touched}[h][i] \gets \textbf{True}$ \Comment{{\scriptsize at iteration $i$ we have encountered the isomorphism class $h$}}
            \ForAll{$e \in \text{child-edges}(v)$}
                \State $\text{vertex\_dict}[v][i] \gets \text{vertex\_dict}[v][i] \cdot \Call{EDGE-W}{\Pi, e, i}$ 
            \EndFor
            % \EndIf
            \State \Return $\text{vertex\_dict}[v][i]$
        \EndFunction
        
    \Function{EDGE-W}{permutation $\Pi$, $e$, index $i$} \Comment{{\scriptsize $e$ is an edge of $\CT$}}
             \State $\text{edge\_id} \gets (\text{child\_vertex}(e),\edgelabel(e))$

            \If{$\text{edge\_id} \notin \text{edge\_dict}$}
                \State $\text{edge\_dict}[\text{edge\_id}] \gets \text{ SUM-TREE}(0, \ldots, 0)$ \Comment{{\scriptsize of size $n$}}
            \EndIf

            \State $\text{edge\_dict}[\text{edge\_id}][\Pi(i)] \gets \Call{VERTEX-W}{\Pi, \text{child\_vertex}(e), i}$

            \If{$\edgelabel(e) = \SW$}
                \State \Return $\text{edge\_dict}[\text{edge\_id}][\Pi[i]].\text{prefix\_sum}$ 
            \Else
                \State \Return $\text{edge\_dict}[\text{edge\_id}][\Pi[i]].\text{suffix\_sum}$
            \EndIf
        \EndFunction
   
    \For{$i = 1$ \textbf{to} $n$}
                \State $\text{result}[i] \gets \Call{VERTEX-W}{\Pi, r, i}$ \Comment{{\scriptsize$r$ is the root vertex of the corner tree $\CT$}}
            \EndFor
        \State \Return $\text{result}$
    \EndFunction
     \end{algorithmic}
     \end{spacing}
\end{algorithm}

\subsection{Algorithm to count occurrences of \ArboNE}
\label{subsection:algo_arbone}
We now provide a high-level explanation of the algorithm that allows to count occurrences of the tree double posets $\ArboNE$ on permutations. We also refer to the code in \cite{diehl_verri_generalization_2025}. There are three sub-algorithms for each type of counting. We refer to \Cref{remark:algo_cases} where we illustrate why an algorithm that counts the following three cases yields the correct counting. Observe that $\TypeA$ and $\TypeAnotB$ are slight variations of \Cref{algorithm_pure_west}.

\subsubsection*{Algorithms to count occurrences of $\TypeA(\arboNE,\SntoDP(\Pi))$}

Let $\arboNE \in \ArboNE, \Pi \in \permutations(n)$, and consider an interval partition $\mathtt{Int}=\{\mathtt{Int}_{h}|h \in H\}$ whose blocks have equal size $m < n$. We count $\TypeA$ as follows. For each block $\mathtt{Int}_{h}$, we run a slightly modified version of \Cref{algorithm_pure_west}. We perform once the loop $i=1,\dots, n$. Each time that $\Pi(i)<\min(\mathtt{Int}_{h})$, we increase the counter {\algfont countCTW} with the occurrences of the pure-west corner tree $\TTDPtoCT(\arboNE|_{V(\arboNE)\setminus \{\MAX\}},\ME)$ whose root is mapped to $i$. If $\Pi(i)\in \mathtt{Int}_{h}$, we return {\algfont countCTW}: these are the occurrences of $\TypeA$ where $\MAX$ gets mapped to $i$. Finally, if $\Pi(i) > \mathtt{Int}_{h}$ we do nothing. We do this for all the blocks in $\mathtt{Int}$, i.e., a total number of $\lceil n/m \rceil$ times. 

\subsubsection*{Algorithm to count occurrences of $\TypeAnotB(\arboNE,\SntoDP(\Pi))$}

Let $\arboNE \in \ArboNE, \Pi \in \permutations(n)$, and consider an interval partition $\mathtt{Int}=\{\mathtt{Int}_{h}|h \in H\}$ whose blocks have equal size $m < n$. We can count $\TypeAnotB$ as follows. For each block $\mathtt{Int}_{h}$, we run a slightly modified version of \Cref{algorithm_pure_west}. We perform once the loop $i=1,\dots, n$. Each time that $\Pi(i)<\min(\mathtt{Int}_{h})$, we increase the counter {\algfont countCT} with the occurrences of the pure-west corner tree $\TTDPtoCT(\arboNE|_{V(\arboNE)\setminus \{\MAX\}},\ME)$ whose root is mapped to $i$. Now each time $i=\min(\mathtt{Int}_{\ell})$ for some $\ell$ we store in {\algfont countCTback}, the occurrences we counted so far in {\algfont countCT}. If $\Pi(i)\in \mathtt{Int}_{h}$, we return {\algfont countCT} - {\algfont countCTback}, i.e. the occurrences of $\TypeAnotB$ where $\MAX$ gets mapped to $i$. Again, if $\Pi(i) > \mathtt{Int}_{h}$ we do nothing. We do this for all the blocks in $\mathtt{Int}$, i.e. $\lceil n/m \rceil$ times.

\subsubsection*{Algorithm to count occurrences of $\TypenotAnotB(\arboNE,\SntoDP(\Pi))$}

The following remark concerns product trees, i.e., an efficient way of updating a square table and computing \quot{rectangular} sums, see also the class {\algfont ProductTree} in \cite{diehl_verri_generalization_2025}. This data structure is used for the algorithm that counts occurrences of $\TypenotAnotB$.
\begin{remark}
Say we have created a $4 \times 4$ table. Then storing the value $c$ at coordinates $x=2, y=3$ yields
\begin{align*}
\begin{matrix}
\textbf{\textup{4}} &0  & 0 & 0 & 0\\
\textbf{\textup{3}} &0  & c & 0 & 0\\
\textbf{\textup{2}} &0  & 0 & 0 & 0\\
\textbf{\textup{1}} &0  & 0 & 0 & 0\\
 & \textbf{\textup{1}} & \textbf{\textup{2}} & \textbf{\textup{3}} & \textbf{\textup{4}}
\end{matrix} 
\end{align*}
which means creating the keys $\{[2,3),[1,3),[1,5)\} \times \{[3,4),[3,5),[1,5)\}$ in a dictionary,
\begin{align*}
    &\{(2, 3, 3, 4): c, (2, 3, 3, 5): c, (2, 3, 1, 5): c, (1, 3, 3, 4): c, (1, 3, 3, 5): c, (1, 3, 1, 5): c,\\&(1, 5, 3, 4): c, (1, 5, 3, 5): c, (1, 5, 1, 5): c\}
\end{align*}
and assign them values $c$. The keys correspond to the pairs of dyadic intervals containing $1$ and $2$, respectively. Therefore, updating a table of size $n$ costs $(\lceil\log_{2}(n)\rceil+1)^{2}$. If we use this encodings for square tables, we can perform rectangular queries using the coarsest interval partitions based on dyadic intervals. For example to sum over $([3,11),[0,3))$, we loop over
\begin{align*}
    [3,11) &= [3,4) \cup [4,8) \cup [8,10) \cup [10,11),\quad\text{and}\\
    [0,3) &= [0,2) \cup [2,3).\end{align*}
Generating such partitions costs at most $(\lceil\log_{2}(n)\rceil+1)^{2}$ operations. We then loop over at most $(\lceil\log_{2}(n)\rceil+1)^{2}$ keys of the dictionary to get the sum of the corresponding values.
\end{remark}

For $\TypenotAnotB$ we proceed as follows. Consider the corner trees dangling from $\ME$, the corner tree between $\ME$ and $\MN$, and the corner trees dangling from $\MN$. For each corner tree, $\CT$, we create a product tree where at entry $(i,\Pi(i))$ we store the number of occurrences of $\CT$ on $\Pi$ where the root is mapped to $i$. Consider an interval partition $\mathtt{Int}=\{\mathtt{Int}_{h}|h \in H\}$ whose blocks have equal size $m < n$. First we loop through the permutation points $(i_{1},\Pi(i_{1}))$, where $i_{1}$ corresponds to the image of $\MAX$. Now say $i_{1}\in \mathtt{Int}_{h}$ for some $h$. We then loop over all pairs of points $(i_{2},\Pi(i_{2}))$ and $(i_{3},\Pi(i_{3}))$ such that $i_{2}<i_{3}, \Pi(i_{2})>\Pi(i_{3})$, and $i_{2},\Pi(i_{3}) \in \mathtt{Int}_{h}$, i.e. $i_{2}$ is the image of $\MN$ and $i_{3}$ is the image of $\ME$, respectively. For the corner trees dangling from $\ME$, on the respective product trees, we query the sum of the entries to the south-west $(i_{3},\Pi(i_{3}))$. Analogously, for the corner trees dangling from $\MN$, we query the sum of the entries to the south-west $(i_{2},\Pi(i_{2}))$. We also query the rectangular sum $(i_{2},i_{3}) \times (\Pi(i_{3}),\Pi(i_{2}))$ for the corner tree between $\ME$ and $\MN$. We then take the product of all these queries.

\begin{lemma}
\label{lemma:complexity_cases}
   Let $\arboNE \in \ArboNE$, $\Pi \in \permutations(n)$, and let $ m < n $ be the size of the blocks of the partitions. Then 
   \begin{itemize}
       \item  $|\TypeA(\arboNE, \SntoDP(\Pi))|$ and $|\TypeAnotB(\arboNEprime, \SntoDP(\Pi^{-1}))|$ require $\tilde{\mathcal O}(n^{2}/m)$
   time and $\tilde\O(n)$ space,
\item  and $|\TypenotAnotB(\arboNE, \SntoDP(\Pi))|$ requires $\tilde{\mathcal O}(nm^2)$ time
   and $\tilde\O(n)$ space.
\end{itemize}
\end{lemma}
\begin{proof}
We only show the time complexity. We first argue for $|\TypeA(\arboNE, \SntoDP(\Pi))|$. The case $|\TypeAnotB(\arboNEprime, \SntoDP(\Pi^{-1}))|$ is analgous. We loop over the permutation $\lceil n/ m \rceil$ times. Within each loop, we count occurrences of the corner tree $\TTDPtoCT(\arboNE|_{V(\arboNE)\setminus \{\MAX\}},\ME)$ or update the counter of occurrences where $\MAX$ is mapped within a certain block. This costs $k n\log_{2}(n)$, where $k$ is the number of vertices of the tree. Therefore, the overall time complexity is $\tilde{\mathcal O}(n^{2}/m)$.

For $|\TypenotAnotB(\arboNE, \SntoDP(\Pi))|$ we argue as follows. Assume that there are $k_{1},k_{2} \in \NN$ corner trees dangling from $\ME$ and $\MN$. We store the occurrences of these corner trees and the corner tree between $\ME$ and $\MN$ in product trees. This costs $(k_{1}+k_{2}+1)n\log_{2}^{2}(n)$ operations. Then we loop once over the permutation, and for each point, being the image of $\MAX$, there are at most $m^{2}$ pairs of points being the images of $\ME$ and $\MN$ respectively, so that we count occurrences of $|\TypenotAnotB(\arboNE, \SntoDP(\Pi))|$. For each pair we query $(k_{1}+k_{2}+1)$ product-trees which costs $(k_{1}+k_{2}+1)\log_{2}^{2}(n)$ operations. Therefore, the overall time complexity is $\tilde{\mathcal O}(nm^{2})$.

\end{proof}
\begin{theorem}
\label{theorem:gen3214}
Let $\Pi \in \permutations(n)$ and $\arboNE \in \ArboNE$. Then counting
\begin{align*}
   |\Hom(\arboNE,\SntoDP(\Pi))|
\end{align*}
is feasible in time $\tilde{\mathcal O}(n^{5/3})$ and space $\tilde\O(n)$.
\end{theorem}
\begin{proof}
Minimizing the sum $n^{2}/m+nm^{2}$ with respect to $m$ yields $m= (\frac{n}{2})^{1/3}$. Picking the size of the blocks to be $m:=\lfloor n^{1/3}\rfloor$ yields the desired result. 
\end{proof}

\subsection{New directions at level 5}
\label{subsection:new_dir_level5}
Using \Cref{theorem:gen3214}, we can count three directions on $\Q[\permutations(5)]$ not spanned by corner trees with five vertices. Leveraging on some simple symmetries yields a total of twelve new directions.

\begin{remark}[Action of $D_{4}$]
\label{remark:actionD4}
If $P$ is a strict poset, we denote its opposite poset as $P^{\textup{op}}$ where
\begin{align*}
(b, a) \in P^{\textup{op}} \iff  (a, b) \in P.
\end{align*}
    Let $\smalldp :=(A,P_{A},Q_{A})$ be a double poset.
    Recall that $\Swap$ was defined in \Cref{remark:swaptalgo}. Define
    \begin{align*}
        \Rev(\smalldp) &:= (A,P^{\textup{op}}_{A},Q_{A})\\
         \Comp(\smalldp) &:= (A,P_{A},Q^{\textup{op}}_{A}).
    \end{align*}
It is not hard to see that the dihedral group with eight elements $D_{4}$ acts on the set of equivalence classes of double posets, through $\Swap,\Rev,\Comp$ and their composition.

Now, it is easy to see that for any $f$ in the group generated by $\Swap,\Rev,\Comp$,
\begin{align*}
    |\Hom(f(\smalldp),f(\largedp))| = |\Hom(\smalldp,\largedp)|.  
\end{align*}

As a consequence, for example
% Observe that for all double posets $\smalldp,\largedp\in \DP$, it holds that
\begin{align*}
    |\Hom(\Rev(\Swap(\smalldp)),\largedp)| = |\Hom(\smalldp,\Swap(\Rev(\largedp))|.  
\end{align*}

Finally, we note that the action can be restricted to the set of permutations realized as double posets,
and thereby yielding an action of $D_4$ on permutations (note that instead, as is, for example, done in \cite{even2021counting}, we could also just immediately
define the actions of $\Rev$, $\Comp$ and $\Swap$, resp., directly on permutations).

Then, with these two actions, the map $\Phi:=\proj_{\mathcal{S}}\circ\; \Phi_{\regmono \leftarrow \hom}$ satisfies
\begin{align*}
    \Phi( f(d) ) = f( \Phi_{\regmono \leftarrow \hom}(d) ),
\end{align*}
 i.e., is \textup{equivariant}. For example, we have
\begin{align*}
    \Phi( \Rev(\Swap(\smalldp)) )
    &=
    \Rev(\Swap( \Phi(\smalldp) )).
\end{align*}

%Similar identities for the other cases apply. Furthermore when $\largedp=\DPtoSn(\Pi)$, we have
%\begin{align*}
%    &|\Hom(\Rev(\Swap(\smalldp)),\DPtoSn(\Pi))| \\&= \sum_{\sigma}|\Epi(\Rev(\Swap(\smalldp)),\DPtoSn(\sigma))|\langle \PC(\Pi),\sigma \rangle\\&=|\Hom(\smalldp,\Swap(\Rev(\DPtoSn(\Pi)))|\\&=\sum_{n\in \NN}\sum_{\sigma\in \permutations(n)}|\Epi(\smalldp,\DPtoSn([\sigma(1)\cdots \sigma(n)]))|\langle \PC(\Pi),[\sigma^{}(1)\cdots \sigma(n)] \rangle
%\end{align*}
%
%
%For example
%As an example, for permutations, we  have
%\begin{align*}
%    \Rev(\SntoDP([\sigma(1)\cdots\sigma(n)])) &\cong  \SntoDP([\sigma(n)\cdots\sigma(1)])\\
%      \Comp(\SntoDP([\sigma(1)\cdots\sigma(n)])) &\cong \SntoDP([(n+1-\sigma(1))\cdots(n+1-\sigma(n))]).
%\end{align*}

\end{remark}

It is known to \cite{even2021counting} that 
\begin{align*}
    &\textup{dim}_\Q \ \DPtoSn\ \proj_{\mathcal{S}}\left(\Phi_{\regmono \leftarrow \hom}\left(\bigoplus_{n \le 5}\Q\left[\twintreedoubleposets(n)   \right]\right)\right)\cap \Q[\permutations(5)]  = 100,
\end{align*}
Using our generalization, we can count
\begin{align*}
\DPtoSn \ \proj_{\mathcal{S}}\left(\Phi_{\regmono \leftarrow \hom}(\Q[\ArboNE(5)])\right).
\end{align*}
This yields three new directions, given by three elements of $\ArboNE$. We get nine more directions if we let $D_{4}$ act on them, see \Cref{remark:actionD4}. Observe that twin tree double posets are closed under the action of $D_{4}$ while this does not hold for general elements of $\ArboNE$. Applying \Rev, \Comp\;and their composition to the double posets in \Cref{figure:three double posets} kicks us out of $\ArboNE$.
\begin{figure}[H]
\centering
\begin{minipage}{0.3\textwidth}
        \centering
      \scalebox{0.2}{\begin{tikzpicture}[every node/.style={circle, draw, inner sep=4pt, minimum size=6mm}]
\node (X) at (8,4) {\scalebox{3}{$\MAX$}};
\node (A) at (-2,0) {\scalebox{3}{$\ME$}};
\node (B) at (2,-4) {\phantom{\scalebox{2}{$\MN$}}};
\node (C) at (6,-8) {\scalebox{3}{$\MN$}};
\node (D) at (-2,-4) {\phantom{\scalebox{2}{$\MN$}}};

\node[draw=none] (H) at (2,-8) {\phantom{\scalebox{2}{$\MN$}}};

\node[draw=none] (K) at (6,-12) {\phantom{\scalebox{2}{$\MN$}}};

    \tikzset{mid arrow/.style={
        postaction={decorate,decoration={
            markings,
            mark=at position .5 with {\arrow[scale=3]{stealth}}
        }}
    }}
        \draw[-{Stealth[scale=3]}] (X) -- (A) node[draw = none, pos=0.4, above, yshift=10pt] {\scalebox{2}{\S}};
        \draw[-{Stealth[scale=3]}] (X) -- (B) node[draw = none, pos=0.4, above, yshift=10pt] {\scalebox{2}{\S}};
         \draw[-{Stealth[scale=3]}] (X) -- (C) node[draw = none, pos=0.4, right,xshift=3pt, yshift=10pt] {\scalebox{2}{\S}};
    \draw[-{Stealth[scale=3]}] (A) -- (B) node[draw = none, pos=0.4, above, yshift=10pt] {\scalebox{2}{\N}};
    \draw[-{Stealth[scale=3]}] (B) -- (C) node[draw = none, pos=0.4, above, yshift=10pt] {\scalebox{2}{\N}};
    \draw[-{Stealth[scale=3]}] (A) -- (D) node[draw = none, pos=0.5, left, yshift=10pt] {\scalebox{2}{\S}};
\end{tikzpicture}}
       \caption*{$\arboNEone$}
    \end{minipage}
     \begin{minipage}{0.3\textwidth}
        \centering
      \scalebox{0.2}{\begin{tikzpicture}[every node/.style={circle, draw, inner sep=4pt, minimum size=6mm}]
\node (X) at (8,4) {\scalebox{3}{$\MAX$}};
\node (A) at (-2,0) {\scalebox{3}{$\ME$}};
\node (B) at (2,-4) {\phantom{\scalebox{2}{$\MN$}}};
\node (C) at (6,-8) {\scalebox{3}{$\MN$}};
\node[draw=none] (D) at (-2,-4) {\phantom{\scalebox{2}{$\MN$}}};

\node (H) at (2,-8) {\phantom{\scalebox{2}{$\MN$}}};

\node[draw=none] (K) at (6,-12) {\phantom{\scalebox{2}{$\MN$}}};

    \tikzset{mid arrow/.style={
        postaction={decorate,decoration={
            markings,
            mark=at position .5 with {\arrow[scale=3]{stealth}}
        }}
    }}
        \draw[-{Stealth[scale=3]}] (X) -- (A) node[draw = none, pos=0.4, above, yshift=10pt] {\scalebox{2}{\S}};
        \draw[-{Stealth[scale=3]}] (X) -- (B) node[draw = none, pos=0.4, above, yshift=10pt] {\scalebox{2}{\S}};
         \draw[-{Stealth[scale=3]}] (X) -- (C) node[draw = none, pos=0.4, right,xshift=3pt, yshift=10pt] {\scalebox{2}{\S}};
    \draw[-{Stealth[scale=3]}] (A) -- (B) node[draw = none, pos=0.4, above, yshift=10pt] {\scalebox{2}{\N}};
    \draw[-{Stealth[scale=3]}] (B) -- (C) node[draw = none, pos=0.4, above, yshift=10pt] {\scalebox{2}{\N}};
    \draw[-{Stealth[scale=3]}] (B) -- (H) node[draw = none, pos=0.5, left, yshift=10pt] {\scalebox{2}{\S}};
\end{tikzpicture}}
       \caption*{$\arboNEtwo$}
    \end{minipage}
    \begin{minipage}{0.3\textwidth}
        \centering
      \scalebox{0.2}{\begin{tikzpicture}[every node/.style={circle, draw, inner sep=4pt, minimum size=6mm}]
\node (X) at (8,4) {\scalebox{3}{$\MAX$}};
\node (A) at (-2,0) {\scalebox{3}{$\ME$}};
\node (B) at (2,-4) {\phantom{\scalebox{2}{$\MN$}}};
\node (C) at (6,-8) {\scalebox{3}{$\MN$}};
\node[draw=none] (D) at (-2,-4) {\phantom{\scalebox{2}{$\MN$}}};

\node[draw=none] (H) at (2,-8) {\phantom{\scalebox{2}{$\MN$}}};

\node (K) at (6,-12) {\phantom{\scalebox{2}{$\MN$}}};

    \tikzset{mid arrow/.style={
        postaction={decorate,decoration={
            markings,
            mark=at position .5 with {\arrow[scale=3]{stealth}}
        }}
    }}
        \draw[-{Stealth[scale=3]}] (X) -- (A) node[draw = none, pos=0.4, above, yshift=10pt] {\scalebox{2}{\S}};
        \draw[-{Stealth[scale=3]}] (X) -- (B) node[draw = none, pos=0.4, above, yshift=10pt] {\scalebox{2}{\S}};
         \draw[-{Stealth[scale=3]}] (X) -- (C) node[draw = none, pos=0.4, right,xshift=3pt, yshift=10pt] {\scalebox{2}{\S}};
    \draw[-{Stealth[scale=3]}] (A) -- (B) node[draw = none, pos=0.4, above, yshift=10pt] {\scalebox{2}{\N}};
    \draw[-{Stealth[scale=3]}] (B) -- (C) node[draw = none, pos=0.4, above, yshift=10pt] {\scalebox{2}{\N}};
    \draw[-{Stealth[scale=3]}] (C) -- (K) node[draw = none, pos=0.5, left, yshift=10pt] {\scalebox{2}{\S}};
\end{tikzpicture}}
       \caption*{$\arboNEthree$}
    \end{minipage}
\caption{Three elements of $\ArboNE$ that span three new directions.}  
  \label{figure:three double posets}
\end{figure}

We improve to 
\begin{align*}
    &\textup{dim}_\Q \ \DPtoSn\ \proj_{\mathcal{S}}\left(\Phi_{\regmono \leftarrow \hom}\left(\bigoplus_{n \le 5}\Q[\twintreedoubleposets(n)]\oplus \Q[\textbf{New}]\right)\right)\cap \Q[\permutations(5)]  = 112,
\end{align*}
where \begin{align*}
    &\textbf{New}:=\{\arboNEone,\arboNEtwo,\arboNEthree,\Rev(\arboNEone),\Rev(\arboNEtwo),\Rev(\arboNEthree),\\&\quad\quad\quad\quad\quad\Comp(\arboNEone),\Comp(\arboNEtwo),\Comp(\arboNEthree),\\&\quad\quad\quad \quad\quad\Rev(\Comp(\arboNEone)),\Rev(\Comp(\arboNEtwo),\Rev(\Comp(\arboNEthree))\}.\end{align*} To check the correctness of our results, in \cite{diehl_verri_generalization_2025}, we provide an implementation of the 5-profile using as a basis for $\bigoplus_{n \le 5}\Q[\permutations(n]$ 
\begin{align*}
    \bigcup_{n \le 5}\twintreedoubleposets(n)\cup \{\DPtoSn(\perm{3214})\} \cup \textbf{New}\cup \textbf{patterns}
\end{align*}
where \textbf{patterns} is the set of patterns 
$$\{\perm{12435},\perm{12453},\perm{13245},\perm{13254},\perm{13425},\perm{14235},\perm{14325},\perm{14352}\}$$
not spanned by our double posets. Note that any set of eight linearly independent vectors not contained in the span would have sufficed.

\section{Conclusion and outlook}
In this work, we have shown that corner trees and permutations belong to certain classes of double posets. This encoding leads to a broader theoretical framework that generalizes the one developed in \cite{even2021counting}. A generalization is necessary, even at the cost of a slower algorithm. Indeed, corner trees fail to count all permutation patterns already at level $4$. To address this issue, in the recent work of \cite{beniamini2024counting}, the authors introduce a generalization of corner trees, called \emph{pattern trees}. Within this framework, they are able to compute the $5$-profile in $\tilde{\mathcal{O}}(n^{7/4})$ time and the 6- and 7-profiles in $\tilde{\mathcal{O}}(n^{2})$ time.

 Here we introduce a family of tree double posets that generalize the permutation pattern $\perm{3214}$. The algorithm is based on ideas that are similar to the ones given in \cite{even2021counting} to count $\perm{3214}$ and the complexity is again $\tilde{\mathcal O}(n^{5/3})$. Using this generalization, we were able to fill twelve of the missing $20$ directions at level $5$ countable in $\tilde{\mathcal O}(n^{5/3})$ time. It remains open whether one can use similar ideas to develop other algorithms and increase the span of permutations countable with this time complexity. We also point out that our framework allows us to consider arbitrary families of double posets and correspondent algorithms to count permutation patterns faster than the naive approach.

\section{Open questions}

\begin{itemize}
\item In \cite{beniamini2024counting}, an algorithm is presented that computes the 5-profile in $\tilde{\mathcal O}(n^{7/4})$ time. We show that 112 directions can be computed in $\tilde{\mathcal O}(n^{5/3})$. Is it possible, using similar ideas and methods from \Cref{section:gen_algo}, to span all 120 directions within $\tilde{\mathcal O}(n^{5/3})$ time? For example, are there additional symmetries beyond the action of $D_{4}$ that we can exploit?
      \item Observe that  \begin{align*}
          &\proj_{\mathcal{S}}\circ\,\Phi_{\regmono \leftarrow \hom}\left(\;   \raisebox{-1cm}{\scalebox{0.5}{\begin{tikzpicture}[every node/.style={circle, draw, inner sep=2pt, minimum size=6mm}]
  % Define the vertices
   \node (A) at (0,2) {};
  \node (B) at (0,0) {};
  \node (C) at (0,-2) {};
 % Define styles for the arrowheads
  \tikzset{mid arrow/.style={
    postaction={decorate,decoration={
      markings,
      mark=at position .5 with {\arrow[scale=3]{stealth}}
    }}
  }}
% Draw the edges with arrows in the middle and labels
    \draw[-{Stealth[scale=3]}] (C) -- (B) node[draw = none, pos=0.1, right, yshift=10pt] {\S};
    \draw[-{Stealth[scale=3]}] (A) -- (B) node[draw = none, pos=0.5, right, yshift=10pt] {\S};
 \end{tikzpicture}}}\quad+\quad\raisebox{-1cm}{\scalebox{0.5}{\begin{tikzpicture}[every node/.style={circle, draw, inner sep=2pt, minimum size=6mm}]
  % Define the vertices
   \node (A) at (0,2) {};
  \node (B) at (0,0) {};
  \node (C) at (0,-2) {};
 % Define styles for the arrowheads
  \tikzset{mid arrow/.style={
    postaction={decorate,decoration={
      markings,
      mark=at position .5 with {\arrow[scale=3]{stealth}}
    }}
  }}
% Draw the edges with arrows in the middle and labels
    \draw[-{Stealth[scale=3]}] (B) -- (C) node[draw = none, pos=0.7, right, yshift=10pt] {\N};
  \draw[-{Stealth[scale=3]}] (B) -- (A) node[draw = none, pos=0.2, right, yshift=10pt] {\N};
 \end{tikzpicture}}}\quad-2\quad\raisebox{-1cm}{\scalebox{0.5}{\begin{tikzpicture}[every node/.style={circle, draw, inner sep=2pt, minimum size=6mm}]
  % Define the vertices
   \node (A) at (0,2) {};
  \node (B) at (0,0) {};
  \node (C) at (0,-2) {};
 % Define styles for the arrowheads
  \tikzset{mid arrow/.style={
    postaction={decorate,decoration={
      markings,
      mark=at position .5 with {\arrow[scale=3]{stealth}}
    }}
  }}
% Draw the edges with arrows in the middle and labels
    \draw[-{Stealth[scale=3]}] (B) -- (C) node[draw = none, pos=0.5, right, yshift=10pt] {\N};
   \draw[-{Stealth[scale=3]}] (A) -- (B) node[draw = none, pos=0.5, right, yshift=10pt] {\N};
 \end{tikzpicture}}}\quad -2\quad\raisebox{-1cm}{\scalebox{0.5}{\begin{tikzpicture}[every node/.style={circle, draw, inner sep=2pt, minimum size=6mm}]
  % Define the vertices
   \node (A) at (0,2) {};
  \node (B) at (0,0) {};
  \node (C) at (0,-2) {};
 % Define styles for the arrowheads
  \tikzset{mid arrow/.style={
    postaction={decorate,decoration={
      markings,
      mark=at position .5 with {\arrow[scale=3]{stealth}}
    }}
  }}
% Draw the edges with arrows in the middle and labels
    \draw[-{Stealth[scale=3]}] (B) -- (C) node[draw = none, pos=0.5, right, yshift=10pt] {\S};
    \draw[-{Stealth[scale=3]}] (A) -- (B) node[draw = none, pos=0.5, right, yshift=10pt] {\S};
 \end{tikzpicture}} }-2\quad\raisebox{-1cm}{\scalebox{0.5}{\begin{tikzpicture}[every node/.style={circle, draw, inner sep=2pt, minimum size=6mm}]
  % Define the vertices
   \node (A) at (0,2) {};
  \node (B) at (0,0) {};
  \node (C) at (0,-2) {};
 % Define styles for the arrowheads
  \tikzset{mid arrow/.style={
    postaction={decorate,decoration={
      markings,
      mark=at position .5 with {\arrow[scale=3]{stealth}}
    }}
  }}
% Draw the edges with arrows in the middle and labels
    \draw[-{Stealth[scale=3]}] (B) -- (C) node[draw = none, pos=0.5, right, yshift=10pt] {\N};
    \draw[-{Stealth[scale=3]}] (A) -- (B) node[draw = none, pos=0.5, right, yshift=10pt] {\S};
 \end{tikzpicture}} }\right)   \\&\\&=   \perm{12} +  \perm{21}. \end{align*}
At level 2, we can already count all directions using corner trees. At level four,  can we use corner trees at higher levels to increase the dimension of the subspace of permutations efficiently counted? 

\item Is there a closed formula for the sequences $|\twintreedoubleposets(n)|$, $|\ArboNE(n)|$?

\end{itemize}

\section*{Acknowledgments}
We would like to thank Diego Caudillo, Moritz Sokoll, and Leonard Schmitz for fruitful discussions.
\printbibliography

\section*{Notation}

\begin{itemize}
    \item $\NN := {0,1,2,\dots}$ — set of non-negative integers.
    \item $\permutations := \bigcup_{n \ge 0} \permutations(n)$, where $\permutations(n) := \{\sigma : [n] \to [n] \mid \sigma \text{ bijection} \}$ — set of all finite permutations.

    \item $\sigma, \tau, \dots \in \permutations$ — small permutations; written in one-line notation, e.g., $\mathtt{[2\,1\,3]}$.

    \item $\Pi, \Lambda, \dots \in \permutations$ — large permutations.

    \item $\Q[\permutations] := \bigoplus_{n \ge 0} \Q[\permutations(n)]$ — free $\Q$-vector space over permutations.

\item $\CornerTrees := \bigcup_{n \ge 0} \CornerTrees(n)$ — set of (isomorphism classes of) corner trees; elements $\CT \in \CornerTrees$. See \Cref{definition:ct}.

\item $\SNPolyTrees := \bigcup_{n \ge 0} \SNPolyTrees(n)$ — set of (isomorphism classes of) SN polytrees; elements $\Tsn \in \SNPolyTrees$. See \Cref{defintion:SN_poly}.

\item $\DP := \bigcup_{n \ge 0} \DP(n)$ — set of (isomorphism classes of) double posets; elements $\smalldp \in \DP$.
\item $\smalldp = (A, P_A, Q_A)$ — double poset with vertex set $A$ and strict partial orders $P_A$ and $Q_A$.
 \item $\vertexset(\smalldp)$ — vertex set of $\smalldp$, i.e. the underlying set of the double poset $\smalldp$.

\item $\largedp \in \DP$ — large double poset.
  \item $\Hom(\smalldp, \smalldp')$ — set of homomorphisms from the double poset $\smalldp$ to the double poset $\smalldp'$.

\item $\Mono(\smalldp, \smalldp')$ — set of injective homomorphisms from the double poset $\smalldp$ to the double poset $\smalldp'$.

\item $\RegMono(\smalldp, \smalldp')$ — set of embeddings from the double poset $\smalldp$ to the double poset $\smalldp'$.

\item $\smalltwin \in \twindoubleposets$ — twin double posets. See \Cref{definition:TwinDP}. 

\item $\smalltree \in \treedoubleposets$ — tree double posets. See \Cref{definition:TreeDP}. 

\item $\smallct \in \twintreedoubleposets$ — twin tree double post. See \Cref{definition:TwinTreeDP}.
      
          \item $\arboNE \in \ArboNE$ — tree double posets for which our algorithm works. See \Cref{definition:algo_guys}.
          \item $<_{\textup{West}}, <_{\textup{South}}$ — notation for the two orders in $\treedoubleposets$ and $\twindoubleposets$.
          
\item $\CTtoSNpoly : \CornerTrees \to \SNPolyTrees$ — map from a corner tree to its SN polytree. See \Cref{proposition:CT_SN_Poly}.
         \item $\SNpolytoCT(\Tsn, v)$ — map from SN polytree $\Tsn$ to corner tree rooted at $v$. See \Cref{remark:rooting_SN_poly}.
         \item $\TTDPtoCT(\smallct, v)$ — map from twin tree double poset to corner tree rooted at $v$. See \Cref{remark:rooting_tt}.
          
         \item $\SNpolytoSDP : \SNPolyTrees \to \twintreedoubleposets$ — map from SN polytree to twin tree double poset. See \Cref{lemma:corresp_SN_poly_Twin_tree_DP}.
         \item $\SDPtoSNpoly : \twintreedoubleposets \to \SNPolyTrees$ — map from twin tree double poset to SN polytree. See \Cref{lemma:corresp_SN_poly_Twin_tree_DP}.

          \item $\mathcal{S} := \SntoDP(\permutations)$ — set of permutations embedded as double posets.
         
          \item $\SntoDP : \permutations \to    \mathcal{S}$ — map from permutation to its double poset representation. See \Cref{definition:perm_as_double}.
           
           \item $\DPtoSn : \mathcal{S} \to \permutations$ — map taking a permutation represented as a double poset to its underlying permutation. See \Cref{definition:perm_as_double}.

\end{itemize}

\appendix

\section{Supplementary material}
\label{section:suppl}
The goal of this appendix is to show that counting homomorphisms between double posets can always be translated to counting double order embeddings, see \Cref{theorem:hom_reg_mono}. This implies that occurrences of double posets on permutations can always be translated to linear combinations of pattern counts, see \Cref{proposition:countpermasperm}.

The appendix is structured as follows.
In Section\;\ref{subsection:cat_prel} we review the categorical concepts relevant to our work. In particular, we show that certain factorizations are essentially unique (this holds  for (\Epi,\RegMono)-factorizations, see \Cref{proposition:rmer}) and lead to certain combinatorial identities, see \Cref{proposition:comb_id}. 
As a warm-up, in Section\;\ref{subsection:fin_set_bin}  we provide characterizations of regular monomorphisms and epimorphisms in the category of finite directed graphs, \DiGraphs, and show that \DiGraphs\;admits an (\Epi,\RegMono)-factorization. In Section\;\ref{subsection:Poset} we show an analogous result for the category of strict posets, \Poset. With some slight modifications, we extend this to the category of double posets, \DPoset\;and prove \Cref{theorem:hom_reg_mono}, in Section\;\ref{subsection:d_strict_poset}.

\subsection{Category theory: preliminary notions}
\label{subsection:cat_prel}

In this section, we recall some basic notions from category theory. We refer to \parencite{addmek1990abstract} for a complete introduction to category theory.

A \emph{category} $\mathcal{C}$ consists of:
\begin{itemize}
    \item A collection of \emph{objects}, denoted as $\text{Ob}(\mathcal{C})$.
    \item A collection of \emph{morphisms} (or arrows) between objects. For each pair of objects $A, B \in \text{Ob}(\mathcal{C})$, there is a set of morphisms $\Hom(A, B)$. These morphisms can be composed, i.e., if $f : A \to B$ and $g : B \to C$ are morphisms, then there is a morphism $g \circ f : A \to C$. Furthermore, for each object $A \in \mathcal{C}$, there is an \emph{identity morphism} $\id_{A} : A \to A$ such that for all morphisms $f: A \to B$ and $g: C \to A$, we have $f \circ \id_{A}= f$ and $\id_{A}\circ g = g$.
\end{itemize}

Composition of morphisms is \emph{associative}, i.e., if $f : A \to B$, $g : B \to C$, and $h : C \to D$ are morphisms, then $$h \circ (g \circ f) = (h \circ g) \circ f.$$

\begin{example}
An example of a category is the category $\FinSet$, where the objects are finite sets and the morphisms are maps between sets. Another example is the category of $\DiGraphs$, where the objects are finite directed graphs and the morphisms are graph homomorphisms.  
\end{example}

A \emph{functor} is a structure-preserving map between categories. Given two categories $\mathcal{C}$ and $\mathcal{D}$, a functor $F : \mathcal{C} \to \mathcal{D}$ consists of:
\begin{itemize}
    \item A function $F : \text{Ob}(\mathcal{C}) \to \text{Ob}(\mathcal{D})$ that assigns to each object $A \in \mathcal{C}$ an object $F(A) \in \mathcal{D}$.
    \item A function $F : \Hom_{\mathcal{C}}(A, B) \to \Hom_{\mathcal{D}}(F(A), F(B))$ that assigns to each morphism $f : A \to B$ in $\mathcal{C}$ a morphism $F(f) : F(A) \to F(B)$ in $\mathcal{D}$.
\end{itemize}
The functor must satisfy:
\begin{itemize}
    \item $F(\id_A) = \id_{F(A)}$ for each object $A \in \mathcal{C}$.
    \item $F(g \circ f) = F(g) \circ F(f)$ for all morphisms $f : A \to B$ and $g : B \to C$ in $\mathcal{C}$.
\end{itemize}

\begin{example}
An example of a functor is the \textbf{forgetful functor} $F: \DiGraphs \to \FinSet$ that sends a directed graph to its underlying set and that sends a graph homomorphism to the corresponding set map. 
\end{example}
\begin{definition}[Products and coproducts]
The \emph{product} of two objects $A$ and $B$ in a category $\mathcal{C}$ is the limit of the diagram consisting of the two objects $A$,$B$ with no morphism between them. It is formed by an object $P$ along with two projections $\pi_A : P \to A$ and $\pi_B : P \to B$ such that for any other object $N$ with morphisms $f_A : N \to A$ and $f_B : N \to B$, there exists a unique morphism $u : N \to P$ making the following diagram commute:

\[
\begin{tikzcd}
& N \arrow[dl, "f_A"'] \arrow[dr, "f_B"] \arrow[d, dashed, "u"] & \\
A & P \arrow[l, "\pi_A"] \arrow[r, "\pi_B"'] & B
\end{tikzcd}
\]

The \emph{coproduct} is the dual notion, where we reverse the arrows.
\end{definition}

\begin{definition}[Equalizers and Coequalizers]

An \emph{equalizer} of two morphisms $f, g : A \to B$ is an object $E$ with a morphism $e : E \to A$ such that $f \circ e = g \circ e$, and for any other object $N$ with a morphism $h : N \to A$ such that $f \circ h = g \circ h$, there exists a unique morphism $u : N \to E$ making the diagram commute.

Dually, a \emph{coequalizer} is the colimit of the same diagram.
 
\end{definition}

\begin{definition}[Pullbacks and Pushouts]
   The \emph{pullback} is a limit of a diagram of the form $A \xlongrightarrow{f} C \xlongleftarrow{g} B$. It consists of an object $P$ and morphisms $\pi_A : P \to A$ and $\pi_B : P \to B$ such that the following diagram commutes:

\[
\begin{tikzcd}
P \arrow[r, "\pi_B"] \arrow[d, "\pi_A"'] & B \arrow[d, "g"] \\
A \arrow[r, "f"'] & C
\end{tikzcd}
\]

Moreover, it satisfies the universal property that for any object $N$ with morphisms $f_A : N \to A$ and $f_B : N \to B$ such that $g \circ f_B = f \circ f_A$, there exists a unique morphism $u : N \to P$.

The \emph{pushout} is the dual concept, where the arrows are reversed. 
\end{definition}

Assume that in our category, finite products and equalizers, always exist. Then the following remarks show that pullbacks always exist. 
\begin{remark}
\label{remark:prod_plus_eq_pull}
Consider the diagram
\begin{figure}[H]
\centering
\begin{tikzcd}
 & Y \arrow[d, "g"]\\ X\arrow[r, "f"] &  Z
\end{tikzcd}
\end{figure}
and the cone
\begin{figure}[H]
\centering
\begin{tikzcd}
 Q \arrow[r, "q_{2}"] \arrow[d, "q_{1}"]& Y \arrow[d, "g"]\\ X\arrow[r, "f"] &  Z
\end{tikzcd}
\end{figure}
Since finite products exist, we have
\begin{figure}[H]
\centering
\begin{tikzcd}
 Q \arrow[dr, "q", dashed]\arrow[drr, bend left, "q_2"] \arrow[ddr, bend right, "q_1"']&&\\
 & X \times Y \arrow[r, "\pi_{Y}"] \arrow[d, "\pi_{X}"]& Y \arrow[d, "g"]\\& X\arrow[r, "f"] &  Z
\end{tikzcd}
\end{figure}
Since equalizers exist, we know that there exist an object $P$ and a morphism $e:P\to X \times Y$ such that
\begin{figure}[H]
\begin{center}
\begin{tikzcd}
P \arrow[r,"e"]
&X \times Y\arrow[r, "f\circ \pi_{X}", shift left=0.75ex] \arrow[r,"g \circ \pi_{Y}"',shift right=0.75ex]
& (B,R_{B}) \\
Q\arrow[u,"\hat{q}",dashed]  \arrow[ru,"q"'] && 
\end{tikzcd}
\end{center}
\end{figure}
commutes. Finally, we have 
\begin{figure}[H]
\centering
\begin{tikzcd}
 Q \arrow[dr, "\hat{q}", dashed]\arrow[drr, bend left, "q_2"] \arrow[ddr, bend right, "q_1"']&&\\
 & P \arrow[r, "\pi_{Y} \circ e"] \arrow[d, "\pi_{X} \circ e"']& Y \arrow[d, "g"]\\& X\arrow[r, "f"] &  Z
\end{tikzcd}
\end{figure}
Dually, coproducts and coequalizers imply the existence of pushouts.
\end{remark}

We now recall the definitions of some types of morphisms.

Let $A, B$ be objects in some category $\Cat$.
A morphism $f \in \Hom_{\Cat}(A, B)$ is
\begin{itemize}
\item
    a \DEF{monomorphism}
    if for all objects $C$ and all $g,h \in \Hom_{\Cat}(C, A)$
    \begin{align*}
      f \circ g  = f \circ h \Rightarrow g = h.
    \end{align*}

  \item 
    an \DEF{epimorphism},
    if for all objects $C$ and for all $g,h \in \Hom_{\Cat}(B, C)$
    \begin{align*}
      g \circ f = h \circ f \Rightarrow g = h.
    \end{align*}

   \item
    a \DEF{regular monomorphism}
    if it is the equalizer of some parallel pair of morphisms,
    i.e. if there is a limit diagram of the form
    \begin{align*}
      A \arr{f} B \rightrightarrows D.
    \end{align*}
    It is well-known that a regular monomorphism is a monomorphism.
    
    \item  an \DEF{isomorphism} if it has a two-sided inverse: there is $f^{-1} \in \Hom_{\Cat}(B,A)$
    with $f^{-1} \circ f = \id_A, f \circ f^{-1} = \id_{B}$.
   \end{itemize}

We will explain how morphisms can be uniquely factored using epimorphisms and regular monomorphisms.

From \parencite{addmek1990abstract}, we recall the following definitions.

Let $E, M$ be classes of morphisms in $\Cat$.
We say that $\Cat$ has \DEF{$(E,M)$-factorization}
if every morphism $f$
in $\Cat$ can be written as $f=e \circ m$
for some $e \in E, m \in M$.

\begin{definition}
\label{definition:EMfact}
A category $\Cat$ is \DEF{$(E,M)$-structured} if it satisfies the following
\begin{enumerate}
 \item if $e \in E$ and $h$ is an isomorphism such that $h \circ e$ exists, then $h \circ e \in E$,
 \item if $m \in M$ and $h$ is an isomorphism such that $m \circ h$ exists, then $m \circ H \in M$,
 \item  $\Cat$ has \textbf{$(E,M)$-factorization},
 \item $\Cat$ has the \textbf{unique} \textbf{$(E,M)$-diagonalization property}, i.e., for each commutative diagram
 \begin{center}
\begin{tikzcd}
A \arrow[r, "e"] \arrow[d,"f"']
& B \arrow[d, "g"] \\
C \arrow[r, "m"']
&  D
\end{tikzcd}
\end{center}
with $e \in E$ and $m \in M$ there exists a unique morphism $d$ such that the following diagram commutes
\begin{center}
\begin{tikzcd}
A \arrow[r, "e"] \arrow[d,"f"']
& B \arrow[d, "g"] \arrow[dl,"d"'] \\
C \arrow[r, "m"']
&  D
\end{tikzcd}
\end{center}
\end{enumerate}
\end{definition}

The proof of the following proposition can be found in \parencite{addmek1990abstract}, see Proposition 14.4.
 
\begin{proposition}
\label{proposition:unique_fact}
If $\mathcal{C}$ is $(E,M)$-structured, then $(E,M)$-factorizations are essentially unique, i.e.,
\begin{enumerate}
\item if $A \arr{e_{i}} C_{i}\arr{m_{i}} B$ are $(E,M)$-factorizations of $A \arr{f}B$ for $i = 1, 2$, then there
exists a (unique) isomorphism $h$, such that the following diagram commutes:
\begin{center}
\begin{tikzcd}
A \arrow[r, "e_{1}"] \arrow[d,"e_{2}"']
& C_{1} \arrow[d, "m_{1}"] \arrow[dl,"h"'] \\
C_{2} \arrow[r, "m_{2}"']
&  B
\end{tikzcd}
\end{center}

\item if $A \arr{f}B = A \arr{e} C\arr{m} B$ is a factorization and $C \arr{h} D$
is an isomorphism, then $A \arr{f}B = A \arr{h \circ e} D\arr{m \circ h^{-1}} B$ is also an $(E,M)$-factorization of $f$.
\end{enumerate}

\end{proposition}

\begin{proposition}
\label{proposition:rmer}
If a category $\Cat$ admits an (\Epi,\RegMono)-factorization, then the category is (\Epi,\RegMono)-structured. The factorization is unique in the sense of \Cref{proposition:unique_fact}.
\end{proposition}
\begin{proof}

 Consider the four conditions from \Cref{definition:EMfact}. Clearly, the $(\Epi,\RegMono)$-factorization satisfies the first two. For the unique diagonalization property, we can argue as follows. Let $e$ be an epimorphism and let $m$ be a regular monomorphism. Assume $g \circ e = m \circ f$, for some morphisms $f$ and $g$, i.e. consider the following commutative diagram.
\begin{figure}[pbth]
\begin{center}
\begin{tikzcd}
A \arrow[r, "e"] \arrow[d, "f"']
& B \arrow[d, "g"]  \\
 C \arrow[r, "m"']
& D
\end{tikzcd}
\end{center}
\end{figure}
Since $n$ is a regular monomorphism, we know that it is the limit of some parallel pair $h, h^{\prime}:D \to X$.
\begin{figure}[H]
\begin{center}
\begin{tikzcd}
 & A \arrow[r, "e"] \arrow[d, "f"']
& B \arrow[d, "g"]  \\
& C \arrow[r, "m"']
& D \arrow[r, shift left, "h"] \arrow[r, shift right, "h^{\prime}"'] & X 
\end{tikzcd}
\end{center}
\end{figure}
Therefore, $h \circ g \circ e = h^{\prime} \circ g \circ e$. Since $e$ is a epimorphism, it follows that $h \circ g = h^{\prime} \circ g$, i.e., $g$ is a cone. Therefore, there exists a unique arrow $d$ such that $m \circ d=g $. Since $m$ is a monomorphism,  $m \circ f = m \circ d \circ e$ implies $f = d \circ e$, and we are done. 
\begin{figure}[H]
\begin{center}
\begin{tikzcd}
 & A \arrow[r, "e"] \arrow[d, "f"']
& B \arrow[d, "g"] \arrow[dl, "d"] \\
& C \arrow[r, "m"']
& D \arrow[r, shift left, "h"] \arrow[r, shift right, "h^{\prime}"'] & X 
\end{tikzcd}
\end{center}
\end{figure}

\end{proof}

Let $\Cat$ be a concrete category where the underlying sets are finite. This means that there is a forgetful functor $U: \Cat \to \FinSet$ which sends the objects of $\Cat$ to their underlying sets and the morphisms in $\Cat$ to their underlying set maps. If such a category $\Cat$ is $(E,M)$-structured, we write
\begin{align*}
    E(A,B)&:=\{f \in \Hom(A,B)|f \in E\},\\
    M(A,B)&:=\{f \in \Hom(A,B)|f \in M\}.
\end{align*}

The following result holds.

\begin{proposition}
\label{proposition:comb_id}
Let $f \in \Hom(A, C)$ be an arrow in $\Cat$, where $\Cat$ is a concrete category where the underlying sets are finite. Additionally assume that $\Cat$ is $(E,M)$-structured. Therefore, we know that there exists an object $B$ in $\Cat$ such that $e \in E(A,B)$, $m \in M(B,C),$ and $f = m \circ e$. Then we have
\begin{align*}
|\{(r,s) \in E(A,B)\times M(B,C)|\;f &= s \circ r \}| = |\Aut(B)|.
\end{align*}
Furthermore
\begin{align*}
|\{t \in \Hom(A, C)|\;\exists e \in E(A,B),\exists m \in M(B,C), t = m \circ e\}| &= \frac{|E(A,B)||M(B,C)|}{|\Aut(B)|}.
\end{align*}
\end{proposition}
\begin{proof}
The first equality follows immediately from \Cref{proposition:unique_fact}. We now prove the second identity. The map
\begin{align*}
    \Psi: E(A,B) \times M(B,C) &\to \{t \in \Hom(A, C)|\;\exists e \in E(A,B),\exists m \in M(B,C), t = m \circ e\}\\
   (e,m) &\mapsto m \circ e.
\end{align*}
is surjective by definition. Furthermore, using the first equality, we clearly have $|\Psi^{-1}(f)| = |\Aut(B)|$.
\end{proof}

\subsection{\DiGraphs: Finite Sets Equipped with Binary Relations}
\label{subsection:fin_set_bin}

We now consider the category of finite directed graphs, which we denote with \DiGraphs.

\begin{definition}
\label{definition:directed_graph}
A directed graph $\graph$ is a pair
\begin{align*}
   \graph:=(A,R_{A})
\end{align*}
where $A$ is a finite set and  $R_{A}\subseteq A \times A$ is a relation.
\end{definition}
We will also write $ \graph = (V(\graph),E(\graph))$ and refer to $V(\graph)$ as the vertex set of $\graph$ and to $E(\graph)$ as the  edge set of $\graph$. Morphisms are maps that preserve relations. We will also refer to a directed graph simply as a \quot{graph}.

\begin{definition}
\label{definition:graph_homomorphism}
Given two graphs $\graph=(A,R_{A}),\,\graph^{\prime}=(B,R_{B})$, a map of sets $f: A \to B$ is a \DEF{graph homomorphism}, written $f: \graph \to \graph^{\prime}$, if
\begin{align*}
    \forall a,a^{\prime} \in A: (a,a^{\prime}) \in R_{A}
\implies (f(a),f(a^{\prime})) \in R_{B}. \end{align*}
\end{definition}
Graph homomorphisms are the morphisms in the category \DiGraphs. It follows from the definition of graph homomorphism that $f$ is an isomorphism if and only if the map $f$ is a bijective graph homomorphism whose inverse is again a graph homomorphism.

We now define induced subgraphs.

\begin{definition}[Induced Subgraph]
\label{definition:induced_subgraph}
Let $\graph:=(V(\graph),E(\graph))$ be an object in \DiGraphs. Let $U \subseteq V(\graph)$. Then
$(U,E(\graph)\cap U \times U)$
is an induced subgraph of $\graph$. 
\end{definition}

The proof of the following proposition is left as an exercise.
\begin{proposition}
\label{proposition:inj_graph_are_mono}
In the category of \DiGraphs, the following holds
\begin{enumerate}
    \item injective graph homomorphisms are monomorphisms,
     \item surjective graph homomorphisms are epimorphisms.
\end{enumerate}
\end{proposition}

\subsubsection*{Finite coproducts}

In \DiGraphs, finite coproducts always exist. We omit the standard proof of the following result.
\begin{proposition}
\label{proposition:coproduct_graphs}
Let $\graph_{1},\graph_{2}\in \DiGraphs$. Then the graph
\begin{align*}
   \graph_{1} \sqcup \graph_{2}&:=(V(\graph_{1} \sqcup \graph_{2}),E(\graph_{1}) \sqcup E(\graph_{2})) 
\end{align*}
forms a coproduct together with the maps
\begin{align*}
    i_{1}:\graph_{1}  &\to \graph_{1} \sqcup \graph_{2},\;\;
   v \mapsto (v,1)\\
     i_{2}:\graph_{2}  &\to \graph_{1} \sqcup \graph_{2},\;\;
       v \mapsto (v,2).
\end{align*}
\end{proposition}

\subsubsection*{Equalizers}

In \DiGraphs, equalizers always exist.

\begin{proposition}
\label{proposition:di_graph_eq_exists}
Let $f,g:(A,R_{A})\to (B,R_{B})$ be a parallel pair in \DiGraphs. Then the cone given by 
\begin{align*}
    \iota: (E,R_{A} \cap (E \times E))\to (A,R_{A}),
\end{align*}
where $E:=\{a\in A|\;f(a) = g(a)\}$, is an equalizer.
\end{proposition}
\begin{proof}
Let $h:(D,R_{D})\to (A,R_{A})$ be another cone. We have $h(D)\subseteq E$ and $h(R_{D})\subseteq R_{A} \cap (h(R_{D}) \times h(R_{D}))\subseteq R_{A} \cap (E \times E)$. Therefore, we can corestrict $h$ to $\hat{h}:(D,R_{D})\to(E,R_{A} \cap (E \times E))$ such that $h = \iota \circ \hat{h}$. Since $\iota$ is injective, it is a monomorphism. It follows that $\hat{h}$ is unique and we are done.
\end{proof}

\subsubsection*{Coequalizers}

In \DiGraphs, coequalizers always exist.

\begin{proposition}
\label{proposition:coeq_digraphs}
Let $\graph_{1},\graph_{2}$ be objects in \DiGraphs.
Consider a parallel pair $f,g:\graph_{1} \to \graph_{2}$. First, we define
\begin{align*}
 \langle R \rangle:=\bigcap_{\substack{\\\\\{(f(v),\,g(v))|\;v\in V(\graph_{1})\}\subseteq \;R \;\subseteq V(\graph_{2})\times V(\graph_{2})\\\\  R\;\textup{is an equivalence relation}}} \hspace{-2cm}R
\end{align*} 
Now define the directed graph $\graphh$, as follows. Let the vertex set
\begin{align*}
    V(\graphh) :=\{[v]_{\sim \langle R \rangle}|\;v \in V(\graph_{2})\}
\end{align*}
be the partition induced by $\langle R \rangle$. The edge set is defined as
\begin{align*}
    E(\graphh) :=\{(h,h^{\prime}) \in  V(\graphh)\times V(\graphh)|\exists v \in h \land \exists v^{\prime} \in h^{\prime}:(v,v^{\prime})\in E(\graph_{2})\}.
\end{align*}
Then, $\graphh$, together with 
\begin{align*}
\pi:\graph_{2} &\to \graphh\\
v &\mapsto [v]_{\sim\langle R \rangle}
\end{align*}
is a coequalizer.
\end{proposition}
\begin{proof}
 Obviously, $\pi \in \Hom(\graph_{2},\graphh)$.
Furthermore, $\pi$ is a cocone, i.e.
\begin{align*}
\forall v \in V(\graph):\;\pi(f(v)) = \pi(g(v))
\end{align*}
since $(f(v),g(v))\in \langle R\rangle$. Given another cocone, i.e. a directed graph $\graph_{3}$ and a map $\pi^{\prime}\in \Hom(\graph_{2},\graph_3)$ such that $\pi^{\prime}\circ f = \pi^{\prime}\circ g$, we define $u:\graphh \to \graph_{3}$ as
\begin{align*}
 u([v]_{\sim\langle R \rangle}):=\pi^{\prime}([v]_{\sim\langle R \rangle}).
\end{align*}
As a set map, it is well-defined since the partition
\begin{align*}
    \{(\pi^{\prime})^{-1}(v)| \in \pi^{\prime}(V(\graph_{2}))\}
\end{align*}
is coarser than $V(\graphh)$. To see that it is a graph homomorphism, we argue as follows. Let $(h,h^{\prime}) \in E(\graphh)$. We know that there exists $v\in h$ and $v^{\prime}\in h^{\prime}$ such that $(v,v^{\prime})\in E(\graph_{2})$.
Therefore we have $(\pi^{\prime}(v),\pi^{\prime}(v^{\prime}))=(\pi^{\prime}([v]_{\sim\langle R \rangle}),\pi^{\prime}([v^{\prime}]_{\sim\langle R \rangle})) = (\pi^{\prime}(h),\pi^{\prime}(h^{\prime}))\in E(\graphh_{3})$.
Therefore $u \in \Hom(\graphh,\graph_{3})$.

Since $\pi$ is an epimorphism, $u$ is the unique map satisfying this condition.

\begin{figure}[pbth]
\centering
\begin{tikzcd}
\grapho_{1}
\arrow[r, swap,"f", shift right=0.75ex] 
\arrow[r, "g", shift right=-0.75ex] 
&\grapho_{2}\arrow[r, "\pi"] \arrow[d, "\pi^{\prime}"] 
& \graphh\arrow[ld, "u"',swap]\\
& \grapho_{3}&.
\end{tikzcd}
\end{figure}

\end{proof}

\subsubsection*{Pushouts}

From the existence of all finite coproducts and coequalizers, we can deduce the existence of all pushouts, see \Cref{remark:prod_plus_eq_pull}.

\begin{proposition}
\label{proposition:pushouts_digraphs}
Let $f: \graph \to \graphh_{1}$ and $g: \graph \to \graphh_{2}$ be graph homomorphisms and let $i_{1}: \graphh_{1} \to \graphh_{1}\sqcup \graphh_{2}$ and $i_{2}: \graphh_{2} \to \graphh_{1}\sqcup \graphh_{2}$ be the canonical injections. Then the graph
\begin{align*}
\graphh_{\partition}:=(V(\graphh_{\partition}),E(\graphh_{\partition})),
\end{align*}
where
$$V(\graphh_{\partition}):=\{[v]_{\sim_{R}}|v\in V(\graphh_{1})\sqcup V(\graphh_{2})\}$$
is the partition induced by

\begin{align*}
 \hspace{-3cm}\langle R \rangle:= \bigcap_{\substack{\\\\ \{(i_{1}(f(v)),\,i_{2}(g(v)))|\;v\in V(\graph)\}\subseteq R\\\\ R\; \subseteq\; V(\graphh_{1})\sqcup V(\graphh_{2})\times V(\graphh_{1})\sqcup V(\graphh_{2})\;\textup{is an equivalence relation}}}\hspace{-3cm}R
\end{align*} 
and 
\begin{align*}
    E(\graphh_{\partition}) :=\{(h,h^{\prime}) \in  V(\graphh_{\partition})\times V(\graphh_{\partition})|\exists v \in h \land \exists v^{\prime} \in h^{\prime}:(v,v^{\prime})\in E(\graphh_{1})\sqcup E(\graphh_{2}) \},
\end{align*}
together with $j_{1}:=\pi_{\partition}\circ i_{1}$ and $j_{2}:=\pi_{\partition}\circ i_{2}$, with $\pi_{\partition}(v)=[v]_{\sim \langle R \rangle}$, is a pushout.
\end{proposition}

The following proposition together with \Cref{proposition:inj_graph_are_mono} characterizes monomorphisms and epimorphisms in \Digraphs\;as the injective and the surjective maps, respectively.

\begin{proposition}
\label{proposition:mono_epi_inj_surj}
In the category of \DiGraphs, the following holds
\begin{enumerate}
    \item monomorphisms are injective graph homomorphisms,
     \item epimorphisms are surjective graph homomorphisms.
\end{enumerate}
\end{proposition}
\begin{proof}
We show 2.
We use contrapositive: if $f$ is not subjective, then it is not an epimorphism. Let $f:\graph \to \graphh$ be a non-surjective graph homomorphism. Denote with $\graphh_{\partition}$, $j_{1}$ and $j_{2}$ the pushout of $f$ with itself, see \Cref{proposition:pushouts_digraphs}. By construction, $f$ is a cone, i.e., $f\circ j_{1}=f\circ j_{2}$. Recall that
$$V(\graphh_{\partition})=\{(b,1),(b,2)|\;b \in f(A)\} \cup \bigcup_{b \in B \setminus f(A)}\{(b,1)\}\cup \bigcup_{b \in B \setminus f(A)}\{(b,2)\}.$$
Since $f$ is non-surjective, clearly $j_{1}\neq j_{2}$ and we are done.

\end{proof}

\subsubsection*{Regular monomorphisms}

We will now show that regular monomorphisms correspond to graph embeddings. We first define graph embeddings.

\begin{definition}
\label{definition:graph_embedding}
Let $f:\graph \to \graphh$ be an arrow in \DiGraphs. We say that $f$ is a \DEF{graph embedding} if
\begin{align*}
    \hat{f}:\graph &\to (f(V(\graph)),E(\graphh)\cap f(V(\graph))\times f(V(\graph)))\\
    v &\mapsto f(v)
\end{align*}
is an isomorphism.
\end{definition}
As we show in the following lemma, the composition of graph embeddings is again a graph embedding.

\begin{lemma}
\label{lemma:comp_graph_emb}
Let $f:\graph_{1}\to \graph_{2}$ and $g:\graph_{2}\to \graph_{3}$ be both graph embeddings. Then $g \circ f$ is a graph embedding.
\end{lemma}
\begin{proof}

Let $h:= g \circ f$. Then
\begin{align*}
    \hat{h}:\graph_{1} &\to (h(V(\graph_{1})),E(\graph_{3})\cap h(V(\graph_{1}))\times h(V(\graph_{1})))\\
    v &\mapsto h(v)
\end{align*}
is an isomorphism. Indeed, the mapping
\begin{align*}
    \hat{h}^{-1}:(h(V(\graph_{1})),E(\graph_{3})\cap h(V(\graph_{1}))\times h(V(\graph_{1}))) &\to \graph_{1}\\
    v &\mapsto h^{-1}(v)
\end{align*}
is the compositional inverse of $\hat{h}$ and a graph homomorphism because
\begin{align*}
&(v,v^{\prime}) \in E(\graph_{3})\cap h(V(\graph_{1}))\times h(V(\graph_{1})) \\&\implies (g^{-1}(h(v)),g^{-1}(h(v^{\prime})))  \in E(\graph_{2})\cap f(V(\graph_{1}))\times f(V(\graph_{1}))\\&
\implies (v,v^{\prime})  \in E(\graph_{1}).
\end{align*}
The first implication holds because $g$ is an embedding, and the second one because $f$ is an embedding.
\end{proof}

\begin{proposition}
\label{proposition:digrap_regmono}
In \DiGraphs, regular monomorphisms are precisely the graph embeddings.
\end{proposition}
\begin{proof}
Let $f:(A,R_{A})\to (B,R_{B})$ be a limit for the parallel pair $h,g:(B,R_{B})\to (D,R_{D})$. We know from \Cref{proposition:di_graph_eq_exists}, that the inclusion map $\iota:(E,R_{B}\cap (E \times E))\to (B,R_{B})$ with $E:=\{b\in B|h(b)=g(b)\}$ is an equalizer as well.
\begin{figure}[pbth]
\begin{center}
\begin{tikzcd}
(E,R_{B}\cap (E \times E)) \arrow[r,"\iota",hook]
&(B,R_{B}) \arrow[r, "h", shift left=0.75ex] \arrow[r,"g"',shift right=0.75ex]
& (D,R_{D}) \\
(A, R_{A})\arrow[u,"u",dashed]  \arrow[ru,"f"'] &&. 
\end{tikzcd}
\end{center}
\end{figure}
Therefore $f=\iota \circ u$ holds, where $u$ is an isomorphism and, in particular, a graph embedding. Since $\iota$ is also a graph embedding, we can use \Cref{lemma:comp_graph_emb} to conclude that $f$ is also a graph embedding and we are done.

Now let $f:\graph \to \graphh$ be a graph embedding. Denote with $\graphh_{\partition}$, $j_{1}$ and $j_{2}$ the pushout of $f$ with itself, see \Cref{proposition:pushouts_digraphs}. By construction, $f$ is a cone. Since
$$V(\graphh_{\partition})=\{(b,1),(b,2)|\;b \in f(A)\} \cup \bigcup_{b \in B \setminus f(A)}\{(b,1)\}\cup \bigcup_{b \in B \setminus f(A)}\{(b,2)\},$$
the graph
$$\graphh_{f(V(\graph))}:=\left(f(V(\graph)),E(\graphh)\cap f(V(\graph)) \times f(V(\graph))\right)$$
together with the inclusion $\iota:\graphh_{f(V(\graph))}\to \graphh$ is an equalizer. 

Therefore we have $f=\iota \circ \hat{f}$. By definition of graph embedding, $\hat{f}$ is an isomorphism; therefore, $f$ is also an equalizer, and we are done.

\begin{figure}[pbth]
\begin{center}
\begin{tikzcd}
\grapho \arrow[r, "f"] \arrow[dd, "\hat{f}",swap, dashed] & \graphh \arrow[r, "j_{1}", shift left=0.75ex] \arrow[r, swap,"j_{2}", shift right=0.75ex] & \graphh_{\partition} \\&&\\
\graphh_{f(V(\graph))}  \arrow[uur, "\iota",hook,swap]&&
\end{tikzcd}
\end{center}
\end{figure}
\end{proof}

\begin{proposition}
\label{proposition:digraphs_facto}
\DiGraphs\;admits an (\Epi,\RegMono)-factorization 
\end{proposition}
\begin{proof}
Let $f:\graph \to \graphh$ be a graph homomorphism. Clearly, we have $f = \iota \circ \hat{f}$ where $\hat{f}:\graph\to (f(V(\graph)),E(\graphh)\cap f(V(\graph))\times f(V(\graph)))$ is the corestriction of $f$ and $\iota: (f(V(\graph)),E(\graphh)\cap f(V(\graph))\times f(V(\graph)) \to \graphh$ is the inclusion map. Obviously, $\hat{f}$ is a surjection and $\iota$ an embedding.
\end{proof}
  
\subsection{\Poset}
\label{subsection:Poset}

We now characterize regular monomorphisms and epimorphisms in \Poset, the category of strict posets. We also show that \Poset\;admits an (\Epi,\RegMono)-factorization.

\Poset\;is a \textit{full} subcategory of \DiGraphs. This means that the set of morphisms between any two objects in \Poset\;is given by all the homomorphisms between these two objects in \DiGraphs. 

We will use that \Poset\;is also a full subcategory of \DAG, the category of directed acyclic graphs. \DAG\;is a full subcategory of \Digraphs\;as well.

\begin{definition}
\label{definition:poset}
Let $A$ be a finite set and $R \subseteq A \times A$. We say that $R$ is a strict poset relation if it is asymmetric and transitive, i.e.
\begin{itemize}
  \item $\forall a,a^{\prime} \in A: (a,a^{\prime}) \in R \implies (a^{\prime},a) \not\in R$,
    \item $\forall a,a^{\prime},a^{\prime\prime}\in A: (a,a^{\prime}) \in R \land (a^{\prime},a^{\prime\prime}) \in R \implies (a,a^{\prime\prime}) \in R$.
\end{itemize}
Observe that asymmetry implies that $R$ is irreflexive, i.e. $$\forall a \in A: (a,a)\not \in R.$$
\end{definition}
\begin{definition}
\label{definition:trans_clos}
Let $A$ be a set, and $R \subset A \times A$. The \DEF{transitive closure} of $R$ is defined as
\begin{align*}
\TransClos(R):=\bigcap_{\substack{S \subset A \times A:\\S\;\text{is transitive and}\; R \subset S}}S.
\end{align*}
\end{definition}

\begin{definition}
\label{definition:trans_red}
Let $A$ be a set and $R \subset A \times A$ be a transitive relation. The \DEF{transitive reduction} of $R$ is defined as
\begin{align*}
\TransRed(R):=\bigcap_{\substack{S \subset A \times A:\\\TransClos(S)=R}}S.
\end{align*}
\end{definition}

\begin{remark}
\label{remark:posets_are_DAG}
If we denote with $\DAG$ the category of directed acyclic graphs, this is also a full subcategory of $\DiGraphs$ of sets equipped with asymmetric relations whose transitive closure is again asymmetric, i.e.
\begin{align*}
(A,P_{A})\in \DAG &\iff (A,\TransClos(P_{A}))\in \Poset.
\end{align*}
\end{remark}

\begin{definition}
\label{definition:Hasse_diagram}
Let $R \subset A \times A$ be a transitive, asymmetric relation, i.e., $(A,R)\in \Poset$. We will refer to $(A,\TransRed(R))$ as its \textbf{Hasse diagram}.
\end{definition}

As for \DiGraphs, the following holds. If we use \Cref{proposition:special_pushout}, the proof is analogous and is omitted.
\begin{proposition}
In \Poset:
 \begin{itemize}
     \item monomorphisms are precisely the injective order preserving maps,
     \item epimorphisms are precisely the surjective order preserving maps.
 \end{itemize}
\end{proposition}
In \Poset, isomorphisms are precisely the invertible graph homomorphisms whose inverses are also graph homomorphisms.

The proof of the following result is immediate.

\begin{theorem}
 For \Poset, finite coproducts always exist.  
\end{theorem}

The following proposition implies that induced subgraphs cannot take us out of the category we are in.
\begin{proposition}
\label{proposition:induced_subgraphs_closed}
 Let $f:(A,R_{A})\to (B,R_{B})$ be a morphism in \Poset. Then 
 \begin{align*}
     (f(A), R_{B}\cap f(A) \times f(A))
 \end{align*}
 is still an object in \Poset.
\end{proposition}
\begin{proof}
Immediate.
\end{proof}

In \Poset \; equalizers always exist.

\begin{proposition}
\label{proposition:poset_exists}
Let $f,g:(A,R_{A})\to (B,R_{B})$ be a parallel pair in \Poset. Then the cone given by 
\begin{align*}
    \iota: (E,R_{A} \cap (E \times E))\to (A,R_{A}),
\end{align*}
where $E:=\{a\in A|\;f(a) = g(a)\}$, is an equalizer.
\end{proposition}
\begin{proof}
Since
\Cref{proposition:induced_subgraphs_closed} holds, the construction for \DiGraphs\;from the proof of \Cref{proposition:di_graph_eq_exists} can be used in the same way.
\end{proof}

In \Poset, one must first verify that the coequalizer is a directed acyclic graph, which might not occur. If the resulting object is in \DAG, then taking its transitive closure will yield an object in \Poset, see \Cref{remark:posets_are_DAG}.
\begin{proposition}
\label{proposition:coeq_poset}
In \Poset\, a coequalizer exists if and only if the coequalizer constructed in \Cref{proposition:coeq_digraphs} is an object in \DAG.
\end{proposition}
\begin{example}
Consider \Cref{figure:coeq_poset}. For the parallel pair $f,g:\{x_{1},x_{2}\}\to \{x_{3},x_{4},x_{5},x_{6}\},$ $h:\{x_{3},x_{4},x_{5},x_{6}\}\to \{\{x_{3}\},\{x_{4},x_{5}\},\{x_{6}\}\}$ is a coequalizer in \Poset.
\begin{figure}[H]
  \centering
  \begin{minipage}{0.32\textwidth}
    \centering
    \scalebox{0.6}{
      \begin{tikzpicture}
        \begin{scope}[every node/.style={circle,thick,draw}]
          \node (A) at (2,4) {$x_{3}$};
          \node  (B) at (2,2) {$x_{4}$};
          \node(C) at (2,0) {$x_{5}$};
          \node  (D) at (2,-2) {$x_{6}$};
          \node  (E) at (-2,2) {$x_{1}$};
          \node (F) at (-2,0) {$x_{2}$};
        \end{scope}
        \draw[->, thick] (A) to node[left] {} (B);
        \draw[->, thick] (C) to node[left] {} (D);
        \draw[->, thick,red] (E) to node[left] {} (C);
        \draw[->, thick,red] (F) to node[left] {} (B);
      \end{tikzpicture}
    }
    \caption*{$f$}
  \end{minipage}
  \hspace{0.01\textwidth}
  \begin{minipage}{0.28\textwidth}
    \centering
    \scalebox{0.6}{
      \begin{tikzpicture}
        \begin{scope}[every node/.style={circle,thick,draw}]
          \node (A) at (2,4) {$x_{3}$};
          \node  (B) at (2,2) {$x_4$};
          \node(C) at (2,0) {$x_{5}$};
          \node  (D) at (2,-2) {$x_{6}$};
          \node  (E) at (-2,2) {$x_{1}$};
          \node (F) at (-2,0) {$x_{2}$};
        \end{scope}
        \draw[->, thick] (A) to node[left] {} (B);
        \draw[->, thick] (C) to node[left] {} (D);
        \draw[->, thick,red] (E) to node[left] {} (B);
        \draw[->, thick,red] (F) to node[left] {} (C);
      \end{tikzpicture}
    }
    \caption*{$g$}
  \end{minipage}
  \hspace{0.01\textwidth}
  \begin{minipage}{0.28\textwidth}
    \centering
    \scalebox{0.6}{
      \begin{tikzpicture}
        \begin{scope}[every node/.style={circle,thick,draw}]
          \node (A) at (-2,4) {$x_{3}$};
          \node  (B) at (-2,2) {$x_{4}$};
          \node(C) at (-2,0) {$x_{5}$};
          \node  (D) at (-2,-2) {$x_{6}$};
          \node  (G) at (2,3) {$\{x_{3}\}$};
          \node  (H) at (2,1) {$\{x_{4},x_{5}\}$};
          \node (I) at (2,-1) {$\{x_{6}\}$};
        \end{scope}
        \draw[->, thick] (A) to node[left] {} (B);
        \draw[->, thick] (C) to node[left] {} (D);
        \draw[->, thick] (G) to node[left] {} (H);
        \draw[->, thick] (H) to node[left] {} (I);
        \draw[->, thick,red] (A) to node[left] {} (G);
        \draw[->, thick,red] (B) to node[left] {} (H);
        \draw[->, thick,red] (C) to node[left] {} (H);
        \draw[->, thick,red] (D) to node[left] {} (I);
        \draw[->, thick,bend left=60] (G) to node[left] {} (I);
      \end{tikzpicture}
    }
    \caption*{$h$}
  \end{minipage}
  \caption{Coequalizer in \Poset}
  \label{figure:coeq_poset}
\end{figure}
\end{example}

The following proposition is used in the proof of \Cref{proposition:poset_reg_mono_order_embeddings}.

\begin{proposition}
\label{proposition:special_pushout}
 $f:\graph \to \graphh$ be an arrow in $\Poset$. Then the pushout of $f$ with itself as constructed in \Cref{proposition:pushouts_digraphs} is an object in $\Poset$.
\end{proposition}
\begin{proof}
Let $f:\graph \to \graphh$ be an arrow. Then the pushout of $f$ with itself is given by the vertex set
$$V(\graphh_{\partition}) =\{\{(v,1),(v,2)\}|\;v \in f(V(\graph))\} \cup \bigcup_{v \in V(\graphh) \setminus f(V(\graph))}\{(v,1)\}\cup \bigcup_{v \in V(\graphh) \setminus f(V(\graph))}\{(v,2)\},$$
and the edge set
\begin{align*}
&\forall u,v \in f(V(\graph)):\;(\{(u,1),(u,2)\},\{(v,1),(v,2)\}) \in E(\graphh_{\partition}) \iff (u,v)\in E(\graphh),\\
&\forall u \in f(V(\graph)),v \in V(\graphh)\setminus f(V(\graph)):\;(\{(u,1),(u,2)\},\{(v,1)\}) \in E(\graphh_{\partition}) \iff (u,v)\in E(\graphh),\\
&\forall u \in f(V(\graph)),v \in V(\graphh)\setminus f(V(\graph)):\;(\{(u,1),(u,2)\},\{(v,2)\}) \in E(\graphh_{\partition}) \iff (u,v)\in E(\graphh),\\
&\forall u,v \in V(\graphh)\setminus f(V(\graph)):\;(\{(u,1)\},\{(v,1)\}) \in E(\graphh_{\partition}) \iff (u,v)\in E(\graphh),\\
&\forall u,v \in V(\graphh)\setminus f(V(\graph)):\;(\{(u,2)\},\{(v,2)\}) \in E(\graphh_{\partition}) \iff (u,v)\in E(\graphh).
\end{align*}
Since $f$ is arrow in $\Poset$, then $f$ is an arrow in $\DAG$. We now show that $\graphh_{\partition}$ is an object in \DAG.
Let $z_{1} \sim \cdots \sim z_{n}$ be a chain in $E(\graphh_{\partition})$. This means that either $z_{i}=\{(u_{i},1)\}$, or $z_{i}=\{(u_{i},2)\}$,  or $z_{i}=\{(u_{i},1),(u_{i},2)\}$. Therefore, we have a chain in $V(\graphh)$ given by $u_{1}\sim u_{2} \sim \cdots \sim u_{n}$. Since $\graphh$ is an object in \DAG, $u_{1}\neq u_{n}$, and therefore $z_{1}\neq z_{n}$. Therefore, $\graphh_{\partition}$ is an object in \DAG. This means that $(V(\graphh_{\partition}),\TransClos(E(\graphh_{\partition})))$ is an object in \Poset, see \Cref{remark:posets_are_DAG}.
\end{proof}

Regular monomorphisms have an analogous characterization to \DiGraphs.
\begin{proposition}
\label{proposition:poset_reg_mono_order_embeddings}
    In \Poset\;regular monomorphisms are precisely the order embeddings.
\end{proposition}
\begin{proof}
We can apply \Cref{proposition:special_pushout}, and then the proof is analogous to that of \Cref{proposition:digrap_regmono}.
\end{proof}
The following proposition has a proof analogous to the one of \Cref{proposition:digraphs_facto}.
\begin{proposition}
 The category \Poset\;admits an (\Epi,\RegMono)-factorization.   
\end{proposition}

\subsection{\DPoset}
\label{subsection:d_strict_poset}
We now deal with the category of double posets, \DPoset. The objects are strict double posets, i.e., triples $\smalldp = \triplesmalldp$ where $A$ is a (finite) set, and $P_{A}$ and $ Q_{A}$ are strict posets on $A$. The morphisms are maps that simultaneously respect both the first and the second relation.
A morphism of double posets $\smalldp =\triplesmalldp$ and $\smalldpprime =\triplesmalldpprime$ is a map $f:A \to B$ such that 
\begin{align*}
    &\forall a_{1},a_{2} \in A: a_{1} <_{P_{A}} a_{2} \implies  f(a_{1}) <_{P_{B}} f(a_{2})\\
      &\forall a_{1},a_{2} \in A: a_{1} <_{Q_{A}} a_{2} \implies  f(a_{1}) <_{Q_{B}} f(a_{2})\\
\end{align*}

We state the following results without proof, as they are essentially analogous to those in Section\;\ref{subsection:Poset}. The same arguments apply, now carried out simultaneously for both relations.

\begin{proposition}
In \DPoset:
 \begin{itemize}
     \item monomorphisms are precisely the injective double order preserving maps,
     \item epimorphisms are precisely the surjective double order preserving maps.
 \end{itemize}
\end{proposition}
Observe that, again, in \DPoset, isomorphisms are precisely the invertible graph homomorphisms whose inverses are also graph homomorphisms. A double order embedding is the straightforward extension of the concept of graph embedding to both relations.

\begin{proposition}
\label{proposition:double_poset_reg_mono_order_embeddings}
    In \DPoset\;regular monomorphisms are precisely the double order embeddings.
\end{proposition}
\begin{proposition}
 The category \DPoset\;admits an (\Epi,\RegMono)-factorization.   
\end{proposition}
We can finally state the main theorem.
\begin{theorem}
\label{theorem:hom_reg_mono}
Let $\smalldp,\smalldpprime,\largedp$ be double posets. Then 
\begin{align*}
    |\Hom(\smalldp,\largedp)| =\sum_{\smalldpprime\in \DP} |\{f \in \Hom(\smalldp,\largedp)|\largedp_{f(V(\smalldp))}\cong \smalldpprime\}|&=\sum_{\smalldpprime \in \DP} \frac{|\Epi(\smalldp,\smalldpprime)|}{|\Aut(\smalldpprime)|}|\RegMono(\smalldpprime,\largedp)|
\end{align*}
where $\DP$ is the set of equivalence classes of double posets and $\largedp_{f(V(\smalldp))}$ is the double poset induced by $f(V(\smalldp))$.
\end{theorem}
\begin{proof}
 The first equality follows from 
 \begin{align*}
     \Hom(\smalldp,\largedp) =\bigsqcup_{\smalldpprime\in \DP} \{f \in \Hom(\smalldp,\largedp)|\largedp_{f(V(\smalldp))}\cong \smalldpprime\}
 \end{align*}
 where the union is taken over disjoint sets. Furthermore, the sets $\{f \in \Hom(\smalldp,\largedp)|\largedp_{f(V(\smalldp))}\cong \smalldpprime\}$ are non-empty only for finitely many elements in $\DP$. The second equality follows from \Cref{proposition:comb_id}.
\end{proof}

\end{document}

%% file: packagesandmore.tex
\usepackage{amsmath}      % Enhanced math 
\usepackage{amssymb}    
\usepackage{colortbl}% Additional math symbols
\makeatother
\usepackage{algorithm}
\usepackage[noend]{algpseudocode}
\usepackage{setspace} % For controlling line spacing
\usepackage{wrapfig}
\usepackage{amsfonts}
\usepackage{bm}           % Bold math symbols
\usepackage{mathtools} 
\usepackage{caption}
\usepackage{subcaption}
\usepackage{float}
\usepackage{placeins}

\usepackage{tikz-cd}
\usepackage{xcolor}
\usepackage{sectsty}
\usepackage{csquotes}
\usepackage{placeins}
\newcommand\edgelabel{\mathfrak e}

\newcommand\MN{\mathtt{MN}}
\newcommand\ME{\mathtt{ME}}
\newcommand\MAX{\mathtt{MAX}}

\usepackage{fancyvrb}
\usepackage{xspace}
\usepackage{subcaption}
\usepackage{float}
\usepackage{imakeidx}
\usepackage{diagbox}
\usepackage{listings}
\newcommand\gray[1]{{\color{gray}#1}}

\newcommand\DEF[1]{\textbf{\textup{#1}}}
\usepackage{algorithm}
\newcommand{\quot}[1]{``#1''}
\usepackage{algpseudocode}
\usepackage{algorithm,algpseudocode}
\usepackage{fontawesome}
\usepackage{amsthm}
\usepackage{tikz}
\usetikzlibrary{decorations.pathreplacing,decorations.markings}
\usepackage{zi4} % Add to preamble (before \begin{document})
\newcommand{\algfont}{\fontfamily{zi4}\selectfont}
\usepackage{xstring}
\newcommand{\perm}[1]{%
  \textcolor{black}{%
  \mathtt{%
  [%
    \StrLen{#1}[\stringLength]% Get the length of the input string
    \foreach \n in {1,...,\stringLength}{% Loop through each character
      \StrChar{#1}{\n}% Get the n-th character of the string
      \ifnum\n<\stringLength
      \,% Insert a comma and space if not the last character
      \fi
    }%
  ]%
  }%
}%
}

\usepackage{mathrsfs}
\usepackage[style=alphabetic,giveninits=true,backend=biber]{biblatex}
\DeclareCiteCommand{\cite}
  {\usebibmacro{prenote}}
  {\usebibmacro{citeindex}%
   \printtext[bibhyperref]{\usebibmacro{cite}}}
  {\multicitedelim}
  {\usebibmacro{postnote}}
\usepackage{microtype}

\usepackage{setspace}

\usepackage{hyperref}
\usepackage{cleveref}
\renewcommand\cref[1]{\FAIL} % please don't use
\usepackage{float}
\usepackage{geometry}
\usepackage{forest}
\allowdisplaybreaks
\newcommand\id{\operatorname{id}}
\newcommand\proj{\operatorname{proj}}
\newcommand{\e}{\textup{\textbf{e}}}
\newcommand{\Q}{\mathbb{Q}}
\newcommand{\NN}{\mathbb{N}}

\newcommand{\NE}{\ensuremath{\mathsf{NE}}}
\newcommand{\NW}{\ensuremath{\mathsf{NW}}}
\newcommand{\SE}{\ensuremath{\mathsf{SE}}}
\newcommand{\SW}{\ensuremath{\mathsf{SW}}}

\newcommand\N{\ensuremath{\mathsf{N}}\xspace}
\renewcommand\S{\ensuremath{\mathsf{S}}\xspace}

\newcommand\West{\ensuremath{\mathsf{West}}\xspace}
\newcommand\South{\ensuremath{\mathsf{South}}\xspace}

\newcommand{\TypeA}{\ensuremath{\mathsf{Type\_A}}\xspace}
\newcommand{\TypeB}{\ensuremath{\mathsf{Type\_B}}\xspace}
\newcommand{\TypeAnotB}{\ensuremath{\mathsf{Type\_A\_not\_B}}\xspace}
\newcommand{\TypeBnotA}{\ensuremath{\mathsf{Type\_B\_not\_A}}\xspace}
\newcommand{\TypenotAnotB}{\ensuremath{\mathsf{Type\_not\_A\_not\_B}}\xspace}

\renewcommand\O{{\mathcal O}}

\newcommand{\vertexset}{\textup{V}}
\newcommand{\edgeset}{\textup{E}}
\newcommand{\permutations}{\ensuremath{\mathfrak{S}}}
\newcommand{\CT}{\ensuremath{\mathsf{ct}}}
\newcommand{\Tsn}{\ensuremath{\mathsf{T_{SN}}}}

\newcommand{\CornerTrees}{\ensuremath{\mathsf{CornerTrees}}}

\newcommand{\CornerTreesfivethree}{\ensuremath{\mathsf{CornerTrees_{\;5/3}}}}

\newcommand{\SNPolyTrees}{\ensuremath{\mathsf{SNpoly}}}
\newcommand{\doubleposets}{\ensuremath{\DP}}
\newcommand{\Profile}{\ensuremath{\mathsf{Profile}}}
\newcommand{\treedoubleposets}{\ensuremath{\mathsf{TreeDP}}}
\newcommand{\twindoubleposets}{\ensuremath{\mathsf{TwinDP}}}
\newcommand{\twintreedoubleposets}{\ensuremath{\mathsf{TwinTreeDP}}}
\newcommand{\PC}{\ensuremath{\mathsf{PC}}}
\newcommand{\Poset}{\ensuremath{\mathsf{PoSet}}}
\newcommand{\DPoset}{\ensuremath{\mathsf{DPoSet}}}
\newcommand{\FinSet}{\ensuremath{\mathsf{FinSet}}}

\newcommand\arr[1]{\xrightarrow{\ #1\ }}
\newcommand{\Cat}{\ensuremath{\mathcal{C}}}

\definecolor{todocolor}{HTML}{fa73ff}
\usepackage[linecolor=white,backgroundcolor=white,bordercolor=white,textsize=tiny,textcolor=todocolor]{todonotes}
\newcommand{\Rev}{\textbf{\textup{Rev}}}
\newcommand{\Comp}{\textbf{\textup{Comp}}}
\newcommand{\Swap}{\textbf{\textup{Swap}}}

\newcommand{\Hom}{\ensuremath{\mathsf{Hom}}}
\renewcommand\hom{\ensuremath{\mathsf{hom}}}
\newcommand{\Mono}{\ensuremath{\mathsf{Mono}}}

\newcommand{\Epi}{\ensuremath{\mathsf{Epi}}}
\newcommand{\RegMono}{\ensuremath{\mathsf{RegMono}}}
\newcommand{\regmono}{\ensuremath{\mathsf{regmono}}}

\newcommand{\Aut}{\ensuremath{\mathsf{Aut}}}

\newcommand{\Digraphs}{\ensuremath{\mathsf{Digraphs}}}

\newcommand{\DAG}{\ensuremath{\mathsf{DAG}}}

\newcommand\partition{\mathsf{P}}
\usepackage[normalem]{ulem}
\newcommand\graph{\textbf{\textup{\fontfamily{qhv}\selectfont\scalebox{0.8}{g}}}}
\newcommand\grapho{\textbf{\textup{\fontfamily{qhv}\selectfont\scalebox{0.8}{g}}}}

\newcommand\graphh{\textbf{\textup{\fontfamily{qhv}\selectfont\scalebox{0.8}{h}}}}
\newcommand\smalldpprime{{\mathtt{d}^{\prime}}}
\newcommand\smalldp{{\mathtt{d}}}
\newcommand\largedp{{\mathtt{D}}}
\newcommand\smallct{{\mathtt{tt}}}
\newcommand\arboNE{\ensuremath{\mathsf{t_{\,5/3}}}}
\newcommand{\DiGraphs}{\ensuremath{\mathsf{DiGraphs}}}

\newcommand\arboNEone{\ensuremath{\mathsf{t^1_{\,5/3}}}}
\newcommand\arboNEtwo{\ensuremath{\mathsf{t^2_{\,5/3}}}}
\newcommand\arboNEthree{\ensuremath{\mathsf{t^3_{\,5/3}}}}

\newcommand\arboNEprime{\ensuremath{\mathsf{t^{\prime}_{\,5/3}}}}
\newcommand\ArboNE{\ensuremath{\mathsf{Tree_{\,5/3}}}}

\newcommand\smalltwin{{\mathtt{twin}}}
\newcommand\smalltree{{\mathtt{t}}}

\usepackage{subfiles}

\newcommand\triplesmalldp{(A,P_{A},Q_{A})}

\newcommand\triplesmalldpprime{(B,P_{B},Q_{B})}

\newcommand\intervalpartI{\mathsf{I}}
\newcommand\intervalpartJ{\mathsf{J}}

\newtheorem{theorem}{Theorem}[section]
\newtheorem{lemma}[theorem]{Lemma}
\newtheorem{proposition}[theorem]{Proposition}

\newtheorem{definition}[theorem]{Definition}

\newtheorem{example}[theorem]{Example}
\newtheorem{remark}[theorem]{Remark}
\newtheorem*{theorem*}{Theorem}

\newcommand\CTtoSNpoly{\Psi_{\mathsf{SN} \leftarrow\mathsf{CT}}}
\newcommand\SNpolytoSDP{\Psi_{\mathsf{TTDP}\leftarrow \mathsf{SN}}}
\newcommand\SNpolytoCT{\Psi_{\mathsf{CT} \leftarrow\mathsf{SN}}}
\newcommand\TTDPtoCT{\Psi_{\mathsf{CT} \leftarrow\mathsf{TTDP}}}
\newcommand\SDPtoSNpoly{\Psi_{\mathsf{SN}\leftarrow \mathsf{TTDP}}}
\newcommand\CTtoDP{\Psi_{\mathsf{TTDP}\leftarrow \mathsf{CT}}}
\newcommand\SntoDP{\Psi_{\mathcal{S}\leftarrow\permutations}}
\newcommand\DPtoSn{\Psi_{\permutations\leftarrow\mathcal{S}}}
\newcommand\DP{\mathsf{DP}}
\newcommand{\TransClos}{\ensuremath{\mathsf{Tr}}}
\newcommand{\TransRed}{\ensuremath{\mathsf{TrRd}}}

\tikzset{mid arrow/.style={
    postaction={decorate,decoration={
        markings,
        mark=at position .5 with {\arrow[line width=2mm, scale=5]{stealth}}
    }}
}}